\documentclass[ssy]{imsart}
\usepackage{amsthm,amsmath,amsfonts}
\usepackage{color,graphicx}
\usepackage[colorlinks,linkcolor=blue,citecolor=cyan,anchorcolor=blue]{hyperref}
\usepackage[numbers,sort]{natbib}
\usepackage{enumerate}
\usepackage{wrapfig}
\usepackage{epstopdf}
\newcommand{\R}{\mathbb{R}}
\newcommand{\D}{\mathbb{D}}

\newcommand{\Z}{\mathbb{Z}}
\newcommand{\E}{\mathbb{E}}
\newcommand{\HH}{\mathcal{H}}
\newcommand{\lipone}{\text{\rm Lip(1)}}
\newcommand{\Prob}{\mathbb{P}}

\providecommand{\abs}[1]{\left\lvert#1\right\rvert}
\providecommand{\norm}[1]{\lVert#1\rVert}
\usepackage[normalem]{ulem}
\newcommand{\st}[1]{}
\newcommand{\blue}[1]{{#1}}

\newtheorem{lemma}{Lemma}
\newtheorem{theorem}{Theorem}

\theoremstyle{definition}
\numberwithin{equation}{section}

\newtheorem{remark}{Remark}
\begin{document}

\begin{frontmatter}

\title{Stein's method for steady-state diffusion approximations: an introduction through the Erlang-A and Erlang-C models}
\runtitle{Stein's method for steady-state diffusion approximations}


\author{\fnms{Anton} \snm{Braverman}\ead[label=e1]{ab2329@cornell.edu}}
\address{School of Operations Research \\
\quad and Information Engineering  \\
Cornell University \\
Ithaca, New York 14853, USA \\
\printead{e1}}
\author{\fnms{J.G.} \snm{Dai}\ead[label=e2]{jd694@cornell.edu}}
\address{School of Operations Research \\
\quad and Information Engineering  \\
Cornell University \\
Ithaca, New York 14853, USA \\
\printead{e2}}
\and
\author{\fnms{Jiekun} \snm{Feng}\ead[label=e3]{jf646@cornell.edu}}
\address{Department of Statistical  Science\\
Cornell University \\
Ithaca, New York 14853, USA \\
\printead{e3}}
\affiliation{Cornell University}

\runauthor{Braverman, Dai, Feng}

\begin{abstract}
  This paper provides an introduction to the Stein method framework in
  the context of steady-state diffusion approximations. The framework
  consists of three components: the Poisson equation and gradient
  bounds, generator coupling, and moment bounds. Working in the
  setting of the Erlang-A and Erlang-C models, we prove that both
  Wasserstein and Kolmogorov distances between the stationary
  distribution of a normalized customer count process, and that of an
  appropriately defined diffusion process decrease at a rate of
  $1/\sqrt{R}$, where $R$ is the offered load. Futhermore, these error
  bounds are \emph{universal}, valid in any load condition from
  lightly loaded to heavily loaded.
\end{abstract}



\end{frontmatter}

\section{Introduction}
In \cite{BravDai2017}, the authors developed a framework based on
Stein's method \cite{Stei1972,Stei1986} to prove the rates of
convergence for steady-state diffusion approximations. Using their framework, they proved convergence rates for the steady-state approximation of the $M/Ph/n+M$ system, the many server queue with customer abandonment and phase-type service times, in the Halfin--Whitt regime \cite{HalfWhit1981}. The framework in \cite{BravDai2017} is modular and has four components: the Poisson equation and gradient
bounds, generator coupling, moment bounds, and state space
collapse (SSC). The purpose of this paper is to provide an accessible introduction to the Stein framework, focusing on two simple and yet fundamental systems. They are the $M/M/n+M$ system,
known as the Erlang-A system, and $M/M/n$ system, known as the
Erlang-C system.  The accessibility is due to the fact that both systems can be represented by a one-dimensional continuous time Markov chain (CTMC). In addition, by focusing on these two systems we are able to present some sharp results that serve as benchmarks that future research should aspire to meet.

Stein's method is a powerful method \blue{used for studying approximations of probability distributions, and is best known for its ability to establish convergence rates. It} has been widely used
in probability, statistics, and their wide range of applications such
as bioinformatics; see, for example, the survey papers \cite{Ross2011,
  Chat2014}, the recent book \cite{ChenGoldShao2011} and the
references within. \blue{Applications of Stein's method always involve some unknown distribution to be approximated, and an approximating distribution. For instance, the first appearance of the method  in \cite{Stei1972} involved the sum of identically distributed \blue{dependent} random variables as the unknown, and the normal as approximating distribution. Other approximating distributions include the Poisson \cite{Chen1975}, binomial \cite{Ehm1991}, and multinomial \cite{Loh1992} distributions, just to name a few. For each approximating  distribution, one needs to establish separately gradient bounds, also known as Stein factors \cite{Wein2000, BarbXia2006}, like those in Lemma~\ref{lem:gradboundsCW} in Section~\ref{sec:roadmap}. In this paper,  the approximating  distribution is the stationary distribution of a diffusion process, and the unknown is the stationary distribution of the CTMC introduced in (\ref{eq:CTMCunscaleddef}) below. }

Both Erlang-A and Erlang-C systems have $n$ homogeneous servers that serve customers in a first-come-first-serve manner. Customers arrive according to a Poisson process with rate $\lambda$, and customer service times are assumed to be i.i.d.\ having exponential distribution with mean $1/\mu$. In the Erlang-A system, each customer
has a patience time and when his waiting time in queue exceeds his
patience time, he abandons the queue without service; the patience
times are assumed to be i.i.d.\ having exponential distribution with mean
$1/\alpha$.  The Erlang-A system is a special case of the systems studied in
\cite{BravDai2017}, where SSC played an important role. The systems in this paper can be represented by a one-dimensional CTMC, meaning that there is no need
to invoke SSC. Therefore, this paper illustrates only the first three
components of the framework proposed in \cite{BravDai2017}.

We will study the birth-death process
\begin{align}
X=\{X(t), t\ge 0\}, \label{eq:CTMCunscaleddef}
\end{align} 
 where $X(t)$ is the number of customers in the system at time $t$.  In the Erlang-A system, $\alpha$ is assumed to be positive and
therefore the mean patience time is finite. This guarantees that the CTMC $X$ is positive recurrent. In the Erlang-C system, $\alpha=0$, and in order for the CTMC to be positive recurrent we need to assume that the offered load to the system, defined as $R = \lambda / \mu$, satisfies
 \begin{equation}
  \label{eq:erlang-cstable}
  R  < n.
\end{equation}
For both Erlang-A and Erlang-C systems, we use $X(\infty)$ to denote the random variable having the stationary distribution of $X$. 

Consider the case when $\alpha = 0$ and \eqref{eq:erlang-cstable} is satisfied. Set 
\begin{align*}
\tilde X(\infty) = (X(\infty) - R) /\sqrt{R},
\end{align*} 
and let $Y(\infty)$ denote a continuous random
variable on $\R$ having density
\begin{equation}
  \label{eq:stddenC}
\kappa \exp\Big(\frac{1}{\mu} {\int_0^xb(y)dy}\Big),
\end{equation}
where $\kappa>0$ is a normalizing constant that makes the density integrate to one,
\begin{equation}
b(x) = \big[(x+\zeta)^--\zeta^-\big]\mu, \quad \text{ and } \quad \zeta =\big(R -n\big)/\sqrt{R}. \label{eq:bandz}
\end{equation}
\blue{Although our choice of notation does not make this explicit, we highlight that the random variable $Y(\infty)$ depends on $\lambda, \mu$, and $n$, meaning that we are actually dealing with a \emph{family} of random variables $\{Y^{(\lambda, \mu, n)}(\infty)\}_{(\lambda, \mu, n)}$. This plays a role in Lemma~\ref{lem:gradboundsCW} in Section~\ref{sec:roadmap} for example, where we need to know exactly how the gradient bounds depend on $\lambda, \mu$, and $n$. The following theorem illustrates the type of result that can be obtained by Stein's method. }
\st{The following theorem is an example of our main results, the rest of which are stated in Section~\ref{sec:results}. The framework to prove this theorem is developed in Section~\ref{sec:roadmap}, where we describe each component of the Stein framework.}
\begin{theorem}
\label{thm:erlangCW}
Consider the Erlang-C system ($\alpha = 0$). For all $n \geq 1,\\ \lambda > 0$, and $\mu > 0$ satisfying $1 \leq R < n$,
\begin{equation}
  \label{eq:erlangCW}
d_W(\tilde X(\infty), Y(\infty)) \equiv  \sup \limits_{\blue{h(x)} \in {\lipone}}  \big|\E h(\tilde X(\infty)) - \E h(Y(\infty))\big|\leq  \frac{205}{\sqrt{R}},
\end{equation}
where 
\begin{equation*}
  \label{eq:wasserstein}
\lipone=\{h: \R\to\R, \abs{h(x)-h(y)}\le
\abs{x-y}\}.
\end{equation*}
\end{theorem}
The framework developed in \cite{BravDai2017} was inspired largely by the work of Gurvich in \cite{Gurv2014} who developed methodologies to 
prove statements similar to Theorem~\ref{thm:erlangCW} for a broad class of multidimensional CTMCs. Along the way, he independently rediscovered
many of the ideas central to Stein's method in the setting of
steady-state diffusion approximations. See \cite{BravDai2017} for a more detailed discussion of Gurvich's results.
%
\begin{table}[bt]
  \begin{center}
  \begin{tabular}{rcc | r cc }
\multicolumn{3}{c}{$n=5$} & \multicolumn{3}{c}{$n=500$} \\
\hline
$R$ & $\E X(\infty)$ & Error &
$R$ & $\E X(\infty)$ & Error \\
\hline
 3        &   3.35  &0.10  & 300 & 300.00 & $6\times 10^{-14}$  \\
 4        &  6.22  & 0.20  & 400 & 400.00 & $2\times 10^{-6}$  \\
 4.9        &  51.47  &0.28 & 490& 516.79 & 0.24   \\
 4.95        &  101.48  & 0.29 & 495& 569.15& 0.28  \\
 4.99        & 501.49  & 0.29 & 499 & 970.89& 0.32 \\
  \end{tabular}
  \caption{Comparing the error $\big|\E X(\infty) - \big( R + \sqrt{R}\E Y(\infty) \big) \big|$ for different system configurations.\label{tab1}}
  \end{center}
\end{table}

Several points are worth mentioning. First, we note that Theorem~\ref{thm:erlangCW} is not a limit theorem. Steady-state approximations are usually justified by some kind of limit theorem. That is, one considers a sequence of queueing systems and proves that the corresponding sequence of steady-state distributions converges to some limiting distribution as traffic intensity approaches one, or as the number of servers goes to infinity. In contrast, our theorem holds for any finite parameter choices of $\lambda, n$, and $\mu $ satisfying \eqref{eq:erlang-cstable} and $R \geq 1$. Second, the error bound in \eqref{eq:erlangCW} is \emph{universal}, as it does not assume any relationship between $\lambda,n$, and $\mu$, other than the stability condition \eqref{eq:erlang-cstable} and the condition that $R \geq 1$. The latter condition is a mere convenience, as Theorem~\ref{thm:erlangCW} could be restated without it, but the error bound would then contain some terms involving $1/R$. One consequence of universality is that the error bound holds  when  \blue{parameters $\lambda,n$, and $\mu$  fall in  one of the 
following asymptotic regimes:
\begin{align*}
&n = \left \lceil R+ \beta R\right \rceil, \quad n = \left \lceil R + \beta \sqrt{R}\right \rceil,  \quad \text{ or }  \quad n = \left \lceil R + \beta \right \rceil,
 \end{align*}
 where $\beta > 0$ is fixed, while  $R \to \infty$. } 
The first two parameter regimes above describe the quality-driven (QD), and quality-and-efficiency-driven (QED) regimes, respectively. The last regime is the nondegenerate-slowdown (NDS) regime, which was studied in \cite{Whit2003, Atar2012}. Universal approximations were previously studied in \cite{GlynWard2003, GurvHuanMand2014}. Third, as part of the universality of Theorem~\ref{thm:erlangCW}, we see that 
\begin{align}
\big|\E X(\infty) - \big( R + \sqrt{R}\E Y(\infty) \big) \big| \leq 205. \label{eq:const_error}
\end{align}
For a fixed $n$, let $\rho = R/n \uparrow 1$. One expects that $\E X(\infty)$ be on the order of $1/(1-\rho)$. Conventional heavy-traffic limit theorems often guarantee that the left hand side of \eqref{eq:const_error} is at most $o(1/\sqrt{1-\rho})$, whereas our error is bounded by a constant regardless of the load condition. This suggests that the diffusion approximation for the Erlang-C system is accurate not only as $R \to \infty$, but also in the heavy-traffic setting when $R \to n$. Table~\ref{tab1} contains some numerical results where we calculate the error on the left side of \eqref{eq:const_error}. The constant $205$ in \eqref{eq:const_error} is unlikely to be a sharp upper bound. In this paper we did not focus too much on optimizing the upper bound, as Stein's method is not known for producing sharp constants.

From Theorem~\ref{thm:erlangCW} we know that the first moment of $\tilde X(\infty)$ can be approximated universally by the first moment of $Y(\infty)$. It is natural to ask what can be said about the approximation of higher moments. We performed some numerical experiments in which we approximate the second and tenth moments of $\tilde X(\infty)$ in a system with $n = 500$. The results are displayed in Table~\ref{tab2}. One can see that the approximation errors grow as the offered load $R$ gets closer to $n$. \blue{We will see in Section~\ref{sec:exten} that this happens because the $(m-1)$th moment appears in the approximation error of the $m$th moment. A similar phenomenon was first observed for the $M/GI/1+GI$ model in Theorem 1 of \cite{GurvHuan2016}.} \st{ suggesting that the approximation of higher moments of $\tilde X(\infty)$ is \emph{not universal}. This differs from the approximation errors of the first moment in Table~\ref{tab1}, where the errors approached some limiting value.} \st{, where Theorem~\ref{thm:poly} presents error bounds for approximating higher moments.}
\begin{table}[tb]
  \begin{center}
  \resizebox{\columnwidth}{!}{
  \begin{tabular}{rcc | c c }
 $R$ & $\E (\tilde X(\infty))^2$ & $\big| \E (\tilde X(\infty))^2 - \E(Y(\infty))^2\big|$ & $\E(\tilde X(\infty))^{10}$ & $\big| \E (\tilde X(\infty))^{10} - \E(Y(\infty))^{10}\big|$ \\
\hline
 300        &  1                    & $4.55\times 10^{-15}$ & $9.77\times 10^2$ & 31.58\\
 400        &  1                    & $5.95\times 10^{-7}$ & $9.70\times 10^{2}$ & 24.44\\
 490        &  6.96                 & 0.11                 & $7.51\times 10^{9}$ & $7.01\times 10^{8}$\\
 495        &  31.56                & 0.27                & $9.10\times 10^{12}$ & $4.34\times 10^{11}$ \\
 499        &  $9.47\times 10^{2}$  & 1.59               & $1.07\times 10^{20}$ & $1.03\times 10^{18}$ \\
 499.9        &  $9.94\times 10^{4}$  & 16.50              & $1.13\times 10^{30}$ & $1.09\times 10^{27}$ \\
  \end{tabular}
  }
  \caption{Approximating the second and tenth moments of $\tilde X(\infty)$ with $n = 500$. The approximation error grows as $R$ approaches $n$ and suggests that the diffusion approximation of higher moments is not universal.\label{tab2}}
  \end{center}
\end{table} 

%

Theorem~\ref{thm:erlangCW} provides rates of convergence under the Wasserstein metric \cite{Ross2011}. The Wasserstein metric is one of the most commonly studied metrics in the context of Stein's method. This is because the the space $\lipone$ is relatively simple to work with, but is also rich enough so that convergence under the Wasserstein metric implies the convergence in
distribution \cite{GibbSu2002}.  Another metric commonly studied in problems involving Stein's method is the Kolmogorov metric, which measures the distance between cumulative distribution functions of two random variables. The Kolmogorov distance between $\tilde X(\infty)$ and $Y(\infty)$ is
\begin{align*}
\sup_{\blue{h(x)} \in \HH_K}  \big|\E h(\tilde X(\infty)) - \E h(Y(\infty))\big|, \quad \text{ where } \quad   {\cal H}_{K}=\{1_{(-\infty, a]}(x): a\in \R \}.
\end{align*}
Theorems~\ref{thm:erlangCK} and \ref{thm:erlangAK} of Section~\ref{sec:results} involve the Kolmogorov metric. A general trend in Stein's method is that establishing convergence rates for the Kolmogorov metric often requires much more effort than establishing rates for the Wasserstein metric, and our problem is no exception. The extra difficulty always comes from the fact that the test functions belonging to the class $\HH_K$ are discontinuous, whereas the ones in $\lipone$ are Lipschitz-continuous. In Section~\ref{sec:kolmogorov}, we describe how to overcome this difficulty in our model setting. 

The first paper to have established convergence rates for steady-state diffusion approximations was \cite{GurvHuanMand2014}, which studied the Erlang-A system using an excursion based approach. Their approximation error bounds are universal. Although the authors in \cite{GurvHuanMand2014} did not study the Erlang-C system, their approach appears to be extendable to it as well. However, their method is not readily generalizable to the multi-dimensional setting.  \st{In this paper we also derive approximation error bounds for the Erlang-A model by using Stein's method. Our notion of universality is stronger than the one in \cite{GurvHuanMand2014}. In most of their results, the authors of \cite{GurvHuanMand2014} allowed $\lambda$ and $n$ to vary without restriction, but they fix $\mu$ and $\alpha$. The only exception is that they consider  separately the NDS regime when $\mu  = \mu(\lambda) = \beta \sqrt{\lambda}$ and $\lambda = n\mu + \beta_1 \mu$ for some $\beta> 0$ and $\beta_1 \in \R$. In contrast, our error bound only depends on the ratio $\alpha/\mu$; see Theorems~\ref{thm:erlangAW} and \ref{thm:erlangAK}.}

We wish to point out that the results proved in this paper are quite sharp, and proving analogous results in a high dimensional setting is likely much more difficult. For instance, the constants in the error bounds in each theorem of this paper can be recovered explicitly, but this is not true in the problem studied in \cite{BravDai2017}. Moreover, the result of \cite{BravDai2017} is restricted to the Halfin--Whitt regime, and is not universal. All of this is because the model considered there is high dimensional.

%

\subsection{Related Literature}
Diffusion approximations are a popular tool in queueing theory, and are usually ``justified'' by heavy traffic limit theorems. For example, a typical limit theorem would say that an appropriately scaled and centered version of the process $X$ in \eqref{eq:CTMCunscaleddef} converges to some limiting diffusion process as the system utilization $\rho$ tends to one.
Proving such limit theorems has been an active area of research in
the last 50 years; see, for example,
\cite{Boro1964,Boro1965,IgleWhit1970,IgleWhit1970a, Harr1978,Reim1984}
for single-class queueing networks, \cite{Pete1991, Bram1998a,
  Will1998a} for multiclass queueing networks,
\cite{KangKellLeeWill2009, YaoYe2012} for bandwidth sharing networks,
\cite{HalfWhit1981,Reed2009,DaiHeTezc2010} for many-server queues. The
convergence used in these limit theorems is the convergence in
distribution on the path space $\D([0, \infty), \R^d)$, endowed with
Skorohod $J_1$-topology \cite{EthiKurt1986,Whit2002}. The
$J_1$-topology on $\D([0,\infty), \R^d)$ essentially means convergence
in $\D([0, T], \R^d)$ for each $T>0$. In particular, it says nothing
about the convergence at ``$\infty$''. Therefore, these limit theorems
do not justify steady-state convergence.

The jump from convergence on $\D([0, T], \R^d)$ to convergence of stationary distributions was first established in the
seminal paper \cite{GamaZeev2006}, where the authors prove an
interchange of limits for generalized Jackson networks of
single-server queues.  The results in \cite{GamaZeev2006} were
improved and extended by various authors for networks of
single-servers \cite{BudhLee2009, ZhanZwar2008, Kats2010}, for
bandwidth sharing networks~\cite{YaoYe2012}, and for many-server
systems \cite{Tezc2008, GamaStol2012, Gurv2014a}.  These ``interchange
of limits'' theorems are qualitative and thus do not provide rates of
convergence as in Theorem~\ref{thm:erlangCW}.

\blue{The first uses of Stein's method for stationary distributions of Markov processes traces back to \cite{Barb1988}, where it is pointed out that Stein's method can be applied anytime the approximating distribution is the stationary distribution of a Markov proccess.  That paper considers the multivariate Poisson, which is the stationary distribution of a certain multi-dimensional birth-death process. One of the major contributions of that paper was to show how viewing the Poisson distribution as the stationary distribution of a Markov chain could be exploited to establish gradient bounds using coupling arguments; cf. the discussion around \eqref{eq:relval} of this paper. A similar idea was subsequently used for the multivariate normal distribution through its connection to the multi-dimensional Ornstein--Uhlenbeck process in \cite{Barb1990, Gotz1991}.}

\blue{Of the papers that use the connection between Stein's method and Markov processes, \cite{BrowXia2001} and the more recent \cite{KusuTudo2012} are the most relevant to this work. The former studies one-dimensional birth-death processes, with the focus being that many common distributions such as the Poisson, Binomial, Hypergeometric, Negative Binomial, etc., can be viewed as stationary distributions of a birth-death process. Although the Erlang-A and Erlang-C models are also birth-death processes, the focus in our paper is on how well these models can be approximated by diffusions, e.g.\ qualitative features of the approximation like the universality in Theorem~\ref{thm:erlangCW}. Diffusion approximations go beyond approximations of birth-death processes, with the real interest lying in cases when a higher-dimensional Markov chain collapses to a one-dimensional diffusion, e.g.\ \cite{Tezc2008,Stol2004,DaiLin2008}, or when the diffusion approximation is multi-dimensional \cite{Harr1978,Reim1984,Pete1991, Bram1998a,  Will1998a}. }

\blue{
In \cite{KusuTudo2012}, the authors apply Stein's method to one-dimensional diffusions. The motivation is again that many common distributions like the gamma, uniform, beta, etc.,  happen to be stationary distributions of diffusions. Their chief result is to establish gradient bounds for a very large class of diffusion processes, requiring only the mild condition that the drift of the diffusion be a decreasing function. However, their result cannot be applied here, because it is impossible to say how their gradient bounds depend on the parameters of the diffusion. Detailed knowledge of this dependence is crucial, because we are dealing with a \emph{family} of approximating  distributions;  cf. \eqref{eq:stddenC} and the comments below (\ref{eq:bandz}).}

\blue{Outside the diffusion approximation domain, Ying has recently
  successfully applied Stein's framework to establish error bounds
  for steady-state mean-field approximations \cite{ying2016,ying2016b}.}
\blue{There is one additional recent line of work \cite{BlanGlyn2007,  JansLeeuZwar2008, JansLeeuZwar2008a,  JansLeeuZwar2011, LeeuZhanZwar2012} that deserves mention, where the theme is corrected diffusion approximations using asymptotic series expansions. In particular, \cite{JansLeeuZwar2011} considers the Erlang-C system and  \cite{LeeuZhanZwar2012} considers the Erlang-A system. In these papers, the authors derive series expansions for various steady-state quantities of interest like the probability of waiting $\Prob(X(\infty) \geq n)$. These types of series expansions are very powerful because they allow one to approximate steady-state quantities of interest within arbitrary precision. However, while accurate, these expansions vary for different performance metrics (e.g.\ waiting probability, expected queue length), and require non-trivial effort to be derived. They also depend on the choice of parameter regime, e.g.\ Halfin-Whitt. In contrast, the results provided by the Stein approach can be viewed as more robust because they capture multiple performance metrics and  multiple parameter regimes at the same time.}

\subsection{Notation}
\label{sec:notation}
For two random variables $U$ and $V$, define their Wasserstein distance to be 
\begin{equation}
  \label{eq:dW}
  d_W(U, V) = \sup_{\blue{h(x)} \in \lipone} \abs{\E[h(U)] -\E[h(V)]},
\end{equation}
where 
\begin{equation*}
  \label{eq:wasserstein}
\lipone=\{h: \R\to\R, \abs{h(x)-h(y)}\le
\abs{x-y}\}.
\end{equation*}
It is known, see for example \cite{Ross2011}, convergence under the 
Wasserstein metric implies convergence in distribution. When $\lipone$ in (\ref{eq:dW}) is replaced by 
\begin{align}
{\cal H}_{K}=\{1_{(-\infty, a]}(x): a\in \R \}, \label{eq:classkolm}
\end{align}
the corresponding distance is the Kolmogorov distance, denoted by $d_K(U, V)$. For $a, b \in \R$, we use $a^+, a^-, a \wedge b$, and $a \vee b$ to denote $\max(a,0)$, $\max(-a,0)$,  $\min(a,b)$, and $\max(a, b)$, respectively.

The rest of the paper is structured as follows. In Section~\ref{sec:results} we state our main results. In Section~\ref{sec:roadmap} we present the Stein framework needed to prove our main results, Theorems~\ref{thm:erlangCW}--\ref{thm:erlangAK}, by introducing three ingredients central to our framework. Namely, the Poisson equation and gradient bounds, generator coupling, and moment bounds. In Section~\ref{sec:proofW} we prove Theorem~\ref{thm:erlangCW}\st{ and \ref{thm:erlangAW}}, which deals with the Wasserstein metric. The Kolmogorov metric presents an additional challenge, because the test functions are discontinuous. In Section~\ref{sec:kolmogorov} we address this new challenge\st{ and then prove Theorems~\ref{thm:erlangCK} and \ref{thm:erlangAK}}. In Section~\ref{sec:exten}, we discuss \blue{the approximation of} higher moments in the Erlang-C model. Proofs of technical lemmas are left to the four appendices.
\section{Main results}
\label{sec:results}
Recall the offered load $R = \lambda/\mu$. For notational convenience we define $\delta>0$ as
\begin{equation*}
  \label{eq:delta}  \delta = \frac{1}{ \sqrt{R}} = \sqrt{\frac{\mu}{\lambda}}.
\end{equation*}
Let $x(\infty)$ be the unique solution to the flow balance equation
\begin{align}
  \lambda = \big(x(\infty)\wedge n\big )\mu + \big(x(\infty)-n\big)^+ \alpha. \label{eq:xinf}
\end{align}
Here, $x(\infty)$ is interpreted as the equilibrium number of customers in the corresponding fluid model,
and is the point at which the arrival rate equals the departure rate. The latter is the sum of the service completion rate and the customer abandonment rate with $x(\infty)$ customers in the system. One can check that the flow balance equation has a unique solution $x(\infty)$ given by
 \begin{equation}
   \label{eq:equi}
  x(\infty) =
  \begin{cases}
    n + \frac{\lambda-n\mu }{\alpha} & \text{ if } R \geq n,\\
    R & \text{ if } R < n.
  \end{cases}
 \end{equation}
By noting that the number of busy servers $x(\infty) \wedge	n$ equals $n$ minus the number of idle servers $(x(\infty)-n)^-$, the equation in \eqref{eq:xinf} becomes
\begin{equation}
  \label{eq:lambdamudiff}
  \lambda - n \mu = \big(x(\infty)-n\big)^+\alpha - \big(x(\infty)-n\big)^-\mu.
\end{equation}
We note that $x(\infty)$ is well-defined even when $\alpha = 0$, because in that case we always assume that $R < n$.

We consider the CTMC
\begin{align} \label{eq:scaledctmc}
\tilde X = \{ \tilde X(t) \equiv \delta(X(t) - x(\infty)),\ t \geq 0\},
\end{align}
and let the random variable $\tilde X(\infty)$ have its stationary distribution. Define 
\begin{equation}
\zeta =\delta\big(x(\infty) -n\big),
\end{equation}
and
\begin{equation}
  \label{eq:b-erlang-A}
  b(x) = \big[(x+\zeta)^--\zeta^-\big]\mu -\big[(x+\zeta)^+ -\zeta^+\big]\alpha \quad \text{ for } x\in \R,
\end{equation}
with convention that $\alpha$ is set to be zero in the Erlang-C
system. \blue{For intuition about the quantity $\zeta$, we note that in the Erlang-C system satisfying (\ref{eq:erlang-cstable}), 
  \begin{align*}
    n = R -\zeta \sqrt{R}.
  \end{align*}
Thus,  $-\zeta=\abs{\zeta}>0$ is precisely the 
``safety coefficient'' in the square-root safety-staffing principle~\cite[equation (15)]{GansKoolMand2003}.
We point out that the event $\{\tilde X(t) = -\zeta \}$ corresponds to the event $\{X(t) = n\}$.}

 Throughout this paper, let $Y(\infty)$ denote a continuous random
variable on $\R$ having density
\begin{equation}
  \label{eq:stdden}
  \nu(x)= \kappa \exp\Big(\frac{1}{\mu} {\int_0^xb(y)dy}\Big),
\end{equation}
where $\kappa>0$ is a normalizing constant that makes the density integrate to one. Note that these definitions are consistent with \eqref{eq:stddenC} and \eqref{eq:bandz}.


\begin{theorem}
\label{thm:erlangAW}
Consider the Erlang-A system ($\alpha > 0$). There exists an increasing function $C_W : \R_+ \to \R_+$ such that for all $n \geq 1, \lambda > 0, \mu>0$, and $\alpha>0$ satisfying $R \geq 1$,
\begin{equation}
  \label{eq:erlangAW}
  d_W(\tilde X(\infty), Y(\infty)) \le C_W(\alpha/\mu)  \delta.
\end{equation}
\end{theorem}
\begin{remark}
The proof of Theorems~\ref{thm:erlangCW} and \ref{thm:erlangAW} uses the same ideas. Therefore, for the sake of brevity, we only give an outline for the proof of Theorem~\ref{thm:erlangAW} and its ingredients in Appendix~\ref{app:AWoutline}, without filling in all the details. It is for this reason that we do not write out the explicit form of $C_W(\alpha/\mu)$, although it can be obtained from the proof. The same is true for Theorem~\ref{thm:erlangAK} below.

\end{remark}
Given two random variables $U$ and $V$, \cite[Proposition 1.2]{Ross2011} implies that when $V$ has a density that is bounded by $C>0$, 
\begin{equation}
  \label{eq:dwdKbound}
  d_K(U, V) \le \sqrt{2C d_W(U, V)}.
\end{equation}
\blue{At best, \eqref{eq:dwdKbound} and Theorems~\ref{thm:erlangCW} and \ref{thm:erlangAW} imply a convergence rate of $\sqrt{\delta}$ for $d_K(\tilde X(\infty), Y(\infty))$. However, this bound is typically too crude, and the following two theorems show that convergence happens at rate $\delta$. Theorem~\ref{thm:erlangCK} is proved in Section~\ref{sec:kproof}. The proof of Theorem~\ref{thm:erlangAK} is outlined in Appendix~\ref{app:AKoutline}.
}
\st{Although (\ref{eq:dwdKbound}) often implies Kolmogorov metric
convergence from the Wasserstein metic convergence, the constant $C$ in
(\ref{eq:dwdKbound}) may depend on $\lambda, n, \mu$, and $\alpha$ in general. Furthermore, the upper bound on the Kolmogorov metric implied by \eqref{eq:dwdKbound} is usually suboptimal, in the sense that the order of magnitude of the bound is incorrect.  Therefore, for the bounds on the Kolmogorov metric we need to prove separately the following two theorems.}

\begin{theorem}
\label{thm:erlangCK}
Consider the Erlang-C system ($\alpha = 0$). For all $n \geq 1,\\ \lambda > 0$, and $\mu>0$ satisfying $1 \leq R < n$,
\begin{equation}
  \label{eq:erlangCK}
  d_K(\tilde X(\infty), Y(\infty)) \le 188\delta.
\end{equation}
\end{theorem}

\begin{theorem}
\label{thm:erlangAK}
Consider the Erlang-A system ($\alpha > 0$). There exists an increasing function $C_K : \R_+ \to \R_+$ such that for all $n \geq 1, \lambda > 0, \mu>0$, and $\alpha>0$ satisfying $R \geq 1$,
\begin{equation}
  \label{eq:erlangAK}
  d_K(\tilde X(\infty), Y(\infty)) \le C_K(\alpha/\mu)\delta.
\end{equation}
\end{theorem}

 Theorems \ref{thm:erlangCW} and
\ref{thm:erlangCK} are new, but versions of Theorems~\ref{thm:erlangAW} and \ref{thm:erlangAK} were first proved in the
pioneering paper \cite{GurvHuanMand2014} using an excursion based approach. \blue{However, our notion of universality in those theorems is stronger than the one in \cite{GurvHuanMand2014}, because most of their results require $\mu$ and $\alpha$ to be fixed. The only exception is in Appendix C of that paper, where the authors consider the NDS regime with $\mu  = \mu(\lambda) = \beta \sqrt{\lambda}$ and $\lambda = n\mu + \beta_1 \mu$ for some $\beta> 0$ and $\beta_1 \in \R$. }


\blue{We emphasize that both constants $C_W$ and $C_K$ are increasing in $\alpha/\mu$. That is, for an Erlang-A system with a higher abandonment rate with respect to its service rate, our error bound becomes larger. The reader may wonder why these constants depend on $\alpha/\mu$, while the constant in the Erlang-C theorems does not depend on anything. Despite our best efforts, we were unable to get rid of the dependency on $\alpha/\mu$. The reason is that the Erlang-C model depends on only three parameters ($\lambda, \mu, n$), while the Erlang-A model also depends on $\alpha$. 
As a result, both the gradient bounds and moment bounds have an extra factor $\alpha/\mu$  in the Erlang-A model. 
\st{This means that more parameters show up in both gradient and moment bounds in the Erlang-A case.} For example, compare Lemma~\ref{lem:gradboundsCW} in Section~\ref{sec:wassergrad_bounds} with Lemma~\ref{lem:gradboundsAWunder} in Appendix~\ref{app:grad_A}. } 

%
%
%
%
%
%
%
%
%
%

\section{Outline of the Stein Framework}
\label{sec:roadmap}
In this section we introduce the main tools needed to prove
Theorems~\ref{thm:erlangCW}--\ref{thm:erlangAK}. At this point we do
not restrict ourselves to either the Erlang-A or Erlang-C systems, as
the framework outlined here is generic and holds for both systems.

The following is an informal outline of the rest of this section. We know that $\tilde X(\infty)$ follows the stationary distribution of the CTMC $\tilde X$, and that this CTMC has a generator $G_{\tilde X}$. To $Y(\infty)$, we will associate a diffusion process with generator $G_Y$. We will start by fixing a test function $h :\R \to \R$ and deriving the identity 
\begin{align}
\big| \E h(\tilde X(\infty)) - \E h(Y(\infty)) \big| = \big| \E G_{\tilde X} f_h(\tilde X(\infty)) - \E  G_{Y} f_h(\tilde X(\infty)) \big|, \label{eq:outline_identity}
\end{align}
where $f_h(x)$ is a solution to the Poisson equation
\begin{align*}
G_{Y} f_h(x) = \E h(Y(\infty)) - h(x), \quad x \in \R.
\end{align*}
We then focus on bounding the right hand side of \eqref{eq:outline_identity}, which is easier to handle than the left hand side. This is done by performing a Taylor expansion of $G_{\tilde X} f_h(x)$ in Section~\ref{sec:taylor}. To bound the error term from the Taylor expansion, we require bounds on various moments of $\big| \tilde X(\infty) \big|$, as well as the derivatives of $f_h(x)$. We refer to the former as moment bounds, and the latter as gradient bounds. These are presented in Sections~\ref{sec:momentbounds} and \ref{sec:wassergrad_bounds}, respectively.

\subsection{The Poisson Equation of a Diffusion Process}
\label{sec:diffPoisson}
The random variable $Y(\infty)$ in
Theorems~\ref{thm:erlangCW}--\ref{thm:erlangAK} is well-defined and
its density is given in (\ref{eq:stdden}). It turns out that $Y(\infty)$
has the stationary distribution of a diffusion process $Y=\{Y(t), t\ge
0\}$, which we will define shortly. We do not prove this claim in this paper since it is not used anywhere in this paper. Nevertheless, it is helpful to think of $Y(\infty)$ in the context of diffusion processes.  The diffusion process $Y$ is the one-dimensional piecewise
Ornstein--Uhlenbeck (OU) process.  Its generator is given by
\begin{equation}
  \label{eq:GY}
G_Y f(x) =  b(x)f'(x) +  \mu f''(x) \text{ \quad for $x \in \R$}, \ f \in C^2(\R),
\end{equation}
where $b(x)$ is defined in (\ref{eq:b-erlang-A}).
Clearly,  $b(0)=0$, and  $b(x)$ is Lipschitz continuous. Indeed,
\begin{displaymath}
\abs{  b(x) -b(y)} \le (\alpha\vee \mu) \abs{x-y} \quad \text{ for } x, y\in \R.
\end{displaymath}
Since the diffusion process $Y$ depends on parameters $\lambda, n, \mu$, and
$\alpha$ in an arbitrary way, there is no appropriate way to talk about the limit of $Y(\infty)$ in terms of these parameters. Therefore, we call $Y$ a \emph{diffusion model}, as opposed to a
\emph{diffusion limit}. Having a diffusion model whose input
parameters are directly taken from the corresponding Markov chain
model is critical to achieve \emph{universal accuracy}. In other words, this diffusion model is accurate in any parameter regime, from underloaded, to critically loaded, and to overloaded. Diffusion models, not limits,  of queueing networks with a given set of parameters have been advanced in \cite{HarrWill1987,HarrNguy1993, DaiHe2013,GlynWard2003, Gurv2014, GurvHuanMand2014,GurvHuan2016}.

The main tool we use is known as the Poisson equation. It allows us to say that $Y(\infty)$ is a good estimate for $\tilde X(\infty)$ if the generator of $Y$ behaves similarly to the generator of $\tilde X$, where $\tilde X$ is defined in \eqref{eq:scaledctmc}. Let $\cal{H}$ be a class of functions $h : \R \to \R$, to be specified shortly.  For each function $\blue{h(x)} \in \cal{H}$, consider the Poisson equation 
\begin{align} \label{eq:poisson}
G_Y f_h(x) =  b(x) f_h'(x) + \mu f_h''(x) = \E h(Y(\infty)) - h(x), \quad  x\in \R.
\end{align}
One may verify by differentiation that for all functions $h: \R \to \R$ satisfying $\E \big|h(Y(\infty))\big| < \infty$, the Poisson equation  has a family of solutions of the form
\begin{align} \label{eq:poissonsolution}
f_h(x) = a_1 + \int_{0}^{x} \bigg[ a_2 \frac{1}{\nu(u)} + \frac{1}{\nu(u)} \int_{-\infty}^{u} \frac{1}{\mu } \big(\E h(Y(\infty)) - h(y)\big) \nu(y) dy\bigg] du, 
\end{align}
where $a_1, a_2 \in \R$ are arbitrary constants, and $\nu(x)$ is as in \eqref{eq:stdden}. 

In this paper, we take $\HH = \lipone$ when we deal
with the Wasserstein metric (Theorems~\ref{thm:erlangCW} and
\ref{thm:erlangAW}), and we choose $\HH = \HH_K$ when we deal with the
Kolmogorov metric (Theorems~\ref{thm:erlangCK} and
\ref{thm:erlangAK}). We claim that $\abs{\E h(Y(\infty))} <
\infty$. Indeed, when $\HH = \HH_K$, this clearly holds. 
When $\mathcal{H} = \lipone$, without loss of generality we take $h(0)= 0$ in \eqref{eq:poisson}, and use the Lipschitz
property of $h(x)$ to see that
\begin{align*}
\abs{\E h(Y(\infty))} \leq  \E \abs{Y(\infty)} < \infty, 
\end{align*}
where the finiteness of $\E \abs{Y(\infty)}$ will be proved in
\eqref{eq:fbound7}. 

 From (\ref{eq:poisson}), one has
\begin{align}
\big| \E h(\tilde X(\infty)) - \E h(Y(\infty)) \big| =&\ \big| \E G_Y f_h(\tilde X(\infty)) \big|.\label{eq:error1}
\end{align}
In (\ref{eq:error1}), $\tilde X(\infty)$ has the stationary
distribution of the CTMC $\tilde X$, not necessarily defined on the same
probability space of $Y(\infty)$.  Actually, $\tilde X(\infty)$ in
(\ref{eq:error1}) can be replaced by any other random variable,
although one does not expect the error on the right side to be small
if this random variable has no relationship with the diffusion process
$Y$.

\subsection{Comparing Generators}
\label{sec:compare}
To prove Theorems~\ref{thm:erlangCW}--\ref{thm:erlangAK}, we need to
bound the right side of (\ref{eq:error1}).  The CTMC $\tilde X$
defined in \eqref{eq:scaledctmc} also has a generator.  We bound the
right side of (\ref{eq:error1}) by showing that the diffusion
generator in (\ref{eq:GY}) is similar to the CTMC generator.

 For any $k \in \Z_+$, we define $x = x_k = \delta(k - x(\infty))$. Then for any function $f:\R\to\R$, the generator of $\tilde X$ is given by  
\begin{align}\label{eq:GX}
G_{\tilde X} f(x) = \lambda (f(x + \delta) - f(x)) + d(k) (f(x-\delta) - f(x)),
\end{align}
where
\begin{align} \label{eq:deathrate}
d(k) = \mu (k \wedge n) + \alpha (k-n)^+,
\end{align}
is the departure rate corresponding to the system having $k$ customers.  One may check that 
\begin{align}
b(x) = \delta( \lambda - d(k)). \label{eq:bk}
\end{align}
The relationship between $G_{\tilde X}$ and the stationary distribution of $\tilde X$ is illustrated by the following lemma.
\begin{lemma} \label{LEM:GZ}
Let $f(x): \R \to \R$ be a function such that $\abs{f(x)} \leq C(1+x)^2$ for some $C > 0$ (i.e.\ $f(x)$ is dominated by a quadratic function), and assume that the CTMC $\tilde X$ is positive recurrent. Then
\begin{align*}
\E \big[ G_{\tilde X} f(\tilde X(\infty)) \big] = 0.
\end{align*}
\end{lemma}
\begin{remark}
We will see in Lemma~\ref{lem:gradboundsCW} later this section, in Lemmas~\ref{lem:gradboundsCK} and \ref{lem:gradboundsAK} of Section~\ref{sec:kolmogorov}, and in  Lemma~\ref{lem:gradboundsAWunder} of Appendix~\ref{app:grad_A} that there is a family of solutions to the Poisson equation \eqref{eq:poisson} whose first derivatives grow at most linearly in both the Wasserstein and Kolmgorov settings, meaning that these solutions satisfy the conditions of Lemma~\ref{LEM:GZ}.
\end{remark}
The proof of Lemma~\ref{LEM:GZ} is provided in Appendix~\ref{app:pflemgz}, and relies on Proposition 3 of \cite{GlynZeev2008}. Suppose for now that for any $\blue{h(x)} \in \cal{H}$, the solution to the Poisson equation $f_h(x)$ satisfies the conditions of Lemma~\ref{LEM:GZ}. We can apply Lemma~\ref{LEM:GZ} to \eqref{eq:error1} to see that
\begin{align}
\big| \E h(\tilde X(\infty)) - \E h(Y(\infty)) \big| =&\ \big| \E G_Y f_h(\tilde X(\infty)) \big| \notag \\
=&\ \big| \E G_{\tilde X} f_h(\tilde X(\infty)) - \E G_Y f_h(\tilde X(\infty)) \big| \notag \\
\leq&\ \E \big| G_{\tilde X} f_h(\tilde X(\infty)) -  G_Y f_h(\tilde X(\infty)) \big|. \label{eq:gen_bound}
\end{align}
While the two random variables on the left side of
(\ref{eq:gen_bound}) are usually defined on different probability
spaces, the two random variables on the right side of
(\ref{eq:gen_bound}) are both functions of $\tilde X(\infty)$. Thus, we have 
achieved a coupling through Lemma \ref{LEM:GZ}. Setting up the Poisson equation is a generic first step one performs any time one wishes to apply Stein's method to a problem. The next step is to bound the equivalent of our $\big| \E G_Y f_h(\tilde X(\infty)) \big|$. This is usually done by using a coupling argument. However, this coupling is always problem specific, and is one of the greatest sources of difficulty one encounters when applying Stein's method. In our case, this generator coupling is natural because we deal with Markov processes $\tilde X$ and $Y$.

\blue{Since the generator completely characterizes the behavior of a Markov process, it is natural to expect that convergence of generators implies convergence of Markov processes. Indeed, the question of weak convergence was studied in detail, for instance in \cite{EthiKurt1986}, using the martingale problem of Stroock and Varadhan \cite{StroVara1979}. However, \eqref{eq:gen_bound} lets us go beyond weak convergence, both because different choices of $h(x)$ lead to different metrics of convergence, and also because the question of convergence rates can be answered. One interpretation of the Stein approach is to view $f_h(x)$ as a Lyapunov function that gives us information about $h(x)$. Instead of searching very hard for this Lyapunov function, the Poisson equation \eqref{eq:poisson} removes the guesswork. However, this comes at the cost of $f_h(x)$ being defined implicitly as the solution to a differential equation.}
%

\subsection{Taylor Expansion}
\label{sec:taylor}

To bound the right side of (\ref{eq:gen_bound}), 
we study the difference $G_{\tilde X} f_h(x) - G_Y f_h(x)$. For that we perform a Taylor expansion on $G_{\tilde X}f_h(x)$. To illustrate this, suppose
that $f_h''(x)$ exists for all $x \in \R$, and is absolutely
continuous. Then for any $k \in \Z_+$, and $x = x_k = \delta(k -
x(\infty))$, we recall that $b(x) = \delta(\lambda - d(k))$ in \eqref{eq:bk} to see
that
\begingroup
\allowdisplaybreaks
\begin{align*}
G_{\tilde X} f_h(x) =&\ \lambda (f_h(x + \delta) - f_h(x)) + d(k) (f_h(x-\delta) - f_h(x)) \notag \\
=&\ f_h'(x) \delta (\lambda - d(k)) + \frac{1}{2} \delta^2 f_h''(x)(\lambda + d(k)) \notag  \\
&+ \frac{1}{2} \lambda \delta^2 (f_h''(\xi) - f_h''(x)) + \frac{1}{2} d(k) \delta^2 (f_h''(\eta) - f_h''(x)) \notag \\
=&\ f_h'(x)  b(x) +  \frac{1}{2} \delta^2 (2\lambda - \frac{1}{\delta} b(x)) f_h''(x) \notag \\
&+ \frac{1}{2} \mu (f_h''(\xi) - f_h''(x)) + \frac{1}{2} ( \lambda -  \frac{1}{\delta}b(x))  \delta^2 (f_h''(\eta) - f_h''(x))\\
=&\ G_Y f_h(x) - \frac{1}{2} \delta f_h''(x)b(x) + \frac{1}{2} \mu (f_h''(\xi) - f_h''(x)) \\
&+ \frac{1}{2} ( \mu -  \delta b(x))   (f_h''(\eta) - f_h''(x)),
\end{align*}%
\endgroup
where $\xi \in [x, x+\delta]$ and $\eta \in [x-\delta,x]$. We invoke the absolute continuity of $f_h''(x)$ to get
\begin{align}
&\ \Big| \E h(\tilde X(\infty)) - \E h(Y(\infty)) \Big| \notag \\
 \leq&\ \frac{1}{2} \delta \E \Big[ \big|f_h''(\tilde X(\infty))b(\tilde X(\infty)) \big| \Big] + \frac{\mu}{2} \E \bigg[ \int_{\tilde X(\infty)}^{\tilde X(\infty) + \delta} \abs{f_h'''(y)}dy\bigg] \notag \\
&+ \frac{\mu}{2} \E \bigg[ \int_{\tilde X(\infty)-\delta}^{\tilde X(\infty)} \abs{f_h'''(y)}dy\bigg] + \frac{1}{2} \delta\E \bigg[  \big| b(\tilde X(\infty))\big| \int_{\tilde X(\infty)-\delta}^{\tilde X(\infty)} \abs{f_h'''(y)}dy\bigg]. \label{eq:first_bounds}
\end{align}
As one can see, to show that the right hand side of \eqref{eq:first_bounds} vanishes as $\delta \to 0$, we must be able to bound the derivatives of $f_h(x)$; we refer to these as gradient bounds. Furthermore, we will also need bounds on moments of $\big| \tilde X(\infty) \big|$; we refer to these as moment bounds. Both moment and gradient bounds will vary between the Erlang-A or Erlang-C setting, and the gradient bounds will be different for the Wasserstein, and Kolmogorov settings. Moment bounds will be discussed shortly, and gradient bounds 
in the Wasserstein setting
will be presented in Section~\ref{sec:wassergrad_bounds}.  We discuss the Kolmogorov setting separately in Section~\ref{sec:kolmogorov}. In that case we face an added difficulty because $f_h''(x)$ has a discontinuity, and we cannot use \eqref{eq:first_bounds} directly.
\subsection{Moment Bounds}
\label{sec:momentbounds}
\blue{The following lemma presents the necessary} moment bounds to bound \eqref{eq:first_bounds} in the Erlang-C model, and is proved in Appendix~\ref{app:momCproof}. These moment
  bounds are used in both the Wasserstein and the
  Kolmogorov metric settings.
\begin{lemma} \label{lem:moment_bounds_C}
Consider the Erlang-C model ($\alpha = 0$). For all $n \geq 1, \lambda > 0$, and $\mu>0$ satisfying $0 < R < n $,
\allowdisplaybreaks
\begin{align} 
&\E \Big[(\tilde X(\infty))^2 1(\tilde X(\infty) \leq -\zeta)\Big] \leq \frac{4}{3} + \frac{2\delta^2}{3}, \label{eq:xsquaredelta}\\
&\E \Big[ \big|\tilde X(\infty) 1(\tilde X(\infty) \leq -\zeta)\big| \Big] \leq \sqrt{\frac{4}{3} + \frac{2\delta^2}{3}}, \label{eq:xminusdelta}\\
&\E \Big[ \big|\tilde X(\infty) 1(\tilde X(\infty) \leq -\zeta)\big| \Big] \leq 2\abs{\zeta} \label{eq:xminuszeta}\\
&\E \Big[\big|\tilde X(\infty)1(\tilde X(\infty) \geq -\zeta)\big| \Big] \leq \frac{1}{\abs{\zeta}} + \frac{\delta^2}{4\abs{\zeta}} + \frac{\delta}{2}, \label{eq:xplus}\\
&\Prob(\tilde X(\infty) \leq -\zeta) \leq (2+\delta)\abs{\zeta}. \label{eq:idle_prob}
\end{align}
\end{lemma} 
 
 \st{Observe that we have the freedom to use either \eqref{eq:xminusdelta} or \eqref{eq:xminuszeta}, depending on which bound suits our needs better in any given situation.}  We see that \eqref{eq:xplus} immediately implies that when $\delta \leq 1$, 
\begin{align}
\abs{\zeta} \Prob( \tilde X(\infty) \geq -\zeta) \leq&\ \abs{\zeta} \wedge \E \Big[\big| \tilde X(\infty) 1(\tilde X(\infty) \geq -\zeta)\big| \Big]\notag \\
 \leq&\ \abs{\zeta} \wedge \Big(\frac{1}{\abs{\zeta}} + \frac{\delta^2}{4\abs{\zeta}} + \frac{\delta}{2} \Big) \notag \\
\leq&\ 7/4, \label{eq:xplusbound}
\end{align}
where to get the last inequality we considered separately the cases where $\abs{\zeta} \leq 1$ and $\abs{\zeta} \geq 1$. This bound will be used in the proofs of Theorems~\ref{thm:erlangCW} and \ref{thm:erlangCK}. 

\blue{One may wonder why the bounds are separated using the indicators $1\{\tilde X(\infty) \leq -\zeta\}$ and $1\{\tilde X(\infty) \geq -\zeta\}$. This is related to the drift $b(x)$ appearing in \eqref{eq:first_bounds}, and the fact that $b(x)$ takes different forms on the regions  $x\leq -\zeta$ and $x\geq -\zeta$. Furthermore, it may be unclear at this point why both \eqref{eq:xminusdelta} and \eqref{eq:xminuszeta} are needed, as the left hand side in both bounds is identical. The reason is that \eqref{eq:xminusdelta} is an $O(1)$ bound (we think of $\delta \leq 1$), whereas \eqref{eq:xminuszeta} is an $O(\abs{\zeta})$ bound. The latter is only useful when $\abs{\zeta}$ is small, but this is nevertheless an essential bound to achieve universal results. As we will see later, it negates $1/\abs{\zeta}$ terms that appear in \eqref{eq:first_bounds} from $f_h''(x)$ and $f_h'''(x)$. }
\st{Note that the event $\{\tilde X(\infty) \leq -\zeta\}$ is equal to the event $\{X(\infty) \leq n\}$. In particular, \eqref{eq:idle_prob} gives a bound on the probability that no customers are waiting in the system. }

\blue{For the Erlang-A model, we also require moment bounds similar to those stated in Lemma~\ref{lem:moment_bounds_C}. Both the proof, and subsequent usage, of the Erlang-A moment bounds are similar to the proof and subsequent usage of the Erlang-C moment bounds. We therefore delay the precise statement of the Erlang-A bounds until Lemma~\ref{lem:moment_bounds_A_under} in Appendix~\ref{app:mom_A}, to avoid distracting the reader with a bulky lemma.}

%

\subsection{Wasserstein Gradient Bounds} \label{sec:wassergrad_bounds}
 \st{In this section, we present the gradient bounds on the solutions to the Poisson equation in the Wasserstein setting ($\HH = \lipone$). In Section~\ref{sec:proofW} we prove Theorems~\ref{thm:erlangCW} and \ref{thm:erlangAW} by applying these gradient bounds together with the moment bounds from Section~\ref{sec:momentbounds} to \eqref{eq:first_bounds}.}

Recall the Poisson equation \eqref{eq:poisson} and the family of solutions to this equation is given by \eqref{eq:poissonsolution}; in particular, this family is parametrized by constants $a_1, a_2 \in \R$. Fix $ \blue{h(x)} \in \mathcal{H} = \lipone$, and let $f_h(x)$ be a solution to the Poisson equation. The following lemma presents  Wasserstein gradient bounds for the Erlang-C model. It is proved in Appendix~\ref{app:wgradient}. \st{ We recall that $R \leq n$ implies $\zeta \leq 0$, and $R \geq n$ implies $\zeta \geq 0$. For both lemmas, we recall the Poisson equation \eqref{eq:poisson} }
\begin{lemma} \label{lem:gradboundsCW}
Consider the Erlang-C model ($\alpha = 0$). The solution to the Poisson equation $f_h(x)$ is twice continuously differentiable, with an absolutely continuous second derivative. Fix a solution in \eqref{eq:poissonsolution} with parameter $a_2 = 0$. Then for all $n \geq 1, \lambda > 0$, and $\mu>0$ satisfying $0 < R < n $, 
\begin{align}
\abs{f_h'(x)} \leq&\
\begin{cases}
\frac{1}{\mu }(6.5 + 4.2/\abs{\zeta}), \quad x \leq -\zeta,\\
\frac{1}{\mu }\frac{1}{\abs{\zeta}}(x + 1 + 2/\abs{\zeta}), \quad x \geq -\zeta.
\end{cases}\label{eq:WCder1} \\
\abs{f_h''(x)} \leq&\ 
\begin{cases}
\frac{32}{\mu }( 1 + 1/\abs{\zeta}), \quad x \leq -\zeta,\\
\frac{1}{\mu \abs{\zeta}}, \quad x \geq -\zeta,
\end{cases} \label{eq:WCder2}
\end{align}
and for those $x \in \R$ where $f_h'''(x)$ exists, 
\begin{align}
\abs{f_h'''(x)} \leq&\ 
\begin{cases}
\frac{1}{\mu}(23 + 13/\abs{\zeta}) , \quad x \leq -\zeta,\\
2/\mu, \quad x \geq -\zeta.
\end{cases} \label{eq:WCder3}
\end{align}
\end{lemma}

\begin{remark}
This lemma validates the Taylor expansion used to obtain \eqref{eq:first_bounds} \st{in the Wasserstein setting,} because  $f_h''(x)$ is absolutely continuous. Furthermore, $f_h(x)$ \st{the function  considered in these lemmas all} satisfies the conditions of Lemma~\ref{LEM:GZ}, because $f_h'(x)$ grows at most linearly.
\end{remark}
\blue{Gradient bounds, also known as Stein factors, are central to any application of Stein's method. The problem of gradient bounds for diffusion approximations can be divided into two cases: the one-dimensional case, and the multi-dimensional case. In the former, the Poisson equation in \eqref{eq:poisson} is an ordinary differential equation (ODE) corresponding to a one-dimensional diffusion process. In the latter, the Poisson equation is a partial differential equation (PDE) corresponding to a multi-dimensional diffusion process.}

\blue{The one-dimensional case is simpler, because the explicit form of $f_h(x)$ is given to us by \eqref{eq:poissonsolution}. To bound $f_h'(x)$ and $f_h''(x)$ we can analyze \eqref{eq:poissonsolution} directly, as we do in the proof of Lemma~\ref{lem:gradboundsCW}. This direct analysis can be used as a go-to method for one-dimensional diffusions, but fails in the multi-dimensional case, because closed form solutions for PDE's are not typically known. In this case, it helps to  exploit the fact that $f_h(x)$ satisfies
\begin{align}
f_h(x) = \int_{0}^{\infty} \Big(\E \big[h(Y(t))\ |\ Y(0) = x\big] - \E h(Y(\infty))\Big) dt, \label{eq:relval}
\end{align}
where $Y = \{Y(t), t\geq 0\}$ is a diffusion process with generator $G_{Y}$ \cite{PardVere2001}. To bound derivatives of $f_h(x)$ based on \eqref{eq:relval}, one may use coupling arguments to bound finite differences of the form $\frac{1}{s}(f_h(x+ s) - f_h(x))$. For examples of coupling arguments, see \cite{Barb1988, BarbBrow1992, BrowXia2001, BarbXia2006, GanXia2015}. A related paper to these types of gradient bounds is \cite{Stol2015}, where the author used a variant of \eqref{eq:relval} for the fluid model of a flexible-server queueing system as a Lyapunov function.  As an alternative to coupling, one may combine \eqref{eq:relval} with a-priori Schauder estimates from PDE theory, as was done in \cite{Gurv2014}. }


 \blue{Just like we did with the moment bounds, we delay the Erlang-A gradient bounds to Lemma~\ref{lem:gradboundsAWunder} in Appendix~\ref{app:grad_A}.} We are now ready to prove Theorem~\ref{thm:erlangCW}\st{ and \ref{thm:erlangAW}}.

\section{Proof of Theorem~\ref{thm:erlangCW}}
\label{sec:proofW}
In this section we prove Theorem~\ref{thm:erlangCW}\st{and \ref{thm:erlangAW}}. Fix $\blue{h(x)} \in \lipone$, and recall that the family of solutions to the Poisson equation is given by \eqref{eq:poissonsolution}. For the remainder of Section~\ref{sec:proofW}, we fix one such solution $f_h(x)$ with $a_2 = 0$. Then by Lemma~\ref{lem:gradboundsCW} \st{\blue{(for $\alpha = 0$)} and Lemma~\ref{lem:gradboundsAWunder} \blue{(for $\alpha > 0$)}}, $f_h''(x)$ is absolutely continuous, implying that \eqref{eq:first_bounds} holds; we recall it here as
\begin{align}
&\ \Big| \E h(\tilde X(\infty)) - \E h(Y(\infty)) \Big| \notag \\
 \leq&\ \frac{1}{2} \delta \E \Big[ \big|f_h''(\tilde X(\infty))b(\tilde X(\infty)) \big| \Big] + \frac{\mu}{2} \E \bigg[ \int_{\tilde X(\infty)}^{\tilde X(\infty) + \delta} \abs{f_h'''(y)}dy\bigg] \notag \\
&+ \frac{\mu}{2} \E \bigg[ \int_{\tilde X(\infty)-\delta}^{\tilde X(\infty)} \abs{f_h'''(y)}dy\bigg] + \frac{1}{2} \delta\E \bigg[  \big| b(\tilde X(\infty))\big| \int_{\tilde X(\infty)-\delta}^{\tilde X(\infty)} \abs{f_h'''(y)}dy\bigg], \label{eq:second_bounds}
\end{align}
\blue{
where $\delta = 1/\sqrt{R} = \sqrt{\mu/\lambda}.$
}

The proof of Theorem~\ref{thm:erlangCW}\st{ and \ref{thm:erlangAW}} simply involves applying the moment bounds and gradient bounds to show that the error bound in \eqref{eq:second_bounds} is small.
\begin{proof}[Proof of Theorem~\ref{thm:erlangCW}]
Throughout the proof we assume that $R \geq 1$, or equivalently, $\delta \leq 1$. We bound each of the terms on the right side of \eqref{eq:second_bounds} individually. We recall here that the support of $\tilde X(\infty)$ is a $\delta$-spaced grid, and in particular this grid contains the point $-\zeta$. In the bounds that follow, we will often consider separately the cases where $\tilde X(\infty) \leq -\zeta - \delta$, and $\tilde X(\infty) \geq -\zeta$. We recall that 
\begin{align}
b(x) = \mu \big[(x+\zeta)^--\zeta^-\big] = 
\begin{cases}
-\mu x, \quad x \leq -\zeta,\\
\mu \zeta, \quad x \geq -\zeta,
\end{cases} \label{eq:bexpand}
\end{align}
and apply the moment bounds \eqref{eq:xminusdelta}, \eqref{eq:xminuszeta}, and the gradient bound \eqref{eq:WCder2}, to see that
\begin{align*}
\E \Big[ \big|f_h''(\tilde X(\infty))b(\tilde X(\infty)) \big| \Big] \leq &\ 32(1 + 1/\abs{\zeta})\E \Big[\big|\tilde X(\infty) \big| 1(\tilde X(\infty) \leq -\zeta - \delta) \Big] \\
&+ \Prob(\tilde X(\infty) \geq -\zeta) \\
\leq &\ 32(1 + 1/\abs{\zeta})\bigg(2\abs{\zeta} \wedge \sqrt{\frac{4}{3} + \frac{2\delta^2}{3}}\bigg) + 1 \\
\leq &\ 32\Big(\sqrt{\frac{4}{3} + \frac{2\delta^2}{3}} + 2\Big) + 1 \\
\leq&\ 32\big(\sqrt{2} + 2\big) + 1 \leq 111.
\end{align*}
Next, we use \eqref{eq:idle_prob} and the gradient bound in \eqref{eq:WCder3} to get
\begin{align*}
&\ \frac{\mu }{2} \E \bigg[\int_{\tilde X(\infty)}^{\tilde X(\infty)+\delta} \abs{f_h'''(y)} dy \bigg] \\
\leq&\   \frac{\delta}{2}\Big((23 + 13/\abs{\zeta})\Prob(\tilde X(\infty) \leq -\zeta - \delta) + 2\Prob(\tilde X(\infty) \geq -\zeta)\Big) \\
\leq&\ \frac{\delta}{2}\Big(23 + \frac{13}{\abs{\zeta}}(3 \abs{\zeta}) \Big) \leq 31\delta.
\end{align*}
By a similar argument, we can show that
\begin{align*}
\frac{\mu }{2}\E \bigg[\int_{\tilde X(\infty)-\delta}^{\tilde X(\infty)} \abs{f_h'''(y)} dy \bigg] \leq 31\delta,
\end{align*}
with the only difference in the argument being that we consider the cases when $\tilde X(\infty) \leq -\zeta$ and $\tilde X(\infty) \geq -\zeta + \delta$, instead of $\tilde X(\infty) \leq -\zeta -\delta$ and $\tilde X(\infty) \geq -\zeta$. Lastly, we use the form of $b(x)$, the moment bounds \eqref{eq:xminusdelta}, \eqref{eq:xminuszeta}, and \eqref{eq:xplusbound}, and the gradient bound \eqref{eq:WCder3} to get 
\begin{align*}
&\ \frac{\delta}{2} \E \bigg[ \big| b(\tilde X(\infty)) \big| \int_{\tilde X(\infty)-\delta}^{\tilde X(\infty)} \abs{f_h'''(y)} dy \bigg] \\
\leq&\ \frac{\delta^2}{2}\Big( (23 + 13/\abs{\zeta})\E \Big[ \big| \tilde X(\infty)\big| 1(\tilde X(\infty) \leq -\zeta) \Big] + 2\abs{\zeta} \Prob(\tilde X(\infty) \geq -\zeta + \delta)\Big)\\
\leq&\ \frac{\delta^2}{2}\Big( (23 + 13/\abs{\zeta})\Big(2\abs{\zeta} \wedge \sqrt{\frac{4}{3} + \frac{2\delta^2}{3}}\Big) + 14/4 \Big)\\
\leq&\ \frac{\delta^2}{2}\bigg(23 \sqrt{2} + 26 + 14/4 \bigg) \leq 32\delta^2.
\end{align*}
Hence, from \eqref{eq:first_bounds} we conclude that for all $R \geq 1$, and $\blue{h(x)} \in \lipone$,
\begin{align}
&\ \Big| \E h(\tilde X(\infty)) - \E h(Y(\infty)) \Big| \leq  \delta (111 + 31 + 31 + 32 \delta) \leq 205\delta, \label{eq:intermproofWC}
\end{align}
which proves Theorem~\ref{thm:erlangCW}.
\end{proof}
\section{The Kolmogorov Metric}
\label{sec:kolmogorov}
In this section we prove Theorem~\ref{thm:erlangCK}\st{ and \ref{thm:erlangAK}}, which is stated in the Kolmogorov setting. The biggest difference between the  Wasserstein and Kolmogorov settings is that in the latter, the test functions $h(x)$ used in the Poisson equation \eqref{eq:poisson} are discontinuous. For this reason, the gradient bounds from Lemma~\ref{lem:gradboundsCW} and Lemma~\ref{lem:gradboundsAWunder} in Appendix~\ref{app:grad_A} do not hold anymore, and new gradient bounds need to be derived separately for the Kolmogorov setting; we present these new gradient bounds in Section~\ref{sec:kgrad}. Furthermore, the solution to the Poisson equation no longer has a continuous second derivative, meaning that the Taylor expansion we used to derive the upper bound in \eqref{eq:first_bounds} is invalid. We discuss an alternative to \eqref{eq:first_bounds} in Section~\ref{sec:ktaylor}. This alternative bound contains a new error term that cannot be handled by the gradient bounds, nor the moment bounds. This term appears because the solution to the Poisson equation has a discontinuous second derivative, and to bound it we present Lemma~\ref{lem:kolmfixC}. We then prove Theorem~\ref{thm:erlangCK} \st{and \ref{thm:erlangAK}} in  Section~\ref{sec:kproof}. 

\subsection{Kolmogorov Gradient Bounds}
\label{sec:kgrad}

Recall that in the Kolmogorov setting, we take the class of test functions for the Poisson equation \eqref{eq:poisson} to be $\HH_K$ defined in \eqref{eq:classkolm}. For the statement of the following two lemmas, we fix $a \in \R$ and set $h(x) = 1_{(-\infty, a]}(x)$. We use $f_a(x)$ instead of $f_h(x)$ to denote a solution to the Poisson equation. We recall that the family of solutions to the Poisson equation is parametrized by constants $a_1, a_2 \in \R$.  The following lemmas state the gradient bounds in the Kolmogorov setting.
\begin{lemma} \label{lem:gradboundsCK}
Consider the Erlang-C model ($\alpha = 0$). Any solution to the Poisson equation $f_a(x)$ is continuously differentiable, with an absolutely continuous derivative. Fix a solution in \eqref{eq:poissonsolution} with parameter $a_2 = 0$. Then for all $n \geq 1, \lambda > 0$, and $\mu > 0$ satisfying $0 < R  < n$,
\begin{align}
\abs{f_a'(x)} \leq 
\begin{cases}
5/\mu , \quad x \leq -\zeta, \\
\frac{1}{\mu \abs{\zeta}}, \quad x \geq -\zeta,
\end{cases} \label{eq:KCder1}
\end{align}
and for all $x \in \R$,
\begin{align}
\abs{f_a''(x)} \leq 3/\mu, \label{eq:KCder2}
\end{align}
where $f_a''(x)$ is understood to be the left derivative at the point $x = a$.
\end{lemma}

\begin{lemma}\label{lem:gradboundsAK}
Consider the Erlang-A model ($\alpha > 0$). Any solution to the Poisson equation $f_a(x)$ is continuously differentiable, with an absolutely continuous derivative.  Fix a solution in \eqref{eq:poissonsolution} with parameter $a_2 = 0$, and fix $n \geq 1, \lambda > 0, \mu > 0$, and $\alpha > 0$. If $0 < R  \leq n$ (an underloaded system), then
\begin{align}
\abs{f_a'(x)} \leq 
\begin{cases}
\frac{1}{\mu }\sqrt{2\pi}e^{1/2} , \quad x \leq -\zeta, \\
\frac{1}{\mu } \Big(\sqrt{\frac{\pi}{2} \frac{\mu }{\alpha}} \wedge\frac{1}{\abs{\zeta}}\Big), \quad x \geq -\zeta,
\end{cases} \label{eq:ACuder1}
\end{align}
and if $n \leq R$ (an overloaded system), then
\begin{align}
\abs{f_a'(x)} \leq 
\begin{cases}
\frac{1}{\mu }\sqrt{\frac{\pi}{2}} , \quad x \leq -\zeta, \\
\frac{1}{\mu } \sqrt{\frac{\pi}{2}}\Big(1 + \sqrt{ \frac{\mu }{\alpha}} \Big), \quad x \geq -\zeta.
\end{cases} \label{eq:ACoder1}
\end{align}
Moreover,  for all  $\lambda>0, n \geq 1, \mu > 0$, and $\alpha > 0$, and all $x \in \R$,
\begin{align}
\abs{f_a''(x)} \leq 3/\mu, \label{eq:ACder2}
\end{align}
where $f_a''(x)$ is understood to be the left derivative at the point $x = a$.
\end{lemma}
Lemmas~\ref{lem:gradboundsCK} and \ref{lem:gradboundsAK} are proved in Appendix~\ref{app:kgradient}. Unlike the Wasserstein setting, these lemmas do not guarantee that $f_a''(x)$ is absolutely continuous. Indeed, for any $a \in \R$, 
\blue{ substituting $h(x) = 1_{(-\infty, a]}(x)$ into~\eqref{eq:poisson} gives us}
\begin{align*}
\mu f_a''(x) =  \Prob(Y(\infty) \leq a) - 1_{(-\infty, a]}(x) - b(x) f_a'(x).
\end{align*}
\blue{Since $b(x) f_a'(x)$ is a continuous function, the above equation} implies that $f_a''(x)$ is discontinuous at the point $x = a$. \blue{Thus,} we can no longer use the error bound in \eqref{eq:first_bounds}, and require a different expansion of $G_{\tilde X} f_a(x)$. 
\subsection{Alternative Taylor Expansion}
\label{sec:ktaylor}
To get an error bound similar to \eqref{eq:first_bounds}, we first define 
\begin{align}
\epsilon_1(x) =&\ \int_x^{x+\delta} (x+\delta -y)(f_a''(y)-f_a''(x-))dy, \label{eq:eps1def} \\
\epsilon_2(x) =&\  \int_{x-\delta}^{x} (y-(x-\delta))(f_a''(y)-f_a''(x-))dy. \label{eq:eps2def}
\end{align}
Now observe that
\begin{align*}
f_a(x+\delta) -f_a(x) =&\  f_a'(x) \delta  + \int_x^{x+\delta} (x+\delta -y)f_a''(y)dy  \\
=&\  f_a'(x) \delta  +  \frac{1}{2}\delta^2 f_a''(x-) \\
&+  \int_x
^{x+\delta} (x+\delta -y)(f_a''(y)-f_a''(x-))dy  \\
 =&\  f_a'(x) \delta  + \frac{1}{2}\delta^2 f_a''(x-) +  \epsilon_1(x),
\end{align*}
and
\begin{align*}
  (f_a(x-\delta) -f_a(x))  =&\ - f_a'(x) \delta  +
 \int_{x-\delta}^x (y-(x-\delta))f_a''(y)dy  \\
=&\ -f_a'(x) \delta  +  \frac{1}{2}\delta^2 f_a''(x-) \\
&+  \int_{x-\delta}^{x} (y-(x-\delta))(f_a''(y)-f_a''(x-))dy  \\
=&\ -f_a'(x) \delta  +  \frac{1}{2}\delta^2 f_a''(x-) + \epsilon_2(x).
\end{align*}
For $k \in \Z_+$ and  $x=x_k=\delta(k-x(\infty))$, we recall the forms of $G_Y f_a(x)$ and $G_{\tilde X} f_a(x)$ from \eqref{eq:GY} and \eqref{eq:GX} to see that
\begin{align*}
  G_{\tilde X}f_a(x) =&\ \lambda  \delta f_a'(x) + \lambda \frac{1}{2}\delta^2 f_a''(x-) + \lambda \epsilon_1(x)\\
  & - d(k) \delta f_a'(x) + d(k) \frac{1}{2} \delta^2 f_a''(x-) + d(k)
  \epsilon_2(x)\\
=&\ b(x) f_a'(x) + \lambda \frac{1}{2}\delta^2 f_a''(x-) + \lambda \epsilon_1(x)\\
&+ (\lambda -\frac{1}{\delta}b(x)) \frac{1}{2} \delta^2 f_a''(x-) + (\lambda-\frac{1}{\delta}b(x) )
  \epsilon_2(x)\\
=&\ G_Y f(x) - b(x)\frac{1}{2} \delta f_a''(x-) + \lambda(\epsilon_1(x)+\epsilon_2(x)) - \frac{1}{\delta}b(x)\epsilon_2(x),
\end{align*}
where in the second equality we used the fact that $b(x) = \delta( \lambda - d(k))$, and in the last equality we use that $\delta^2 \lambda = \mu$. Combining this with \eqref{eq:gen_bound}, we have an error bound similar to \eqref{eq:first_bounds}:
\begin{align}
&\ \Big| \Prob(\tilde X(\infty) \leq a) - \Prob(Y(\infty) \leq a) \Big| \notag \\
 \leq&\ \frac{1}{2} \delta \E \Big[ \big|f_a''(\tilde X(\infty)-)b(\tilde X(\infty)) \big| \Big] + \lambda \E \Big[ \big| \epsilon_1(\tilde X(\infty)) \big| \Big] \notag \\
& + \lambda \E \Big[ \big|\epsilon_2(\tilde X(\infty))\big| \Big] + \frac{1}{\delta} \E \Big[ \big| b(\tilde X(\infty))\epsilon_2(\tilde X(\infty))\big| \Big], \label{eq:third_bounds}
\end{align}
where $\epsilon_1(x)$ and $\epsilon_2(x)$ are as in \eqref{eq:eps1def} and \eqref{eq:eps2def}. To bound the error terms in \eqref{eq:third_bounds} that are associated with $\epsilon_1(x)$ and $\epsilon_2(x)$, we need to analyze the difference $f_a''(y) - f_a''(x-)$ for $\abs{x-y} \leq \delta$. Since $f_a(x)$ is a solution to the Poisson equation (\ref{eq:poisson}), we see that for any $x, y \in \R$ with $y \neq a$, 
\begin{align*}
f_a''(y)-f_a''(x-) = \frac{1}{\mu} \big[ 1_{(-\infty, a]}(x) - 1_{(-\infty, a]}(y) + b(x)f_a'(x) - b(y)f_a'(y) \big].
\end{align*}
Therefore, for any $y \in [x, x + \delta]$  with $y \neq a$,
\begin{align}
&\ \abs{ f_a''(y)-f_a''(x-)} \notag \\
\le &\  \frac{1}{\mu} \big[ 1_{(a-\delta, a]}(x) + \abs{b(x)}\abs{f_a'(x) -f_a'(y)}
+ \abs{b(x)-b(y))}\abs{f_a'(y)} \big] \notag \\
 & \le  \frac{1}{\mu} \big[ 1_{(a-\delta, a]}(x) + \delta \abs{b(x)}\norm{f''} + \abs{b(x)-b(y))}\abs{f_a'(y)} \big], \label{eq:eps1b2}
\end{align}
and likewise, for any $y \in [x-\delta, x]$ with $y \neq a$, 
\begin{align}
&\ \abs{ f_a''(y)-f_a''(x-)} \notag \\
\le &\ \frac{1}{\mu} \big[ 1_{(a, a + \delta]}(x) + \abs{b(x)}\abs{f_a'(x) -f_a'(y)}
+ \abs{b(x)-b(y))}\abs{f_a'(y)} \big] \notag \\
 & \le  \frac{1}{\mu} \big[ 1_{(a, a + \delta]}(x) + \delta \abs{b(x)}\norm{f''} + \abs{b(x)-b(y))}\abs{f_a'(y)} \big].\label{eq:eps2b2}
\end{align}
The inequalities above contain the indicators $1_{(a-\delta, a]}(x)$ and $1_{(a, a + \delta]}(x)$. When we consider the upper bound in \eqref{eq:third_bounds}, these indicators will manifest themselves as probabilities $\Prob(a - \delta < \tilde X(\infty) \leq a)$ and $\Prob(a  < \tilde X(\infty) \leq a+\delta)$. To this end we present the following lemma, which will be used in the proof of Theorem~\ref{thm:erlangCK}.
\begin{lemma}\label{lem:kolmfixC}
Consider the Erlang-C model ($\alpha = 0$). Let $W$ be an arbitrary random variable with cumulative distribution function $F_W:\R \to [0,1]$. Let $\omega(F_W)$ be the modulus of continuity of $F_W$, defined as 
\begin{align*}
\omega(F_W) = \sup_{\substack{x, y \in \R \\ x \neq y}} \frac{\abs{F_W(x)-F_W(y)}}{\abs{x-y}}.
\end{align*}
Recall that $d_K(\tilde X(\infty), W)$ is the Kolmogorov distance between $X(\infty)$ and $W$. Then for any $a \in \R$, $n \geq 1$, and $0 < R< n$,  
\begin{align*}
\Prob( a - \delta < \tilde X(\infty) \leq a + \delta) \leq  \omega(F_W)2\delta  + d_K(\tilde X(\infty), W) + 9\delta^2 + 8\delta^4.
\end{align*}
\end{lemma}
This lemma is proved in Appendix~\ref{app:kolmfixC}. We will apply Lemma~\ref{lem:kolmfixC} with $W = Y(\infty)$ in the proof of Theorem~\ref{thm:erlangCK} that follows. The following lemma guarantees that the modulus of continuity of the cumulative distribution function of $Y(\infty)$ is bounded by a constant independent of $\lambda, n$, and $\mu$. Its proof is provided in Appendix~\ref{app:densboundC}.
\begin{lemma} \label{lem:densboundC}
Consider the Erlang-C model ($\alpha = 0$), and let $\nu(x)$ be the density of $Y(\infty)$. Then for for all $n \geq 1, \lambda > 0$, and $\mu>0$ satisfying $0 < R<n $,
\begin{align*}
\abs{\nu(x)} \leq \sqrt{\frac{2}{\pi}} , \quad x \in \R.
\end{align*}
\end{lemma}
Lemmas~\ref{lem:kolmfixC} and \ref{lem:densboundC} are stated for the Erlang-C model, but one can easily repeat the arguments in the proofs of those lemmas to prove analogues for the Erlang-A model. Therefore, we state the following lemmas without proof.
\begin{lemma}\label{lem:kolmfixA}
Consider the Erlang-A model ($\alpha > 0$). Let $W$ be an arbitrary random variable with cumulative distribution function $F_W:\R \to [0,1]$. Let $\omega(F_W)$ be the modulus of continuity of $F_W$. Then for any $a \in \R$, $\alpha > 0$, $n \geq 1$, and $R > 0$,  
\begin{align*}
&\ \Prob( a - \delta < \tilde X(\infty) \leq a + \delta) \\
\leq&\  \omega(F_W)2\delta  + d_K(\tilde X(\infty), W) + 9 \Big( \frac{\alpha}{\mu } \vee 1 \Big)\delta^2 + 8\Big( \frac{\alpha}{\mu } \vee 1 \Big)^2\delta^4.
\end{align*}
\end{lemma}
\begin{lemma} \label{lem:densboundA}
Consider the Erlang-A model ($\alpha > 0$), and let $\nu(x)$ be the density of $Y(\infty)$. Fix $n \geq 1, \lambda > 0, \mu>0$, and $\alpha > 0$. If $0 < R \leq n $, then
\begin{align*}
\abs{\nu(x)} \leq \sqrt{\frac{2}{\pi}} , \quad x \in \R,
\end{align*}
and  if $n  \leq R$, then
\begin{align*}
\abs{\nu(x)} \leq \sqrt{\frac{2}{\pi}} \sqrt{\frac{\alpha}{\mu } }, \quad x \in \R.
\end{align*}
\end{lemma}

\subsection{Proof of Theorem~\ref{thm:erlangCK}}
\label{sec:kproof}
\begin{proof}[Proof of Theorem~\ref{thm:erlangCK}] 
Throughout the proof we assume that $R \geq 1$, or equivalently, $\delta \leq 1$. For $h(x) = 1_{(-\infty, a]}(x)$, we let $f_a(x)$ be a solution the Poisson equation \eqref{eq:poisson} with parameter $a_2 = 0$. In this proof we will show that for all $a \in \R$, 
\begin{align}
&\ \abs{\Prob(\tilde X(\infty) \leq a) - \Prob(Y(\infty) \leq a)} \leq \frac{1}{2}\Prob(a - \delta < \tilde X(\infty) \leq a + \delta) + 75\delta, \label{eq:intermproofKC}
\end{align}
The upper bound in \eqref{eq:intermproofKC} is similar to \eqref{eq:intermproofWC}, however \eqref{eq:intermproofKC} has the extra term
\begin{align}
\frac{1}{2}\Prob(a - \delta < \tilde X(\infty) \leq a + \delta). \label{eq:extraterm}
\end{align}
The reason this term appears in the Kolmogorov setting but not in the Wasserstein setting is because $f_a''(x)$ is discontinuous in the Kolmogorov case, as opposed to the Wasserstein case where $f_h''(x)$ is continuous. Applying Lemmas~\ref{lem:kolmfixC} and \ref{lem:densboundC} to the right hand side of \eqref{eq:intermproofKC}, and taking the supremum over all $a \in \R$ on both sides, we see that 
\begin{align*}
&\ d_K(\tilde X(\infty), Y(\infty)) \leq  \frac{1}{2}d_K(\tilde X(\infty), Y(\infty)) + 2\sqrt{\frac{2}{\pi}} \delta  + 9\delta^2 + 8\delta^4 +  75\delta,
\end{align*}
or 
\begin{align*}
d_K(\tilde X(\infty), Y(\infty)) \leq 188\delta.
\end{align*}
\blue{We want to add that Lemma~\ref{lem:kolmfixC} makes heavy use of the birth-death structure of the Erlang-C model, and that it is not obvious how to handle \eqref{eq:extraterm} more generally.}

To prove Theorem~\ref{thm:erlangCK} it remains to verify \eqref{eq:intermproofKC}, which we now do. The argument we will use is similar to the argument used to prove \eqref{eq:intermproofWC} in Theorem~\ref{thm:erlangCW}. We will bound each of the terms in \eqref{eq:third_bounds}, which we recall here as 
\begin{align*}
&\ \Big| \Prob(\tilde X(\infty) \leq a) - \Prob(Y(\infty) \leq a) \Big| \notag \\
 \leq&\ \frac{1}{2} \delta \E \Big[ \big|f_a''(\tilde X(\infty)-)b(\tilde X(\infty)) \big| \Big] + \lambda \E \Big[ \big| \epsilon_1(\tilde X(\infty))\big| \Big] \notag \\
& + \lambda \E \Big[ \big| \epsilon_2(\tilde X(\infty))\big| \Big] + \frac{1}{\delta} \E \Big[ \big| b(\tilde X(\infty))\epsilon_2(\tilde X(\infty))\big| \Big].
\end{align*}
We also recall the form of $b(x)$ from \eqref{eq:bexpand}. We use the moment bounds \eqref{eq:xminusdelta} and \eqref{eq:xplusbound}, and the gradient bound \eqref{eq:KCder2} to see that
\begin{align}
&\ \E \Big[ \big|f_a''(\tilde X(\infty)-)b(\tilde X(\infty)) \big|\Big]  \notag \\
\leq&\ \frac{3}{\mu}\E \Big[\big|b(\tilde X(\infty)) \big| \Big] \notag \\
=&\ 3  \E \Big[\big|\tilde X(\infty) 1(\tilde X(\infty) \leq  -\zeta - \delta)\big| \Big]+ 3 \abs{\zeta} \Prob (\tilde X(\infty) \geq -\zeta) \notag  \\
\leq&\ 3\sqrt{\frac{4}{3} + \frac{2\delta^2}{3}} + 3\Big(\abs{\zeta}\wedge \E \Big[ \big|\tilde X(\infty)\big| 1(\tilde X(\infty) \geq -\zeta) \Big]\Big) \notag \\
\leq&\ 3\sqrt{2} + \frac{21}{4} \leq 10. \label{eq:term1kolmc}
\end{align}
Next, we use \eqref{eq:eps1b2}, \eqref{eq:term1kolmc}, and the gradient bound \eqref{eq:KCder1} to get
\begingroup
\allowdisplaybreaks
\begin{align*}
&\ \lambda \E \Big[ \big|\epsilon_1(\tilde X(\infty))\big| \Big] \\
=&\ \lambda \E \bigg[\int_{\tilde X(\infty)}^{\tilde X(\infty)+\delta} (\tilde X(\infty)+\delta -y) \big|f_a''(y)-f_a''(\tilde X(\infty)-)\big|dy\bigg] \\
\leq&\ \frac{\lambda}{\mu} \E \bigg[ 1_{(a-\delta, a]}(\tilde X(\infty))\int_{\tilde X(\infty)}^{\tilde X(\infty)+\delta} (\tilde X(\infty)+\delta -y)dy\bigg] \\
&+ \frac{\lambda}{\mu}  \delta^3 \E \Big[ \big| b(\tilde X(\infty)) \big| \Big]\norm{f_a''} + \frac{\lambda}{\mu} \delta \E \bigg[\int_{\tilde X(\infty)}^{\tilde X(\infty)+\delta} \big|b(\tilde X(\infty))-b(y))\big| \big|f_a'(y)\big|dy \bigg] \\
\leq&\ \frac{1}{2} \Prob( a - \delta < \tilde X(\infty) \leq a) + 10 \delta +  5\delta \\
\equiv&\ \frac{1}{2} \Prob( a - \delta < \tilde X(\infty) \leq a) + 15\delta,
\end{align*}%
\endgroup
where in the last inequality we used the fact that for $y \in [\tilde X(\infty), \tilde X(\infty) + \delta]$, 
\begin{align*}
b(\tilde X(\infty)) - b(y) = \mu \delta 1(\tilde X(\infty) \leq -\zeta - \delta).
\end{align*}
By a similar argument, one can check that
\begin{align*}
\lambda \E \Big[ \big|\epsilon_2(\tilde X(\infty))\big| \Big]\leq&\ \frac{1}{2} \Prob( a  < \tilde X(\infty) \leq a+\delta) +  15\delta,
\end{align*}
with the only difference in the argument being that we consider the cases when $\tilde X(\infty) \leq -\zeta$ and $\tilde X(\infty) \geq -\zeta + \delta$, instead of $\tilde X(\infty) \leq -\zeta -\delta$ and $\tilde X(\infty) \geq -\zeta$.
Lastly, we use the first inequality in \eqref{eq:eps2b2} to see that
\begin{align*}
&\ \frac{1}{\delta} \E \Big[ \big| b(\tilde X(\infty))\epsilon_2(\tilde X(\infty))\big| \Big]\\
\leq&\ \frac{1}{\mu} \E \bigg[ \big| b(\tilde X(\infty))\big| \int_{\tilde X(\infty)-\delta}^{\tilde X(\infty)}  \Big[ 1_{(a,a+\delta]}(\tilde X(\infty)) \\
& \hspace{5cm}+\big|b(\tilde X(\infty))\big|\Big(\big|f_a'(\tilde X(\infty))\big| + \big|f_a'(y)\big| \Big) \\
& \hspace{5cm} + \big|b(\tilde X(\infty))-b(y))\big|\big|f_a'(y)\big| \Big]dy \bigg]\\
\leq&\  \delta \frac{1}{\mu} \E \Big[ \big|b(\tilde X(\infty))\big| \Big] + \delta\frac{1}{\mu} \E \Big[ \big|b^2(\tilde X(\infty))f_a'(\tilde X(\infty))\big| \Big] \\
&+  \frac{1}{\mu}\E \bigg[ \big| b^2(\tilde X(\infty))\big| \int_{\tilde X(\infty)-\delta}^{\tilde X(\infty)} \abs{f_a'(y)}dy\bigg] +5\delta^2 \E \Big[  \big| \tilde X(\infty)1(\tilde X(\infty) \leq -\zeta)\big| \Big]\\
\leq&\  \frac{10}{3} \delta  + \delta\frac{1}{\mu} \E \Big[ \big|b^2(\tilde X(\infty))f_a'(\tilde X(\infty))\big| \Big] \\
&+  \frac{1}{\mu}\E \bigg[ \big| b^2(\tilde X(\infty))\big| \int_{\tilde X(\infty)-\delta}^{\tilde X(\infty)} \abs{f_a'(y)}dy\bigg] +5 \sqrt{2}\delta^2,
\end{align*}
where in the last inequality we used \eqref{eq:term1kolmc} and the moment bound \eqref{eq:xminusdelta}. Now by \eqref{eq:xsquaredelta} and \eqref{eq:xplusbound},
\begin{align*}
&\ \delta\frac{1}{\mu} \E \Big[ \big| b^2(\tilde X(\infty))f_a'(\tilde X(\infty))\big| \Big] \\
\leq&\ 5\delta  \E \big[ \tilde X^2(\infty)1( \tilde X(\infty) \leq -\zeta ) \big] + \delta \abs{\zeta} \Prob(\tilde X(\infty) \geq -\zeta + \delta)\\
\leq&\ 10 \delta + \delta \frac{7}{4} \leq 12\delta,
\end{align*}
and similarly,
\begin{align*}
\frac{1}{\mu}\E \bigg[ \big| b^2(\tilde X(\infty))\big| \int_{\tilde X(\infty)-\delta}^{\tilde X(\infty)} \abs{f_a'(y)}dy\bigg]\leq&\ 12 \delta.
\end{align*}
Therefore,
\begin{align*}
\frac{1}{\delta} \E \Big[ \big| b(\tilde X(\infty))\epsilon_2(\tilde X(\infty))\big| \Big] \leq \frac{10}{3}\delta  + 24\delta + 5\sqrt{2}\delta^2 \leq 35 \delta.
\end{align*}
This verifies \eqref{eq:intermproofKC} and concludes the proof of Theorem~\ref{thm:erlangCK}.
\end{proof}

\section{Extension: Erlang-C Higher Moments}
\label{sec:exten}
In this section we consider the approximation of higher moments for the Erlang-C model. We begin with the following result.

\begin{theorem}\label{thm:poly}
Consider the Erlang-C system ($\alpha = 0$), and fix an integer $m > 0$. There exists a constant $C = C(m)$, such that for all $n \geq 1, \lambda > 0$, and $\mu>0$ satisfying $1 \leq R < n$,
\begin{equation}
  \label{eq:highmomCW}
\big|\E (\tilde X(\infty))^m - \E (Y(\infty))^m\big|\leq (1+1/\abs{\zeta}^{m-1})C(m)\delta,
\end{equation} 
where $\zeta$ is defined in \eqref{eq:bandz}.
\end{theorem}
The proof of this theorem follows the standard Stein framework in Section~\ref{sec:roadmap}, but we do not provide it in this paper. The most interesting aspect of \eqref{eq:highmomCW} is the appearance of $1/\abs{\zeta}^{m-1}$ in the bound on the right hand side, \blue{which of course only matters when $\abs{\zeta}$ is small}. \st{This suggest that the diffusion approximation of higher moments may not hold universally.} To \blue{check whether the bound is sharp}, we performed some numerical experiments illustrated in Table~\ref{tab3}. The results suggest that the approximation error does indeed grow like $1/\abs{\zeta}^{m-1}$. \st{ meaning that the diffusion approximation is not universal for higher moments.} 

A better way to understand the growth parameter $1/\abs{\zeta}^{m-1}$ is through its relationship with $\E (\tilde X(\infty))^{m-1}$. We claim that  $\E (\tilde X(\infty))^{m-1} \approx 1/\abs{\zeta}^{m-1}$  for small values of $\abs{\zeta}$. The following lemma, which is proved in Appendix~\ref{app:order_mag},  is needed.
\begin{lemma} \label{lem:order_mag}
For any integer $m \geq 1$, and all $n \geq 1, \lambda > 0$, and $\mu>0$ satisfying $R < n$,
\begin{align}
\lim_{\zeta \uparrow 0} \abs{\zeta}^{m}\E (Y(\infty))^{m} = m!.
\end{align}
\end{lemma}
Multiplying both sides of \eqref{eq:highmomCW} by $\abs{\zeta}^{m}$ and applying Lemma~\ref{lem:order_mag}, we see that for all $n \geq 1, \lambda > 0$, and $\mu>0$ satisfying $1 \leq R < n$,
\begin{align*}
\lim_{\zeta \uparrow 0} \abs{\zeta}^{m}\E (\tilde X(\infty))^{m} = m!.
\end{align*}
In other words,  we can rewrite \eqref{eq:highmomCW} as 
\begin{align*}
&\ \big|\E (\tilde X(\infty))^m - \E (Y(\infty))^m\big|\\
\leq&\ \Big(1+\frac{1}{\abs{\zeta}^{m-1} \big|\E (\tilde X(\infty))^{m-1}\big|}\big|\E (\tilde X(\infty))^{m-1}\big|\Big)C(m)\delta \\
\leq&\ \Big(1+\big|\E (\tilde X(\infty))^{m-1}\big|\Big)\tilde C(m)\delta,
\end{align*}
where $\tilde C(m)$ is a redefined version of $C(m)$. That the approximation error in Table~\ref{tab3} increases is then attributed to the fact that $\E \tilde X(\infty)$ increases as $\zeta \uparrow 0$. As we mentioned before, the appearance of the $(m-1)$th moment in the approximation error of the $m$th moment was also observed recently in \cite{GurvHuan2016} for the virtual waiting time in the $M/GI/1+GI$ model, potentially suggesting a general trend.

\st{In \cite{GurvHuanMand2014}, the authors prove error bounds for the diffusion approximation of higher moments in the Erlang-A model. They show that with $\alpha$ and $\mu$ fixed, the approximation error of the $m$th moment is bounded by $C(\alpha, \mu, m)/\sqrt{\lambda}$, where $C(\alpha, \mu, m) > 0$ is a constant that depends on $\mu, \alpha$, and $m$ (they also prove a similar result when the system is in the NDS regime). In contrast to the Erlang-C model, the approximation of the Erlang-A model remains universal even for higher moments. }
\begin{table}[h!]
  \begin{center}
  \resizebox{\columnwidth}{!}{
  \begin{tabular}{rccc | ccc }
$R$ &$\abs{\zeta}$ & $\E (\tilde X(\infty))^2$ & Error & $\abs{\zeta}\times $Error & $\abs{\zeta}^{0.5}\times $Error & $\abs{\zeta}^{1.5}\times $Error \\
\hline
 499  & $4.48 \times 10^{-2}$  &  $9.47\times 10^{2}$  & 1.59 & $7.10\times 10^{-2}$ & 0.34 & $1.50\times 10^{-2}$ \\
 499.9  & $4.50 \times 10^{-3}$   &  $9.94\times 10^{4}$  & 16.50 & $7.38\times 10^{-2}$ & 1.10 & $4.94\times 10^{-3}$\\
 499.95 &  $2.20 \times 10^{-3}$ &  $3.99\times 10^{5}$  & 33.08 & $7.40\times 10^{-2}$ & 1.56& $3.50\times 10^{-3}$ \\
 499.99 & $4.47 \times 10^{-4}$  &  $9.99\times 10^{6}$  & 165.67 & $7.41\times 10^{-2}$ & 3.50 & $1.57\times 10^{-3}$\\
  \end{tabular}
  }
  \end{center}
  \caption{The error term above equals $\big| \E (\tilde X(\infty))^2 - \E (Y(\infty))^2 \big|$ and grows as $R \to n$. The error term still grows when multiplied by $\abs{\zeta}^{0.5}$, and the error term shrinks to zero when multiplied by $\abs{\zeta}^{1.5}$. However, when multiplied by $\abs{\zeta}$, the error term appears to converge to some limiting value, suggesting that the error does indeed grow at a rate of $1/\abs{\zeta}$. We observed consistent behavior for higher moments of $\tilde X(\infty)$ as well. \label{tab3}}
\end{table} 

%

%

\section*{Acknowledgement}
This research is supported in part by NSF Grants CNS-1248117, CMMI-1335724, and
CMMI-1537795. The first two authors were stimulated from discussions with  the participants of the 2015 Workshop on New Directions in Stein's Method held at the Institute for Mathematical Sciences at the National University of Singapore and  they would like to thank the financial support from the Institute. The authors also wish to thank two anonymous referees for suggestions to improve the presentation of the paper.

\newpage

\appendix

\section*{Appendices}
\setcounter{section}{0}
Appendix~\ref{app:moment} handles the moment bounds, while Appendix~\ref{app:gradbounds} handles the gradient bounds. Appendix~\ref{app:EAproofs} contains outlines for the proofs of Theorems~\ref{thm:erlangAW} and \ref{thm:erlangAK}, and in Appendix~\ref{app:misc} we prove several miscellaneous lemmas.

\section{Moment Bounds} \label{app:moment}
In Section~\ref{app:momCproof}, we first prove Lemma~\ref{lem:moment_bounds_C}, establishing the moment bounds for Erlang-C model.
In Section~\ref{app:mom_A}, we state and prove Lemma~\ref{lem:moment_bounds_A_under},  establishing the moment bounds for Erlang-A model.
\subsection{\blue{Erlang-C Moment Bounds}}
\label{app:momCproof}
\begin{proof}[Proof of Lemma~\ref{lem:moment_bounds_C}]
We first prove \eqref{eq:xsquaredelta}, \eqref{eq:xminusdelta}, and \eqref{eq:xplus}.
Recalling the generator $G_{\tilde X}$ defined in \eqref{eq:GX}, we apply it to the function $V(x) = x^2$ to see that for $k \in \Z_+$ and $x = x_k = \delta(k - x(\infty))$, 
\begin{align}
G_{\tilde X} V(x) =&\ \lambda( 2x\delta + \delta^2) + \mu (k \wedge n)(-2x\delta + \delta^2) \notag \\
=&\ 2x\delta(\lambda - n\mu  + \mu (k - n)^-) + \mu + \delta^2 \mu (k \wedge n) \notag \\
=&\ 2x\mu ( \zeta + (x+\zeta)^-) +\mu  + \delta^2\mu (n - \frac{\lambda}{\mu} + \frac{\lambda}{\mu} - (k - n)^-) \notag \\
=&\ 2x\mu ( \zeta + (x+\zeta)^-) + \mu  -\delta \mu \zeta + \mu  - \delta\mu (x+\zeta)^- \notag \\
=&\ 1(x \leq -\zeta)\mu  \big(-2x^2 + \delta x \big) + 1(x > -\zeta) \mu \big(2x \zeta -\delta \zeta  \big) + 2\mu \notag  \\
\leq&\ 1(x \leq -\zeta) \mu \big(-\frac{3}{2}x^2 + \frac{\delta^2}{2} \big) + 1(x > -\zeta)\mu  \big(2x \zeta -\delta \zeta  \big) + 2\mu. \label{eq:gv1}
\end{align}
Instead of splitting the last two lines into the cases  $x \leq -\zeta$ and $x > -\zeta$, we could have also considered  $x < -\zeta$ and $x \geq -\zeta$ instead, and would have obtained
\begin{align}
G_{\tilde X} V(x)=&\ 1(x < -\zeta)\mu  \big(-2x^2 + \delta x \big) + 1(x \geq -\zeta) \mu \big(2x \zeta -\delta \zeta  \big) + 2\mu  \notag \\
\leq&\ 1(x < -\zeta) \mu \big(-\frac{3}{2}x^2 + \frac{\delta^2}{2} \big) + 1(x \geq -\zeta)\mu  \big(2x \zeta -\delta \zeta  \big) + 2\mu. \label{eq:gv2}
\end{align}
We take expected values on both sides of \eqref{eq:gv1} with respect to $\tilde X(\infty)$, and apply Lemma~\ref{LEM:GZ} to see that
\begin{align}
0 \leq& -\frac{3}{2}\mu \E \big[(\tilde X(\infty))^2 1(\tilde X(\infty) \leq -\zeta)\big] \notag \\
&+ \mu \abs{\zeta} \E \big[\big(-2\tilde X(\infty) +\delta \big)1(\tilde X(\infty) > -\zeta)  \big] + 2\mu  + \frac{\mu \delta^2}{2}.\label{eq:momineq1}
\end{align}
This implies that when $\abs{\zeta} > \delta/2$,
\begin{align*}
0 \leq& -\frac{3}{2}\mu \E \big[(\tilde X(\infty))^2 1(\tilde X(\infty) \leq -\zeta)\big] + 2\mu  + \frac{\mu \delta^2}{2},
\end{align*}
and when $\abs{\zeta} \leq \delta/2$,
\begin{align*}
0 \leq& -\frac{3}{2}\mu \E \big[(\tilde X(\infty))^2 1(\tilde X(\infty) \leq -\zeta)\big]+  2\mu  + \mu \delta^2.
\end{align*}
Therefore, 
\begin{align*}
&\E \big[(\tilde X(\infty))^2 1(\tilde X(\infty) \leq -\zeta)\big] \leq \frac{4}{3} + \frac{2\delta^2}{3}, 
\end{align*}
which proves \eqref{eq:xsquaredelta}. Jensen's inequality immediately gives us
\begin{align*}
\E \Big[\big|\tilde X(\infty)  1(\tilde X(\infty) \leq -\zeta)\big|\Big] \leq \sqrt{\E \big[(\tilde X(\infty))^2 1(\tilde X(\infty) \leq -\zeta)\big]},
\end{align*}
which proves \eqref{eq:xminusdelta}. Furthermore, \eqref{eq:momineq1} also gives us 
\begin{align*}
\E \Big[\big|\tilde X(\infty)1(\tilde X(\infty) > -\zeta)\big| \Big] \leq \frac{1}{\abs{\zeta}} + \frac{\delta^2}{4\abs{\zeta}} + \frac{\delta}{2},
\end{align*}
which is not quite \eqref{eq:xplus} because the inequality above has $1(\tilde X(\infty) > -\zeta)$ as opposed to $1(\tilde X(\infty) \geq -\zeta)$ as in \eqref{eq:xplus}. However, we can use \eqref{eq:gv2} to get the stronger bound 
\begin{align*}
\E \Big[\big|\tilde X(\infty)1(\tilde X(\infty) \geq -\zeta)\big| \Big] \leq \frac{1}{\abs{\zeta}} + \frac{\delta^2}{4\abs{\zeta}} + \frac{\delta}{2},
\end{align*}
which proves \eqref{eq:xplus}.

We now prove \eqref{eq:xminuszeta}, or 
\begin{align} \label{eq:xminusproofpart1}
\E \Big[\big|\tilde X(\infty)  1(\tilde X(\infty) \leq -\zeta)\big|\Big] \leq 2\abs{\zeta}.
\end{align}
We use the triangle inequality to see that
\begin{align*}
\E \Big[\big|\tilde X(\infty)  1(\tilde X(\infty) \leq -\zeta) \big|\Big] \leq&\ \abs{\zeta} +  \E \Big[\big|\tilde X(\infty) +\zeta \big| 1(\tilde X(\infty) \leq -\zeta) \Big].
\end{align*}
The second term on the right hand side is just \blue{the} expected number of idle servers, scaled by $\delta$. We now show that this expected value equals $\abs{\zeta}$. Applying the generator $G_{\tilde X}$ to the test function $f(x) = x$, one sees that for all $k \in \Z_+$ and $x = x_k = \delta(k-x(\infty))$, 
\begin{align*}
G_{\tilde X} f(x) = \delta\lambda - \delta\mu (k \wedge n) = \mu \big[\zeta + (x + \zeta)^-\big].
\end{align*}
Taking expected values with respect to $\tilde X(\infty)$ on both sides, and applying Lemma~\ref{LEM:GZ}, 
we arrive at
\begin{align} \label{eq:idle_expect}
\E \Big[\big|(\tilde X(\infty) +\zeta) 1(\tilde X(\infty) \leq -\zeta)\big| \Big] = \abs{\zeta},
\end{align}
which proves \eqref{eq:xminuszeta}.

We move on to prove \eqref{eq:idle_prob}, or 
\begin{align} \label{eq:inlineidleprob}
\Prob(\tilde X(\infty) \leq -\zeta) \leq (2+\delta)\abs{\zeta}.
\end{align}
Let $I$ be the unscaled expected number of idle servers. Then by \eqref{eq:idle_expect},
\begin{align*}
I = \E(X(\infty) - n)^- =  \frac{1}{\delta}\E \Big[\big|(\tilde X(\infty) +\zeta) 1(\tilde X(\infty) \leq -\zeta)\big| \Big]   = \frac{1}{\delta} \abs{\zeta}.
\end{align*}
Now let $\{\nu_k\}_{k=0}^{\infty}$ be the stationary distribution of $X$ (the unscaled CTMC). We want to prove an upper bound on the probability
\begin{align*}
\Prob(\tilde X(\infty) \leq -\zeta) = \sum_{k=0}^{n} \nu_k \leq \sum_{k=0}^{\lfloor n - \sqrt{R} \rfloor} \nu_k +  \sum_{k= \lceil n - \sqrt{R}\rceil}^{n} \nu_k.
\end{align*}
Observe that 
\begin{align*}
I = \sum_{k=0}^{n} (n-k) \nu_k \geq \sqrt{R}\sum_{k=0}^{\lfloor n - \sqrt{R} \rfloor}  \nu_k.
\end{align*}
Now let $k^*$ be the first index that maximizes $\{\nu_k\}_{k=0}^{\infty}$, i.e.
\begin{align*}
k^* = \inf \{k \geq 0 : \nu_k \geq \nu_j, \text{ for all $j \neq k$}\}.
\end{align*}
Then
\begin{align} 
\Prob(\tilde X(\infty) \leq -\zeta) = \sum_{k=0}^{\lfloor n - \sqrt{R} \rfloor} \nu_k +  \sum_{k= \lceil n - \sqrt{R} \rceil}^{n} \nu_k \leq &\ \frac{I}{\sqrt{R}} +  (\sqrt{R}+1) \nu_{k^*}  \notag\\
=&\ \abs{\zeta} + (\sqrt{R}+1) \nu_{k^*}.\label{eq:prob_interm}
\end{align}
Applying $G_{\tilde X}$ to the test function $f(x) = (k \wedge k^*)$, we see that for all $k \in \Z_+$ and $x = x_k = \delta(k - x(\infty))$,
\begin{align*}
G_{\tilde X} f(x) = \delta \lambda 1(k < k^*) - \delta \mu(k \wedge n) 1(k \leq k^*).
\end{align*}
Taking expected values with respect to $X(\infty)$ on both sides and applying Lemma~\ref{LEM:GZ}, we see that
\begin{align*}
\Prob(X(\infty) \leq k^*) = \frac{\mu}{n\mu-\lambda} \E \big[(X(\infty)-n)^-1(X(\infty) \leq k^*) \big] - \nu_{k^*} \frac{\lambda}{n\mu-\lambda} \geq 0.
\end{align*}
Using the inequality above, together with the fact that $k^* \leq n$, we see that 
\begin{align*}
\nu_{k^*} \leq&\ \frac{\mu}{\lambda} \E \big[(X(\infty)-n)^-1(X(\infty) \leq k^*) \big] \\
\leq&\ \frac{\mu}{\lambda} \E \big[(X(\infty)-n)^-1(X(\infty) \leq n) \big] = \frac{I}{R} = \frac{\abs{\zeta}}{\sqrt{R}}.
\end{align*}
The fact that $k^* \leq n$ is a consequence of $\lambda < n\mu$, and can be verified through the flow balance equations of the CTMC X. We combine the bound above with \eqref{eq:prob_interm} to arrive at \eqref{eq:idle_prob}, which concludes the proof of this lemma.

\end{proof}

\subsection{\blue{Erlang-A Moment Bounds}}
\label{app:mom_A}
The following lemma states the necessary moment bounds for the Erlang-A model. The underloaded and overloaded cases have to be handled separately. Since the drift $b(x)$ is different between the Erlang-A and Erlang-C models, the quantities bounded in the following lemma will resemble those in Lemma~\ref{lem:moment_bounds_C}, but will not be identical.
\begin{lemma}\label{lem:moment_bounds_A_under}
Consider the Erlang-A model ($\alpha>0$). Fix $n \geq 1, \lambda > 0, \\ \mu > 0$, and $\alpha > 0$. If $0 < R \leq n $ (an underloaded system), then
\allowdisplaybreaks
\begin{align}
&\E\Big[ \big(\tilde X(\infty)\big)^2 1(\tilde X(\infty)\leq -\zeta)\Big] \leq \frac{1}{3}\Big(\frac{\alpha}{\mu }\delta^2 + \delta^2 + 4 \Big), \label{eq:mwuK1}\\
&\E\Big[ \big|\tilde X(\infty)1(\tilde X(\infty)\leq -\zeta)\big|\Big] \leq \sqrt{\frac{1}{3}\Big(\frac{\alpha}{\mu }\delta^2 + \delta^2 + 4 \Big)}, \label{eq:mwu1}\\
&\E\Big[ \big|\tilde X(\infty)1(\tilde X(\infty)\leq -\zeta)\big|\Big] \leq 2\abs{\zeta} + \frac{\alpha}{\mu } \sqrt{\frac{1}{3}\Big(\frac{\mu }{\alpha}\delta^2 + \frac{\mu }{\alpha}4 + \delta^2  \Big)},  \label{eq:mwu2}\\
&\E \Big[ \big|\tilde X(\infty)1(\tilde X(\infty)\geq -\zeta)\big|\Big]\notag \\ 
\leq&\ \bigg(1 + \frac{\delta^2}{4} +  \frac{\delta}{2} \sqrt{\frac{1}{3}\Big(\frac{\alpha}{\mu }\delta^2 + \delta^2 + 4 \Big)}\bigg) \Big(\frac{\mu }{\mu \wedge \alpha} \wedge \frac{1}{\abs{\zeta}} \Big), \label{eq:mwu3}\\
&\E \Big[ (\tilde X(\infty)+\zeta)^2 1(\tilde X(\infty)\geq -\zeta)\Big]\leq \frac{1}{3}\Big(\frac{\mu }{\alpha}\delta^2 + \frac{\mu }{\alpha}4 + \delta^2  \Big), \label{eq:mwuK2}\\
&\E \Big[ (\tilde X(\infty)+\zeta)1(\tilde X(\infty)\geq -\zeta)\Big]\leq \sqrt{\frac{1}{3}\Big(\frac{\mu }{\alpha}\delta^2 + \frac{\mu }{\alpha}4 + \delta^2  \Big)}, \label{eq:mwu4}\\
&\E \Big[ (\tilde X(\infty)+\zeta)1(\tilde X(\infty)\geq -\zeta)\Big]\leq \frac{1}{\abs{\zeta}} \Big( \frac{\delta^2}{4}\frac{\alpha}{\mu } + \frac{\delta^2}{4} + 1 \Big), 
\label{eq:mwu5}\\
&\Prob(\tilde X(\infty)\leq -\zeta) \leq (2+\delta)\bigg(\abs{\zeta} + \frac{\alpha}{\mu } \sqrt{\frac{1}{3}\Big(\frac{\mu }{\alpha}\delta^2 + \frac{\mu }{\alpha}4 + \delta^2  \Big)}\bigg). \label{eq:mwu6}
\end{align}
and if $n \leq R$ (an overloaded system), then
\begin{align}
&\E\Big[ \big|\tilde X(\infty)1(\tilde X(\infty)\leq -\zeta)\big|\Big]\leq \sqrt{\frac{1}{\alpha \wedge \mu } \Big( \alpha \frac{\delta^2}{4} + \mu \Big) },
\label{eq:mwo7}\\
&\E\Big[ \big|\tilde X(\infty)1(\tilde X(\infty)\leq -\zeta)\big|\Big]\leq  \frac{1}{\abs{\zeta}} \Big(\frac{\delta^2}{4}+\frac{\mu}{\alpha}\Big),
\label{eq:mwo8}\\
&\E \Big[(\tilde X(\infty))^2 1(\tilde X(\infty) \geq -\zeta)\Big] \leq  \frac{1}{3}\Big(\delta^2 + 4\frac{\mu }{\alpha} \Big), \label{eq:mwo2}\\
&\E\Big[ \big|\tilde X(\infty)1(\tilde X(\infty)\geq -\zeta)\big|\Big] \leq  \sqrt{\frac{1}{3}\Big(\delta^2 + 4\frac{\mu }{\alpha} \Big)}, \label{eq:mwo1}\\
&\E \Big[\big| (\tilde X(\infty)+\zeta)1(\tilde X(\infty)\leq -\zeta)\big|\Big]\leq \frac{1}{\abs{\zeta}}\Big(\frac{\delta^2}{4}+1\Big),
\label{eq:mwo3}\\
&\E \Big[ (\tilde X(\infty)+\zeta)^21(\tilde X(\infty)\leq -\zeta)\Big]\leq  \frac{\delta^2}{4}\frac{\alpha}{\mu }+1, \label{eq:mwoK1}\\
&\E \Big[\big| (\tilde X(\infty)+\zeta)1(\tilde X(\infty)\leq -\zeta)\big|\Big]\leq  \sqrt{ \frac{\delta^2}{4}\frac{\alpha}{\mu }+1}, \label{eq:mwo4}\\
&\E \Big[\big| (\tilde X(\infty)+\zeta)1(\tilde X(\infty)\leq -\zeta)\big|\Big]\leq  \frac{\alpha}{\mu} \sqrt{\frac{1}{3}\Big(\delta^2 + 4\frac{\mu }{\alpha} \Big)},
\label{eq:mwo5}\\
&\Prob(\tilde X(\infty) \leq -\zeta)\leq (3+\delta)\frac{16}{\sqrt 2}\Big(\frac{\delta^2}{4}+1\Big)  \bigg(\Big(\frac{1}{\zeta}\vee \frac{\alpha}{\mu}\Big)\wedge \sqrt{\frac{\alpha}{\mu}}\bigg).\label{eq:mwo10}
\end{align}
\end{lemma}
\subsubsection{Proof Outline for Lemma~\ref{lem:moment_bounds_A_under}: The Underloaded System}\label{app:momboundA_under}
The proof of the underloaded case of Lemma~\ref{lem:moment_bounds_A_under} is very similar to that of Lemma~\ref{lem:moment_bounds_C}. Therefore, we only outline some key intermediate steps needed to obtain the results. We remind the reader that when $R \leq n$, then $\zeta \leq 0$. We first show how to establish \eqref{eq:mwuK1}, which is proved in a similar fashion to \eqref{eq:xsquaredelta} of Lemma~\ref{lem:moment_bounds_C} -- by applying the generator $G_{\tilde X}$ to the Lyapunov function $V(x) = x^2$. The following are some useful intermediate steps for any reader wishing to produce a complete proof. The first step to prove \eqref{eq:mwuK1} is to get an analogue of \eqref{eq:gv1}. Namely, when $x \leq -\zeta$, 
\begin{align*}
G_{\tilde X} V(x) =&\ -2\mu x^2 + \mu\delta x + 2\mu \leq -\frac{3}{2}\mu x^2 + \mu \delta^2/2 + 2\mu ,
\end{align*}
and when $x \geq -\zeta$,
\begin{align}
G_{\tilde X} V(x) =&\ -2\alpha (x+\zeta)^2 + \alpha \delta(x+\zeta)  -2\mu \abs{\zeta}(x+\zeta) \notag  \\
&- 2\abs{\zeta}\alpha(x+\zeta) + \mu\abs{\zeta}(\delta - 2\abs{\zeta}) + 2\mu \notag \\
\leq &\ -\frac{3}{2} \alpha (x+\zeta)^2  -2\mu \abs{\zeta}(x+\zeta) + \delta^2\alpha/2 + \delta^2\mu/8 + 2\mu. \label{eq:genmwu1}
\end{align}
From here, we use Lemma~\ref{LEM:GZ} to get a statement similar to \eqref{eq:momineq1}, from which we can infer \eqref{eq:mwuK1} and by applying Jensen's inequality to \eqref{eq:mwuK1}, we get \eqref{eq:mwu1}. Observe that this procedure yields \eqref{eq:mwuK2}, \eqref{eq:mwu4}, and \eqref{eq:mwu5} as well. We now describe how to prove \eqref{eq:mwu3}, which requires only a slight modification of \eqref{eq:genmwu1}. Namely, for $x \geq -\zeta$,
\begin{align*}
G_{\tilde X} V(x) =&\ 2x\big(-\alpha (x +\zeta) + \mu \zeta \big) - \delta \big(-\alpha (x +\zeta) + \mu \zeta \big) + 2\mu.
\end{align*}
From this, we can deduce that since $x \geq -\zeta$, 
\begin{align*}
G_{\tilde X} V(x) \leq -2(\mu \wedge \alpha) x^2 - \delta \big(-\alpha (x +\zeta) + \mu \zeta \big) + 2\mu,
\end{align*}
and also
\begin{align*}
G_{\tilde X} V(x) \leq -2\mu \abs{\zeta} x - \delta \big(-\alpha (x +\zeta) + \mu \zeta \big) + 2\mu.
\end{align*}
Then Lemma~\ref{LEM:GZ} can be applied as before to see that both 
\begin{align}
2\mu \abs{\zeta}\E \Big[ \big|\tilde X(\infty)1(\tilde X(\infty)\geq -\zeta)\big|\Big] \text{ and }  2(\mu \wedge \alpha)\E \Big[ \big(\tilde X(\infty)\big)^2 1(\tilde X(\infty)\geq -\zeta)\Big] \label{eq:whatbounded}
\end{align}
are bounded by
\begin{align*}
2\mu + \mu \delta^2/2 - \delta \E \Big[ \big(-\alpha (\tilde X(\infty) +\zeta) + \mu \zeta \big)1(\tilde X(\infty) \geq -\zeta)\Big].
\end{align*}
Applying the generator $G_{\tilde X}$ to the test function $f(x) = x$ and taking expected values with respect to $\tilde X(\infty)$, we get $\E b(\tilde X(\infty)) = 0$, or
\begin{align} \label{eq:zerodrift}
\E \Big[ \big(-\alpha (\tilde X(\infty) +\zeta) + \mu \zeta \big)1(\tilde X(\infty) \geq -\zeta)\Big] = \mu \E \Big[\tilde X(\infty) 1(\tilde X(\infty) < -\zeta)\Big].
\end{align}
When combined with \eqref{eq:mwu1}, this implies that 
\begin{align*}
&\ 2\mu + \mu \delta^2/2 - \delta \E \Big[ \big(-\alpha (\tilde X(\infty) +\zeta) + \mu \zeta \big)1(\tilde X(\infty) \geq -\zeta)\Big] \\
\leq&\ 2\mu + \mu \delta^2/2 +\mu \delta\sqrt{\frac{1}{3}\Big(\frac{\alpha}{\mu }\delta^2 + \delta^2 + 4 \Big)},
\end{align*}
which proves \eqref{eq:mwu3}, because the quantity above is an upper bound for \eqref{eq:whatbounded}. To prove \eqref{eq:mwu2}, we manipulate \eqref{eq:zerodrift} to get 
\begin{align*} 
\E \Big[\big| (\tilde X(\infty) + \zeta) 1(\tilde X(\infty) \leq -\zeta)\big|\Big]  = \abs{\zeta} + \frac{\alpha}{\mu } \E \Big[\big| (\tilde X(\infty) + \zeta)  1(\tilde X(\infty) > -\zeta)\big|\Big],
\end{align*}
to which we can apply the triangle inequality and \eqref{eq:mwu4} to conclude \eqref{eq:mwu2}. Lastly, the proof of \eqref{eq:mwu6} is nearly identical to the proof of \eqref{eq:idle_prob} in Lemma~\ref{lem:moment_bounds_C}. The key step is to obtain an analogue of \eqref{eq:prob_interm}. 

\subsubsection{Proof Outline for Lemma~\ref{lem:moment_bounds_A_under}: The Overloaded System}\label{app:momboundA_over}
The proof of the overloaded case of Lemma~\ref{lem:moment_bounds_A_under} is also similar to that of Lemma~\ref{lem:moment_bounds_C}. Therefore, we only outline some key intermediate steps needed to obtain the results; the bounds in this lemma are not proved in the order in which they are stated. We remind the reader that when $R \geq n$, then $\zeta \geq 0$. We start by proving \eqref{eq:mwo2}. Although the left hand side of \eqref{eq:mwo2} is slightly different from \eqref{eq:xsquaredelta} of Lemma~\ref{lem:moment_bounds_C}, it is proved using the same approach -- by applying the generator $G_{\tilde X}$ to the Lyapunov function $V(x) = x^2$. The following are some useful intermediate steps for any reader wishing to produce a complete proof. The first step to prove \eqref{eq:mwo2} is to get analogue of \eqref{eq:gv1}. Namely, when  $x \leq -\zeta$,
\begin{align}
G_{\tilde X} V(x) =&\ -2\mu (x+\zeta)^2 + \mu \delta(x+\zeta) \notag \\
&+ 2(\mu +\alpha)\abs{\zeta}(x+\zeta) - 2\alpha \zeta^2 - \alpha \delta \zeta + 2\mu \notag \\
\leq &\ -2\mu (x+\zeta)^2  + 2(\mu +\alpha)\abs{\zeta}(x+\zeta)  + 2\mu, \label{eq:genmwo1}
\end{align} 
and when $x \geq -\zeta$, 
\begin{align*}
G_{\tilde X} V(x) =&\ -2\alpha x^2 + \alpha\delta x + 2\mu \leq -\frac{3}{2}\alpha x^2 + \alpha \delta^2/2 + 2\mu.
\end{align*}
From here, we use Lemma~\ref{LEM:GZ} to get a statement similar to \eqref{eq:momineq1}, which implies  \eqref{eq:mwo2}. Applying Jensen's inequality to \eqref{eq:mwo2} yields \eqref{eq:mwo1}. The procedure used to get \eqref{eq:mwo2} also yields \eqref{eq:mwo3}, \eqref{eq:mwoK1}, and \eqref{eq:mwo4}. 

We now describe how to prove \eqref{eq:mwo7} and \eqref{eq:mwo8}, which requires only a slight modification of \eqref{eq:genmwo1}. Namely, we use the fact that for $x \leq -\zeta$,
\begin{align*}
G_{\tilde X} V(x) =&\ 2x\big(-\mu (x +\zeta) + \alpha \zeta \big) - \delta \big(-\mu (x +\zeta) + \alpha \zeta \big) + 2\mu.
\end{align*}
From this, one can deduce that since $x \leq -\zeta$, 
\begin{align*}
G_{\tilde X} V(x) \leq -2(\mu \wedge \alpha) x^2+ 2\mu,
\end{align*}
and also
\begin{align*}
G_{\tilde X} V(x) \leq -2\alpha \abs{\zeta} \abs{x}  + 2\mu.
\end{align*}
Then Lemma~\ref{LEM:GZ} and Jensen's inequality can be applied as before to get both \eqref{eq:mwo7} and \eqref{eq:mwo8}.

We now prove \eqref{eq:mwo5}. Observe that 
\begin{align*}
&\E \Big[\big| \tilde X(\infty) 1(\tilde X(\infty) \geq -\zeta) \big| \Big] \\
=&\ \E \Big[\big| (\tilde X(\infty)+\zeta - \zeta )1(\tilde X(\infty) \geq -\zeta) \big| \Big] \\
\geq&\ \E \Big[\big| (\tilde X(\infty)+\zeta)1(\tilde X(\infty) > -\zeta) \big| - \zeta 1(\tilde X(\infty) > -\zeta) \Big] \\
\geq&\ \E \Big[\big| (\tilde X(\infty)+\zeta)1(\tilde X(\infty) > -\zeta) \big| \Big] - \zeta \\
=&\ \frac{\mu}{\alpha} \E \Big[\big| (\tilde X(\infty) + \zeta) 1(\tilde X(\infty) \leq -\zeta)\big|\Big],
\end{align*}
where the last equality comes from applying the generator $G_{\tilde X}$ to the function $f(x) = x$ and taking expected values with respect to $\tilde X(\infty)$ to see that $\E b(\tilde X(\infty)) = 0$, or
\begin{align} \label{eq:zerodrift2}
\E \Big[ \big(-\mu (\tilde X(\infty) +\zeta) + \alpha \zeta \big)1(\tilde X(\infty) \leq -\zeta)\Big] = \alpha \E \Big[\tilde X(\infty) 1(\tilde X(\infty) > -\zeta)\Big].
\end{align}
Therefore, 
\begin{align*}
\E \Big[\big| (\tilde X(\infty) + \zeta) 1(\tilde X(\infty) \leq -\zeta)\big|\Big] 
\leq \frac{\alpha}{\mu } \E \Big[\big| \tilde X(\infty) 1(\tilde X(\infty) \geq -\zeta) \big| \Big],
\end{align*}
and we can invoke  \eqref{eq:mwo1}  to conclude \eqref{eq:mwo5}.

We now prove \eqref{eq:mwo10}, which requires additional arguments that we have not used in the proof of Lemma~\ref{lem:moment_bounds_C}. We assume for now that 
\begin{align} \label{eq:temp_assumption}
\lambda \leq n\mu + \frac{1}{2}\sqrt n \mu.
\end{align} 
Fix $\gamma \in (0, 1/2)$, and define
\begin{align}
\label{eq:j12}
J_1=\sum_{k=0}^{\lfloor n-\gamma \sqrt R \rfloor}\nu_k, \quad J_2=\sum^n_{k=\lceil n-\gamma \sqrt R\rceil }\nu_k,
\end{align}
where $\{\nu_k\}_{k=0}^{\infty}$ is the stationary distribution of $X$. We note that by \eqref{eq:temp_assumption}, 
\begin{align*}
n/\sqrt{R} \geq \sqrt{R}-\frac{1}{2}\sqrt{n/R} \geq \sqrt{R}-1/2 \geq 1/2,
\end{align*}
which implies that $n-\gamma \sqrt R > 0$. Then 
\begin{align*}
\Prob(\tilde X(\infty) \leq -\zeta) = \Prob(X(\infty) \leq n) \leq J_1 + J_2.
\end{align*} 
To bound $J_1$ we observe that 
\begin{align*}
\E\left[\left|\widetilde X(\infty)+\zeta\right|1_{\{\widetilde X(\infty)\leq-\zeta\}}\right] = \frac{1}{\sqrt{R}} \sum_{k=0}^{n} (n-k) \nu_k \geq \gamma\sum_{k=0}^{\lfloor n-\gamma \sqrt R \rfloor}  \nu_k = \gamma J_1.
\end{align*}
Combining \eqref{eq:mwo3}--\eqref{eq:mwo5}, we conclude that
\begin{align}
J_1 \leq&\ \frac{1}{\gamma}\frac{2}{\sqrt 3}\Big(\frac{\delta^2}{4}+1\Big)\Big(\frac{1}{\zeta}\wedge \sqrt{\frac{\alpha}{\mu}\vee 1}\wedge \frac{\alpha}{\mu}
\sqrt{\frac{\mu}{\alpha}\vee 1}\Big)\nonumber \\
\leq&\ \frac{1}{\gamma}\frac{2}{\sqrt 3}\Big(\frac{\delta^2}{4}+1\Big)\Big(\frac{1}{\zeta}\wedge \sqrt{\frac{\alpha}{\mu}}\Big). \label{eq:J1}
\end{align}
Now to bound $J_2$, we apply $G_{\tilde X}$ to the test function $f(x) = k\wedge n$, where $x=\delta (k-x(\infty))$, and take the expectation with respect to $\tilde X(\infty)$ to see that 
\begin{gather*}
0 = -\lambda \nu_n+(\lambda - n\mu )\Prob(X(\infty)\leq n)+\mu \E \left[\left(X(\infty)-n\right)^-1_{\{X(\infty)\leq n\}}\right].
\end{gather*}
Noticing that 
\begin{align*}
\E \left[\left(X(\infty)-n\right)^-1_{\{X(\infty)\leq n\}}\right] = \frac{1}{\delta} \E\left[\left|\widetilde X(\infty)+\zeta\right|1_{\{\widetilde X(\infty)\leq-\zeta\}}\right],
\end{align*}
we arrive at
\begin{align}
\nu_n \leq  \delta \frac{2}{\sqrt 3}\Big(\frac{\delta^2}{4}+1\Big) \left(\frac{1}{\zeta}\wedge \sqrt{\frac{\alpha}{\mu}}\right)+\frac{\lambda -n\mu}{\lambda}\Prob(X(\infty)\leq n).\label{eq:pin}
\end{align}
The flow balance equations
\begin{gather*}
\lambda \nu_{k-1} = k\mu \nu_k,\quad k=1,2,\cdots,n
\end{gather*}
imply that $\nu_0<\nu_1<\cdots<\nu_{n-2}<\nu_{n-1}\leq \nu_n$, and therefore
\begin{align}
&\ J_2 \leq (\gamma \sqrt R + 1)\nu_n \notag  \\
\leq&\ (\gamma \sqrt R + 1)\Big[\delta \frac{2}{\sqrt 3}\Big(\frac{\delta^2}{4}+1\Big) \left(\frac{1}{\zeta}\wedge \sqrt{\frac{\alpha}{\mu}}\right)+\frac{\lambda -n\mu}{\lambda}\Prob(X(\infty)\leq n)\Big] \nonumber\\
=&\ (\gamma +\delta)\frac{2}{\sqrt 3}\Big(\frac{\delta^2}{4}+1\Big) \left(\frac{1}{\zeta}\wedge \sqrt{\frac{\alpha}{\mu}}\right)\notag\\
&+(\gamma \sqrt R + 1)\frac{\lambda -n\mu}{\lambda}J_1 +(\gamma \sqrt R + 1)\frac{\lambda -n\mu}{\lambda}J_2 \label{eq:J2}
\end{align}
We use \eqref{eq:temp_assumption}, the fact that $\gamma \in (0,1/2)$, and that $R \geq n \geq 1$ to see that
\begin{align*}
(\gamma \sqrt R + 1)\frac{\lambda -n\mu}{\lambda} \leq (\gamma \sqrt R + 1)\frac{\sqrt{n}}{2R} \leq \frac{1}{2}(\gamma + 1/\sqrt{R}) = \frac{1}{2}(\gamma + 1)< 3/4.
\end{align*} 
Then by rearranging terms in (\ref{eq:J2}) and applying \eqref{eq:J1} we conclude that
\begin{align*}
\frac{1}{4} J_2 \leq&\  (\gamma +\delta)\frac{2}{\sqrt 3}\Big(\frac{\delta^2}{4}+1\Big) \left(\frac{1}{\zeta}\wedge \sqrt{\frac{\alpha}{\mu}}\right)+\frac{3}{4} \frac{1}{\gamma}\frac{2}{\sqrt 3}\Big(\frac{\delta^2}{4}+1\Big)\Big(\frac{1}{\zeta}\wedge \sqrt{\frac{\alpha}{\mu}}\Big) \\
=&\ \Big(\gamma +\delta+\frac{3}{4} \frac{1}{\gamma} \Big) \frac{2}{\sqrt 3}\Big(\frac{\delta^2}{4}+1\Big) \left(\frac{1}{\zeta}\wedge \sqrt{\frac{\alpha}{\mu}}\right).
\end{align*}
Hence, we have just shown that under assumption \eqref{eq:temp_assumption}, 
\begin{align*}
\Prob(\tilde X(\infty)\leq -\zeta) \leq J_1 + J_2 \leq&\ \frac{1}{\gamma}\frac{2}{\sqrt 3}\Big(\frac{\delta^2}{4}+1\Big)\Big(\frac{1}{\zeta}\wedge \sqrt{\frac{\alpha}{\mu}}\Big) \notag \\
&+ 4\Big(\gamma +\delta+\frac{3}{4} \frac{1}{\gamma} \Big) \frac{2}{\sqrt 3}\Big(\frac{\delta^2}{4}+1\Big) \left(\frac{1}{\zeta}\wedge \sqrt{\frac{\alpha}{\mu}}\right) \notag \\
\leq&\ (3+\delta)\frac{8}{\sqrt 3}\Big(\frac{\delta^2}{4}+1\Big) \left(\frac{1}{\zeta}\wedge \sqrt{\frac{\alpha}{\mu}}\right),
\end{align*}
where to get the last inequality we fixed $\gamma \in (0,1/2)$ that solves $\gamma + 1/\gamma = 3$. 

We now wish to establish the same result without assumption \eqref{eq:temp_assumption}, i.e.\ when $\lambda > n\mu +\frac{1}{2} \sqrt{n}\mu$. For this, we rely on the following comparison result. Fix $n,\mu$ and $\alpha$ and let $X^{(\lambda)}(\infty)$ be the steady-state customer count in an Erlang-A system with arrival rate $\lambda$, service rate $\mu$, number of servers $n$, and abandonment rate $\alpha$. Then for any $0 < \lambda_1 < \lambda_2$, 
\begin{align}
\Prob(X^{(\lambda_2)}(\infty) \leq n) \leq \Prob(X^{(\lambda_1)}(\infty) \leq n). \label{eq:compare}
\end{align}
This says that with all other parameters being held fixed, an Erlang-A system with a higher arrival rate is less likely to have idle servers. \blue{For a simple proof involving a coupling argument, see page 163 of \cite{Lind1992}.}


Therefore, for $\lambda > n\mu +\frac{1}{2} \sqrt{n}\mu$,
\begin{align*}
 &\ \Prob(X^{(\lambda)}(\infty)\leq n) \leq \Prob(X^{(n\mu +\frac{1}{2} \sqrt{n}\mu )}(\infty)\leq n) \\
 \leq&\ (3+\delta)\frac{8}{\sqrt 3}\Big(\frac{\delta^2}{4}+1\Big) \left(\frac{1}{\zeta^{(n\mu +\frac{1}{2} \sqrt{n}\mu )}}\wedge \sqrt{\frac{\alpha}{\mu}}\right)
\end{align*}
where $\zeta^{(n\mu +\frac{1}{2} \sqrt{n}\mu )}$ is the $\zeta$ corresponding to $X^{(n\mu +\frac{1}{2} \sqrt{n}\mu )}(\infty)$, and satisfies 
\begin{align*}
\frac{1}{\zeta^{(n\mu +\frac{1}{2} \sqrt{n}\mu )}} =   \frac{2\alpha}{\mu }\sqrt{\frac{n + \sqrt{n}/2}{n} } \leq  \frac{2\alpha}{\mu }\sqrt{\frac{3}{2}}.
\end{align*}
This concludes the proof of \eqref{eq:mwo10}.

\section{Gradient Bounds}
\label{app:gradbounds}
In Section~\ref{app:wgradient}, we first prove Lemma~\ref{lem:gradboundsCW}, establishing the Wasserstein gradient bounds for Erlang-C model.
In Section~\ref{app:grad_A}, we state and prove Lemma~\ref{lem:gradboundsAWunder},  establishing the Wasserstein gradient bounds for Erlang-A model. In Section~\ref{app:kgradient} we prove Lemmas~\ref{lem:gradboundsCK} and \ref{lem:gradboundsAK}, establishing the Kolmogorov gradient bounds for both Erlang-C and Erlang-A models. 

The arguments below follow the proof of \cite[Lemma 13.1]{ChenGoldShao2011}. \st{For completeness, we repeat many of those arguments in this section. }
We recall the Poisson equation \eqref{eq:poisson}, or
\begin{align*}
G_Y f_h(x) =  b(x) f_h'(x) + \mu f_h''(x)= \E h(Y(\infty)) - h(x), \quad  x\in \R,\ \blue{h(x)} \in \HH.
\end{align*} 
Furthermore, recall that $\nu(x)$ is the density of $Y(\infty)$, and satisfies 
\begin{align*}
\nu(x)= \kappa \exp\Big(\frac{1}{\mu} {\int_0^xb(y)dy}\Big),
\end{align*}
where $\kappa$ is a normalizing constant. Now recall that the family of solutions to the Poisson equation is given by \eqref{eq:poissonsolution}, and is parametrized by constants $a_1, a_2 \in \R$. We fix a solution $f_h(x)$ with $a_2 = 0$, and see that for this solution 
\begin{align}
f_h'(x) =\frac{1}{\nu(x)} \int_{-\infty}^{x} \frac{1}{\mu } \big(\E h(Y(\infty))- h(y) \big) \nu(y) dy, \quad x \in \R. \label{eq:fprime1}
\end{align}
Observe that since $\nu(x)$ is the density of $Y(\infty)$, 
\begin{align*}
\int_{-\infty}^{\infty}\frac{1}{\mu } \big(\E h(Y(\infty))- h(y) \big) \nu(y) dy = 0, 
\end{align*}
which implies that
\begin{align}
f_h'(x) =&\ -\frac{1}{\nu(x)} \int_{x}^{\infty} \frac{1}{\mu }\big(\E h(Y(\infty))- h(y) \big) \nu(y) dy. \label{eq:fprime2}
\end{align}
We will see that to establish gradient bounds, we will have to make use of both expressions for $f_h'(x)$ in \eqref{eq:fprime1} and \eqref{eq:fprime2}, depending on the value of $x$. It is here that the relationship between the diffusion process $Y$ and the random variable $Y(\infty)$ surfaces. If $\E h(Y(\infty))$ were to be replaced by any other constant, then \eqref{eq:fprime2} would not hold. The reason $\E h(Y(\infty))$ is the ``correct" constant is because $Y(\infty)$ has the stationary distribution of the diffusion process $Y$.

 We now proceed to prove the gradient bounds. We do this first in case when $\HH = \lipone$ (the Wasserstein setting), and then when $\HH = \HH_K$ (the Kolmogorov setting).
\subsection{Wasserstein Gradient Bounds} \label{app:wgradient}
Fix $\blue{h(x)} \in \lipone$; without loss of generality we assume that $h(0) = 0$. We now derive some equations that will be useful to prove Lemma~\ref{lem:gradboundsCW}. From the form of $f_h'(x)$ in \eqref{eq:fprime1} and \eqref{eq:fprime2}, we have
\begin{align*}
f_h'(x) \leq&\ \frac{1}{\mu \nu(x)} \int_{-\infty}^{x} \abs{y}\nu(y) dy + \frac{\E \big|Y(\infty)\big|}{\mu \nu(x)} \int_{-\infty}^{x}  \nu(y) dy,\\
f_h'(x) \leq&\ \frac{1}{\mu \nu(x)} \int_{x}^{\infty}\abs{y}\nu(y) dy + \frac{\E \big|Y(\infty)\big|}{\mu \nu(x)} \int_{x}^{\infty}  \nu(y) dy.
\end{align*}
Now by differentiating the Poisson equation \eqref{eq:poisson}, we see that for those $x \in \R$ where $h'(x)$ and $b'(x)$ are defined, 
\begin{align*}
f_h'''(x) = \frac{1}{\mu }\big[-h'(x) - b'(x) f_h'(x) - b(x) f_h''(x)\big].
\end{align*}
 By observing that $b(x) = \mu \nu'(x)/\nu(x)$, we can rearrange the terms above and multiply both sides by $\nu(x)$ to get 
\begin{align}
\frac{\nu(x)}{\mu }(-h'(x) - f_h'(x) b'(x)) = \nu(x) f_h'''(x) + \nu'(x) f_h''(x) = \big( f_h''(x) \nu(x)\big)'. \label{eq:fppder}
\end{align}
Suppose we know that
\begin{align} \label{eq:zerolimits}
\lim_{x \to \pm \infty} \nu(x)f_h''(x)   = 0.
\end{align}
\st{a fact that will be verified separately for each of Lemmas~\ref{lem:gradboundsCW} and \ref{lem:gradboundsAWunder}, right after bounding $f_h'(x)$ in the proofs of those lemmas.} Since $\lim_{x \to -\infty} \nu(x)f_h''(x) = 0$, we integrate \eqref{eq:fppder} to get
\begin{align*}
f_h''(x) =&\ \frac{1}{\nu(x)}\int_{-\infty}^{x} \frac{1}{\mu } (-h'(y) - f_h'(y) b'(y))\nu(y) dy,
\end{align*}
and since $\lim_{x \to \infty} \nu(x)f_h''(x) = 0$, it is also true that
\begin{align*}
f_h''(x) =&\  -\frac{1}{\nu(x)}\int_{x}^{\infty}\frac{1}{\mu }(-h'(y) - f_h'(y) b'(y))\nu(y) dy.
\end{align*}
To summarize, we have just shown that provided \eqref{eq:zerolimits} holds,
\begin{align}
\abs{f_h'(x)} \leq&\ \frac{1}{\mu \nu(x)} \int_{-\infty}^{x} \abs{y}\nu(y) dy + \frac{\E \big|Y(\infty)\big|}{\mu \nu(x)} \int_{-\infty}^{x}  \nu(y) dy,\label{eq:der11}\\
\abs{f_h'(x)} \leq&\ \frac{1}{\mu \nu(x)} \int_{x}^{\infty}\abs{y}\nu(y) dy + \frac{\E \big|Y(\infty)\big|}{\mu \nu(x)} \int_{x}^{\infty}  \nu(y) dy,\label{eq:der12}\\
f_h''(x) =&\ \frac{1}{\nu(x)} \int_{-\infty}^{x} \frac{1}{\mu }(-h'(y) - f_h'(y)b'(y)) \nu(y) dy \label{eq:der21} \\
=&\ -\frac{1}{\nu(x)} \int_{x}^{\infty} \frac{1}{\mu }(-h'(y) - f_h'(y)b'(y)) \nu(y) dy,  \label{eq:der22} \\
f_h'''(x) =&\ \frac{1}{\mu }\big[-h'(x) - f_h''(x)b(x) - f_h'(x) b'(x)\big] \label{eq:der3},
\end{align}
where $f_h'''(x)$ is defined for all $x \in \R$ such that $h'(x)$ and $b'(x)$ exist. The multiple bounds for $f_h'(x)$ are convenient, because we can choose which one to use based on the value of $x$. For example, suppose $\zeta \leq 0$. Then when $x \leq 0$, we will use \eqref{eq:der11}, when $x \geq -\zeta$ we will use \eqref{eq:der12}, and when $x \in [0,-\zeta]$ we will use both \eqref{eq:der11} and \eqref{eq:der12} and choose the better bound. The same applies to  $f_h''(x)$. 
\subsubsection{Proof of Lemma~\ref{lem:gradboundsCW}}
The following lemma presents several bounds that will be used to prove Lemma~\ref{lem:gradboundsCW}. We prove it at the end of this section.
\begin{lemma}
\label{lem:lowlevelCW}
Consider the Erlang-C model ($\alpha = 0$), and let $\nu(x)$ be the density of $Y(\infty)$. Then
\allowdisplaybreaks
\begin{align}
&\frac{1}{\nu(x)} \int_{-\infty}^{x} \nu(y) dy \leq 
\begin{cases}
\sqrt{\frac{\pi}{2}}, \quad x \leq 0, \label{eq:fbound1} \\
\sqrt{2\pi}e^{\frac{1}{2}\zeta^2}, \quad x \in [0,-\zeta],
\end{cases}\\
&\frac{1}{\nu(x)}\int_{x}^{\infty} \nu(y) dy \leq 
\begin{cases}
\sqrt{\frac{\pi}{2}} + \frac{1}{\abs{\zeta}}, \quad x \in [0,-\zeta], \\
\frac{1}{\abs{\zeta}}, \quad x \geq -\zeta,
\end{cases} \label{eq:fbound2} \\
&\frac{1}{\nu(x)} \int_{-\infty}^{x} \abs{y}\nu(y) dy \leq 
\begin{cases}
1, \quad x \leq 0,\\
2e^{\frac{1}{2}\zeta^2} -1, \quad x \in [0,-\zeta],
\end{cases} \label{eq:fbound3} \\
&\frac{1}{\nu(x)}\int_{x}^{\infty} \abs{y}\nu(y) dy \leq 
\begin{cases}
2 + \frac{1}{\zeta^2}, \quad x \in [0,-\zeta],\\
\frac{x}{\abs{\zeta}} + \frac{1}{\zeta^2}, \quad x \geq -\zeta \label{eq:fbound4},
\end{cases}\\
&\frac{\abs{b(x)}}{\mu \nu(x)} \int_{-\infty}^{x} \nu(y) dy \leq 1 ,\quad x \leq 0, \label{eq:fbound5} \\
&\frac{\abs{b(x)}}{\mu \nu(x)}\int_{x}^{\infty} \nu(y) dy \leq 2, \quad x \geq 0, \label{eq:fbound6} \\
&\E \abs{Y(\infty)} \leq \frac{1}{\abs{\zeta}} + 1. \label{eq:fbound7}
\end{align}
\end{lemma}

\begin{proof}[Proof of Lemma~\ref{lem:gradboundsCW}]
We begin by bounding $f_h'(x)$ by applying Lemma~\ref{lem:lowlevelCW} to \eqref{eq:der11} and \eqref{eq:der12}. For $x \leq -\zeta$, we apply \eqref{eq:fbound1}, \eqref{eq:fbound3}, and \eqref{eq:fbound7} to \eqref{eq:der11}, and for $x \geq 0$, we apply \eqref{eq:fbound2}, \eqref{eq:fbound4}, and \eqref{eq:fbound7} to \eqref{eq:der12} to see that
\begin{align}
\mu \abs{f_h'(x)} \leq&\ 1 + \sqrt{\frac{\pi}{2}} \Big( 1 + \frac{1}{\abs{\zeta}} \Big) \leq 2.3 + 1.3/\abs{\zeta}, \quad x \leq 0, \notag \\
\mu \abs{f_h'(x)} \leq&\ \min\bigg\{ 2e^{\frac{1}{2}\zeta^2} -1 + \sqrt{2\pi}e^{\frac{1}{2}\zeta^2}\Big( 1 + \frac{1}{\abs{\zeta}}\Big), \notag \\
&\ \hspace{1.5cm} 2 + \frac{1}{\zeta^2} + \Big(\sqrt{\frac{\pi}{2}} + \frac{1}{\abs{\zeta}} \Big)\Big( 1 + \frac{1}{\abs{\zeta}} \Big) \bigg\}, \quad x \in [0,-\zeta], \notag \\
\mu \abs{f_h'(x)} \leq&\ \frac{x}{\abs{\zeta}} + \frac{1}{\zeta^2} + \frac{1}{\abs{\zeta}} \Big( 1 + \frac{1}{\abs{\zeta}} \Big) \leq \frac{1}{\abs{\zeta}}(x + 1 + 2/\abs{\zeta}), \quad x \geq -\zeta. \label{eq:cgrad1}
\end{align}
For $x \in [0, -\zeta]$, observe that when $\abs{\zeta} \leq 1$, then 
\begin{align*}
2e^{\frac{1}{2}\zeta^2} -1 + \sqrt{2\pi}e^{\frac{1}{2}\zeta^2}\Big( 1 + \frac{1}{\abs{\zeta}}\Big) \leq 3.3 - 1 + 4.2\Big( 1 + \frac{1}{\abs{\zeta}}\Big) = 6.5 + 4.2/\abs{\zeta},
\end{align*}
and when $\abs{\zeta} \geq 1$, then 
\begin{align*}
2 + \frac{1}{\zeta^2} + \Big(\sqrt{\frac{\pi}{2}} + \frac{1}{\abs{\zeta}} \Big)\Big( 1 + \frac{1}{\abs{\zeta}} \Big) \leq 3 + (1.3 + 1)\Big( 1 + \frac{1}{\abs{\zeta}} \Big) = 5.3 + 2.3/\abs{\zeta}.
\end{align*}
Therefore, 
\begin{align} \label{eq:dsimple}
\abs{f_h'(x)} \leq  
\begin{cases}
\frac{1}{\mu }(6.5 + 4.2/\abs{\zeta}), \quad x \leq -\zeta,\\
\frac{1}{\mu }\frac{1}{\abs{\zeta}}(x + 1 + 2/\abs{\zeta}), \quad x \geq -\zeta.
\end{cases}
\end{align}
Before proceeding to bound $\abs{f_h''(x)}$ and $\abs{f_h'''(x)}$, we first verify \eqref{eq:zerolimits}, or 
\begin{align*}
\lim_{x \to \pm \infty} \nu(x) f_h''(x) = 0.
\end{align*}
To do so, we rearrange the Poisson equation \eqref{eq:poisson} to get
\begin{align*}
\nu(x) f_h''(x) = -\nu(x) \frac{b(x)}{\mu } f_h'(x) + \frac{\nu(x)}{\mu }\big[ \E h(Y(\infty)) - h(x) \big].
\end{align*} 
It is now obvious that 
\begin{align*}
\lim_{x \to \pm \infty} \nu(x)f_h''(x) \to 0 = 0,
\end{align*}
because $\blue{h(x)} \in \lipone$, the drift $b(x)$ is piecewise linear, and $f_h'(x)$ is bounded as in \eqref{eq:dsimple}, but on the other hand $\nu(x)$ decays exponentially fast as $x \to \infty$, and decays even faster as $x \to -\infty$. To bound $\abs{f_h''(x)}$, we use \eqref{eq:der21} and \eqref{eq:der22}, together with the facts that $\norm{h'} \leq 1$  and
\begin{align*}
b'(x) = -\mu 1(x < -\zeta), \quad x \in \R,
\end{align*}
to see that 
\begin{align}
\abs{f_h''(x)} \leq&\ \frac{1}{\nu(x)} \int_{-\infty}^{x} \frac{1}{\mu }\big(1 + \mu \abs{f_h'(y)} 1(y < -\zeta)\big) \nu(y) dy \label{eq:d21bound}\\
\abs{f_h''(x)} \leq&\ \frac{1}{\nu(x)} \int_{x}^{\infty} \frac{1}{\mu }\big(1 + \mu \abs{f_h'(y)} 1(y < -\zeta)\big) \nu(y) dy \label{eq:d22bound}.
\end{align}
We know $\abs{f_h'(x)}$ is bounded as in \eqref{eq:dsimple}. For $x \leq -\zeta$, we apply \eqref{eq:fbound1} to \eqref{eq:d21bound} and for $x \geq 0$ we apply \eqref{eq:fbound2} to \eqref{eq:d22bound} to conclude that
\begin{align} \label{eq:der2_final}
\mu \abs{f_h''(x)} \leq 
\begin{cases}
\sqrt{\frac{\pi}{2}} \big( 1 + 6.5 + 4.2/\abs{\zeta}\big), \quad x \leq 0,\\
\min \big\{\sqrt{2\pi}e^{\frac{1}{2}\zeta^2}, \sqrt{\frac{\pi}{2}} + \frac{1}{\abs{\zeta}}\big\} \big( 1 + 6.5 + 4.2/\abs{\zeta}\big), \quad x \in [0,-\zeta],\\
\frac{1}{\abs{\zeta}}, \quad x \geq -\zeta.
\end{cases}
\end{align}
By considering separately the cases when $\abs{\zeta} \leq 1$ and $\abs{\zeta} \geq 1$, we see that
\begin{align}
\min \Big\{\sqrt{2\pi}e^{\frac{1}{2}\zeta^2}, \sqrt{\frac{\pi}{2}} + \frac{1}{\abs{\zeta}}\Big\} \leq 4.2, \label{eq:arithmetic}
\end{align}
and therefore,
\begin{align} \label{eq:der2_simple_bound}
\abs{f_h''(x)} \leq 
\begin{cases}
\frac{32}{\mu }( 1 + 1/\abs{\zeta}), \quad x \leq -\zeta,\\
\frac{1}{\mu \abs{\zeta}}, \quad x \geq -\zeta.
\end{cases}
\end{align}
Lastly, we bound $\abs{f_h'''(x)}$, which exists for all $x \in \R$ where $h'(x)$ and $b'(x)$ exist. The fact that $\blue{h(x)} \in \lipone$ together with \eqref{eq:der3} tells us that 
\begin{align*}
\abs{f_h'''(x)} \leq&\ \frac{1}{\mu } \big(1 + \abs{f_h''(x)b(x)} + \abs{f_h'(x) b'(x)}\big).
\end{align*} 
For $x \geq -\zeta$, we use the forms of $b(x)$ and $b'(x)$ together with the bounds on $\abs{f_h'(x)}$ and $\abs{f_h''(x)}$ in \eqref{eq:dsimple} and \eqref{eq:der2_simple_bound} to see that 
\begin{align*}
\abs{f_h'''(x)} \leq&\ \frac{1}{\mu } \Big( 1 + \mu \abs{\zeta} \frac{1}{\mu \abs{\zeta}}\Big).
\end{align*}
Although tempting, it is not sufficient to use the bound on $\abs{f_h''(x)}$ in \eqref{eq:der2_final} and the form of $b(x)$  to bound $\abs{f_h''(x)b(x)}$ for all $x \leq -\zeta$. Instead, we multiply both sides of \eqref{eq:d21bound} and \eqref{eq:d22bound} by $\abs{b(x)}$ to see that
\begin{align}
\norm{f_h''(x) b(x)} \leq&\ \big(1 + \sup_{y \leq x} \mu \abs{f_h'(y)}\big)\frac{\abs{b(x)}}{\mu \nu(x)} \int_{-\infty}^{x} \nu(y) dy, \quad x \leq 0, \notag \\
\norm{f_h''(x) b(x)} \leq&\ \big(1 + \sup_{y \in [x,-\zeta]} \mu \abs{f_h'(y)}\big)\frac{\abs{b(x)}}{\mu \nu(x)} \int_{x}^{\infty} \nu(y) dy, \quad x \in [0,-\zeta]. \label{eq:cgradfb}
\end{align}
By invoking \eqref{eq:fbound5} and \eqref{eq:fbound6} together with the bound on $\abs{f_h'(x)}$ from \eqref{eq:dsimple}, we conclude that
\begin{align*}
\abs{f_h'''(x)}\leq&\ \frac{1}{\mu }\big(  1 + 2(1 +  6.5 + 4.2/\abs{\zeta}) + 6.5 + 4.2/\abs{\zeta}\big), \quad x \leq -\zeta.
\end{align*}
Therefore, for those $x \in \R$ where $h'(x)$ and $b'(x)$ exist,
\begin{align*}
\abs{f_h'''(x)} \leq
\begin{cases}
\frac{1}{\mu }( 23 + 13/\abs{\zeta}) , \quad x \leq -\zeta,\\
\frac{2}{\mu }, \quad x \geq -\zeta.
\end{cases}
\end{align*}
This concludes the proof of Lemma~\ref{lem:gradboundsCW}.
\end{proof}
\begin{proof}[Proof of Lemma~\ref{lem:lowlevelCW} ]
Recall that in the Erlang-C model, we set $\alpha = 0$, which makes
\begin{align*}
\zeta = \delta(x(\infty) - n) = \delta(\frac{\lambda}{\mu} - n) < 0, \quad 
b(x) = 
\begin{cases}
-\mu x, \quad x \leq -\zeta,\\
\mu \zeta, \quad x \geq -\zeta.
\end{cases}
\end{align*} 
The density of $Y(\infty)$ has the form
\begin{align}\label{eq:pic}
\nu(x) = 
\begin{cases}
a_{-} e^{-\frac{1}{2}x^2}, \quad x \leq - \zeta,\\
a_{+} e^{-\abs{\zeta} x}, \quad x \geq -\zeta,
\end{cases} 
\end{align}
where $a_{-}$ and $a_{+}$ are normalizing constants that make $\nu(x)$ continuous at the point $x = -\zeta$.

In the proof of this lemma we often rely on the fact that for any $c > 0$ and  $x\geq 0$,
\begin{equation}
\label{eq:usefulbound}
e^{\frac{1}{2}cx^2}\int_{x}^{\infty}{ e^{-\frac{1}{2} cy^2} dy} \leq \int_{0}^{\infty}{ e^{-\frac{1}{2} cy^2} dy} = \sqrt{\frac{\pi}{2c}}.
\end{equation}
One can verify that the left hand side of (\ref{eq:usefulbound}) peaks at $x=0$ by using the bound
\begin{equation} \label{eq:normcdfbound}
\int_{x}^{\infty}{ e^{-\frac{1}{2} cy^2} dy} \leq \int_{x}^{\infty}{\frac{cy}{cx}e^{-\frac{1}{2} cy^2} dy} = \frac{e^{-cx^2/2}}{cx}
\end{equation}
to see that the derivative of the left side of (\ref{eq:usefulbound}) is negative for $x>0$. 
We now prove \eqref{eq:fbound1}. When $x \leq 0$, we use \eqref{eq:usefulbound} and the symmetry of the function $e^{-\frac{1}{2}y^2}$ to see that
\begin{align*}
\frac{1}{\nu(x)} \int_{-\infty}^{x} \nu(y) = e^{\frac{1}{2}x^2} \int_{-\infty}^{x} e^{-\frac{1}{2}y^2} dy \leq \sqrt{\frac{\pi}{2}},
\end{align*}
and when $x \in [0,-\zeta]$, we use \eqref{eq:usefulbound} again to get
\begin{align*}
\frac{1}{\nu(x)} \int_{-\infty}^{x} \nu(y) =&\ e^{\frac{1}{2}x^2} \int_{-\infty}^{0} e^{-\frac{1}{2}y^2} dy + e^{\frac{1}{2}x^2} \int_{0}^{x} e^{-\frac{1}{2}y^2} dy \\
\leq&\ 2e^{\frac{1}{2}\zeta^2}\int_{-\infty}^{0} e^{-\frac{1}{2}y^2} dy = 2\sqrt{\frac{\pi}{2}}e^{\frac{1}{2}\zeta^2}.
\end{align*}
We now prove \eqref{eq:fbound2}. Observe that when $x \geq -\zeta$,
\begin{align*}
\frac{1}{\nu(x)}\int_{x}^{\infty} \nu(y) dy = e^{\abs{\zeta}x} \int_{x}^{\infty} e^{-\abs{\zeta} y} dy = \frac{1}{\abs{\zeta}},
\end{align*}
and when $x \in [0,-\zeta]$, we use the fact that $a_- = a_+ e^{-\frac{1}{2}\zeta^2}$ together with \eqref{eq:usefulbound} to see that
\begin{align*}
\frac{1}{\nu(x)}\int_{x}^{\infty} \nu(y) dy =&\ \frac{a_+}{a_-} e^{\frac{1}{2}x^2} \int_{-\zeta}^{\infty} e^{-\abs{\zeta} y} dy +  e^{\frac{1}{2}x^2} \int_{x}^{-\zeta} e^{-\frac{1}{2}y^2} dy\\
\leq&\ e^{\frac{1}{2}\zeta^2} e^{\frac{1}{2}x^2} \Big( \frac{1}{\abs{\zeta}} e^{-\zeta^2} \Big) + \sqrt{\frac{\pi}{2}}\\
\leq&\ \frac{1}{\abs{\zeta}} + \sqrt{\frac{\pi}{2}}.
\end{align*}
Moving on to show \eqref{eq:fbound3}, when $x \leq 0$ we have 
\begin{align*}
\frac{1}{\nu(x)}\int_{-\infty}^{x} \abs{y}\nu(y) dy =&\ e^{\frac{1}{2}x^2} \int_{-\infty}^{x}-y e^{-\frac{1}{2}y^2} dy = 1,
\end{align*}
and when $x \in [0,-\zeta]$,
\begin{align*}
\frac{1}{\nu(x)}\int_{-\infty}^{x} \abs{y}\nu(y) dy =&\ e^{\frac{1}{2}x^2} \int_{-\infty}^{0}-y e^{-\frac{1}{2}y^2} dy + e^{\frac{1}{2}x^2} \int_{0}^{x}y e^{-\frac{1}{2}y^2} dy\\
=&\ e^{\frac{1}{2}x^2} + e^{\frac{1}{2}x^2}(1 - e^{-\frac{1}{2}x^2}).
\end{align*}
We now prove \eqref{eq:fbound4}. When $x \in [0,-\zeta]$, we again use that $a_- = a_+ e^{-\frac{1}{2}\zeta^2}$ to obtain
\begin{align*}
\frac{1}{\nu(x)}\int_{x}^{\infty} \abs{y}\nu(y) dy =&\ e^{\frac{1}{2}x^2} \int_{x}^{-\zeta}y e^{-\frac{1}{2}y^2} dy + \frac{a_+}{a_-} e^{\frac{1}{2}x^2} \int_{-\zeta}^{\infty}y e^{-\abs{\zeta} y} dy\\
=&\ e^{\frac{1}{2}x^2}(e^{-\frac{1}{2}x^2} - e^{-\frac{1}{2}\zeta^2}) + e^{\frac{1}{2}\zeta^2}e^{\frac{1}{2}x^2}\Big(1 + \frac{1}{\zeta^2} \Big)e^{-\zeta^2}\\
\leq&\ 2+ \frac{1}{\zeta^2}.
\end{align*}
When $x \geq -\zeta$, 
\begin{align*}
\frac{1}{\nu(x)}\int_{x}^{\infty} \abs{y}\nu(y) dy = e^{\abs{\zeta} x} \int_{x}^{\infty}y e^{-\abs{\zeta} y} dy  = \frac{x}{\abs{\zeta}} + \frac{1}{\zeta^2}.
\end{align*}
We move on to prove \eqref{eq:fbound5} and \eqref{eq:fbound6}. When $x \leq 0$, we use \eqref{eq:normcdfbound} to see that 
\begin{align*}
\frac{\abs{b(x)}}{\mu \nu(x)} \int_{-\infty}^{x} \nu(y) dy = -x e^{\frac{x^2}{2}} \int_{-\infty}^{x} e^{-\frac{1}{2}y^2}dy \leq -x e^{\frac{x^2}{2}} \Big(-\frac{1}{x}e^{-\frac{1}{2}x^2} \Big) = 1.
\end{align*}
When $x \in [0,-\zeta]$, we also use the fact that $a_+= a_- e^{\frac{1}{2}\zeta^2}$ to get
\begin{align*}
\frac{\abs{b(x)}}{\mu \nu(x)} \int_{x}^{\infty} \nu(y) dy =&\ x e^{\frac{x^2}{2}} \int_{x}^{-\zeta} e^{-\frac{1}{2}y^2}dy + \frac{a_+}{a_-} x e^{\frac{x^2}{2}} \int_{-\zeta}^{\infty} e^{-\abs{\zeta} y}dy\\
\leq&\ x e^{\frac{x^2}{2}} \Big(\frac{1}{x}e^{-\frac{1}{2}x^2} \Big) + e^{\frac{1}{2}\zeta^2}x e^{\frac{1}{2}x^2}\Big(\frac{1}{\abs{\zeta}}e^{-\zeta^2} \Big)\\
=&\ 1 + \frac{x}{\abs{\zeta}}e^{\frac{1}{2}(x^2-\zeta^2)} \leq 2.
\end{align*}
When $x \geq -\zeta$, 
\begin{align*}
\frac{\abs{b(x)}}{\mu \nu(x)} \int_{x}^{\infty} \nu(y) dy =&\ \abs{\zeta} e^{\abs{\zeta}x} \int_{x}^{\infty} e^{-\abs{\zeta} y}dy = 1.
\end{align*}
Lastly we prove \eqref{eq:fbound7}. Letting $V(x) = x^2$, and recalling the form of $G_Y$ from \eqref{eq:GY}, we consider 
\begin{align}
G_Y V(x) =&\ 2x\mu (\zeta + (x + \zeta)^-) + 2\mu \notag \\
=&\ -2\mu x^2 1(x < -\zeta) - 2x\mu \abs{\zeta} 1(x \geq -\zeta) + 2\mu. \label{eq:cgradlyap}
\end{align}
By the standard Foster-Lyapunov condition (see  \cite[Theorem 4.3]{MeynTwee1993b} for example), this implies that
\begin{align*}
2\E \big[ (Y(\infty))^2 1(Y(\infty) < -\zeta)\big] + 2\abs{\zeta} \E \big[Y(\infty)1(Y(\infty) \geq -\zeta)\big] \leq 2,
\end{align*}
and in particular, 
\begin{align*}
&\E \big[Y(\infty)1(Y(\infty) \geq -\zeta)\big] \leq \frac{1}{\abs{\zeta}}, \\
& \E \blue{\Big[\big|}Y(\infty)1(Y(\infty) < -\zeta)\blue{\big|\Big]} \leq  \sqrt{\E \big[ (Y(\infty))^2 1(Y(\infty) < -\zeta)\big]} \leq 1,
\end{align*}
where we applied Jensen's inequality in the second set of inequalities. This concludes the proof of Lemma~\ref{lem:lowlevelCW}. 
\end{proof}

\subsection{\blue{Erlang-A Wasserstein Gradient Bounds}} \label{app:grad_A}
Below we state the Erlang-A gradient bounds, the proof of which is similar to that of Lemma~\ref{lem:gradboundsCW}. We only outline the necessary steps needed for a proof, and emphasize all the differences with the proof of Lemma~\ref{lem:gradboundsCW}.\st{; this is done in Appendix~\ref{app:gradAWunder} for the underloaded system and Appendix~\ref{app:gradAWover} for the overloaded system.} Furthermore, we do not seek an explicit value for the constant $C$ below, although it can certainly be recovered. 

\begin{lemma} \label{lem:gradboundsAWunder} 
Consider the Erlang-A model ($\alpha > 0$). The solution to the Poisson equation $f_h(x)$ is twice continuously differentiable, with an absolutely continuous second derivative. Fix a solution in \eqref{eq:poissonsolution} with parameter $a_2 = 0$. Then there exists a constant $C > 0$ independent of $\lambda, n, \mu$, and $\alpha$ such that for all $n \geq 1, \lambda > 0, \mu>0$, and $\alpha > 0$ satisfying $0 < R \leq n$ (an underloaded system), 
\begin{align}
\abs{f_h'(x)}\leq&\ 
\begin{cases}
C \left(\sqrt{\frac{\mu}{\alpha}}\wedge \frac{1}{\abs{\zeta}}+1\right)\frac{1}{\mu},\quad x\leq -\zeta,\\
C \left(\frac{\mu}{\alpha}+\sqrt{\frac{\mu}{\alpha}}\wedge \frac{1}{\abs{\zeta}}+1\right)\frac{1}{\mu},\quad x\geq -\zeta,
\end{cases} \label{eq:gwu1} \\
\abs{f_h''(x)} \leq&\
\begin{cases}
C \left(\sqrt{\frac{\mu}{\alpha}}\wedge \frac{1}{\abs{\zeta}}+1\right)\frac{1}{\mu}, \quad x\leq 0, \\
C \left[\left(\frac{\alpha}{\mu}+\sqrt{\frac{\alpha}{\mu}}+1\right)\left(\sqrt{\frac{\mu}{\alpha}}\wedge \frac{1}{\abs{\zeta}}\right)+1\right]\frac{1}{\mu},\quad x\in [0,-\zeta],\\
C\left(\frac{\alpha}{\mu}+\sqrt{\frac{\alpha}{\mu}}+1\right)\left(\sqrt{\frac{\mu}{\alpha}}\wedge \frac{1}{\abs{\zeta}}\right)\frac{1}{\mu},\quad x\geq -\zeta,
\end{cases} \label{eq:gwu2} \\
\end{align}
and for those $x \in \R$ where $f_h'''(x)$ exists,
\begin{align}
\abs{f_h'''(x)} \leq&\ 
\begin{cases}
C\left(\sqrt{\frac{\mu}{\alpha}}\wedge \frac{1}{\abs{\zeta}}+1\right)\frac{1}{\mu},
\quad x\leq 0,\\
C\left(\sqrt{\frac{\mu}{\alpha}}\wedge \frac{1}{\abs{\zeta}}+\frac{\alpha}{\mu}+\sqrt{\frac{\alpha}{\mu}}+1\right)\frac{1}{\mu}, \quad x\in [0,-\zeta], \\
C \left(\frac{\alpha}{\mu}+\sqrt{\frac{\alpha}{\mu}}+1\right)\frac{1}{\mu},\quad x\geq -\zeta,
\end{cases} \label{eq:gwu3}
\end{align}
and for all $n \geq 1, \lambda > 0, \mu>0$, and $\alpha > 0$ satisfying $n \leq R$ (an overloaded system),
\begin{align}
\abs{f_h'(x)}\leq&\
\begin{cases}
C \Big(\frac{1}{\mu}+ \frac{1}{\sqrt \alpha}\frac{1}{\sqrt\mu} + \frac{\zeta}{\mu} \wedge \frac{1}{\alpha}\Big), \quad x\leq-\zeta,\\
C \Big(\frac{1}{\mu}+\frac{1}{\sqrt\alpha}\frac{1}{\sqrt\mu}+\frac{1}{\alpha}\Big), \quad x\geq -\zeta,
\end{cases} \label{eq:gwo1} \\
\abs{f_h''(x)}\leq &\
\begin{cases}
C\Big(\frac{1}{\mu}+ \frac{1}{\sqrt \alpha}\frac{1}{\sqrt\mu} + \frac{\zeta}{\mu} \wedge \frac{1}{\alpha}\Big),\quad x\leq-\zeta,\\
 C\Big(\frac{\alpha}{\mu}+\sqrt{\frac{\alpha}{\mu}}+1\Big)\frac{1}{\mu}|x| +
C\Big(\frac{1}{\mu}+\frac{1}{\sqrt \alpha}\frac{1}{\sqrt\mu}\Big),\quad x\geq -\zeta,
\end{cases} \label{eq:gwo2}\\
\end{align}
and for those $x \in \R$ where $f_h'''(x)$ exists,
\begin{align}
\abs{f_h'''(x)}\leq&\ \frac{C}{\mu}\Big(1+\sqrt{\frac{\mu}{\alpha}}+\zeta \wedge \frac{\mu}{\alpha}\Big),\quad x\leq -\zeta,\label{eq:gwo3}\\
\abs{f_h'''(x)}\leq&\ \frac{C}{\mu}\Big(\frac{\alpha}{\mu}+\sqrt{\frac{\alpha}{\mu}}+1\Big)
\Big(1+\frac{\alpha}{\mu}x^2\Big)
+\frac{C}{\mu}\Big(\frac{\alpha}{\mu}+\sqrt{\frac{\alpha}{\mu}}\Big)\abs{x},\quad x\geq -\zeta,\label{eq:gwo41}\\
\abs{f_h'''(x)} \leq&\
\frac{C}{\mu}\Big(\frac{\alpha}{\mu}+\sqrt{\frac{\alpha}{\mu}}+1\Big)
+ \frac{C}{\mu}\Big(\frac{\alpha}{\mu}+\sqrt{\frac{\alpha}{\mu}}+1\Big)^2 \abs{x} ,\quad x\geq -\zeta. \label{eq:gwo42}
\end{align}

\end{lemma}

\subsubsection{Proof Outline for Lemma~\ref{lem:gradboundsAWunder}: The Underloaded System} \label{app:gradAWunder}
\st{In this section we prove the underloaded case of Lemma~\ref{lem:gradboundsAWunder}. The proof is very similar to that of Lemma~\ref{lem:gradboundsCW}. Therefore, we only outline some key intermediate steps needed to obtain the results.} To prove Lemma~\ref{lem:gradboundsAWunder}, we need the following version of Lemma~\ref{lem:lowlevelCW}.

\begin{lemma} \label{lem:lowlevWunder}
Consider the Erlang-A model ($\alpha > 0$) with $0 < R \leq n$, and let $\nu(x)$ be the density of $Y(\infty)$. Then
\allowdisplaybreaks
\begin{align}
&\frac{1}{\nu(x)} \int_{-\infty}^{x} \nu(y) dy \leq
\begin{cases}
\sqrt{\frac{\pi}{2}}, \quad x \leq 0 , \\
\sqrt{2\pi}e^{\frac{\zeta^2}{2}}, \quad x \in [0,-\zeta],
\end{cases}\label{eq:ingredient1}\\
&\frac{1}{\nu(x)}\int_{x}^{\infty} \nu(y) dy \leq
\begin{cases}
\sqrt{\frac{\pi}{2}}+\sqrt{\frac{\pi}{2}\frac{\mu}{\alpha}}\wedge \frac{1}{\abs{\zeta}}, \quad x \in [0,-\zeta], \\
\sqrt{\frac{\pi}{2}\frac{\mu}{\alpha}}\wedge \frac{1}{\abs{\zeta}}, \quad x \geq -\zeta,
\end{cases} \label{eq:ingredient2} \\
&\frac{1}{\nu(x)} \int_{-\infty}^{x} \abs{y}\nu(y) dy \leq
\begin{cases}
1, \quad x \leq 0,\\
2e^{\frac{\zeta^2}{2}}-1, \quad x \in [0,-\zeta],
\end{cases} \label{eq:ingredient3} \\
&\frac{1}{\nu(x)}\int_{x}^{\infty} \abs{y}\nu(y) dy \leq
\begin{cases}
2+\frac{1}{\zeta^2}, \quad x \in [0,-\zeta],\\
1+\frac{\mu}{\alpha}, \quad x \geq -\zeta \label{eq:ingredient4},
\end{cases}\\
& \frac{\abs{b(x)}}{\mu \nu(x)}\int^x_{-\infty}\nu(y) dy \leq
1, \quad x \leq 0, \label{eq:ingredient6}\\
&\frac{\abs{b(x)}}{\mu \nu(x)}\int^\infty_x \nu(y) dy \leq 2, \quad x \geq 0, \label{eq:ingredient7}\\
&\E \abs{Y(\infty)} \leq 1+\sqrt{\frac{\mu}{\alpha}}\wedge \frac{1}{\abs{\zeta}} . \label{eq:ingredient5}
\end{align}
\end{lemma}
To prove this lemma, we first observe that 
\begin{align*}
b(x) = 
\begin{cases}
-\mu x, \quad x \leq -\zeta,\\
-\alpha(x+\zeta)+ \mu \zeta, \quad x \geq -\zeta,
\end{cases}
\end{align*}
and
\begin{align}
\nu(x) =
\begin{cases}
a_- e^{-\frac{1}{2}x^2},\quad x\leq -\zeta,\\
 a_+e^{-\frac{\alpha}{2\mu}\left(x+\zeta-\frac{\mu}{\alpha}\zeta\right)^2},\quad x\geq -\zeta,
\end{cases} \label{eq:densAunder}
\end{align}
where $a_-$ and $a_+$ are normalizing constants that make $\nu(x)$ continuous and integrate to one.
By comparing the density in \eqref{eq:densAunder} to \eqref{eq:pic} for the region $x \leq -\zeta$, we immediately see that \eqref{eq:ingredient1}, \eqref{eq:ingredient3} and \eqref{eq:ingredient6} are restatements of \eqref{eq:fbound1}, \eqref{eq:fbound3}, and \eqref{eq:fbound5} from Lemma~\ref{lem:lowlevelCW}, and hence have already been established. The proof of \eqref{eq:ingredient5} involves applying $G_Y$ to the Lyapunov function $V(x) = x^2$ to see that
\begin{align*}
G_Y V(x) =&\ -2\mu x^2 1(x < -\zeta) + 2\big(-\alpha x^2 + x\zeta(\mu - \alpha) \big)1(x \geq -\zeta) + 2\mu \\
\leq&\ -2\mu x^2 1(x < -\zeta) -2 (\alpha \wedge\mu ) x^2 1(x \geq -\zeta) + 2\mu,  
\end{align*}
and 
\begin{align*}
G_Y V(x) =&\ -2\mu x^2 1(x < -\zeta) + 2\big(-\alpha x(x+\zeta) - \mu \abs{\zeta}x \big)1(x \geq -\zeta) + 2\mu \\
\leq&\ -2\mu x^2 1(x < -\zeta) -2\mu \abs{\zeta} x1(x \geq -\zeta) + 2\mu.
\end{align*}
One can compare these inequalities to \eqref{eq:cgradlyap} in the proof of Lemma~\ref{lem:lowlevelCW} to see that \eqref{eq:ingredient5} follows by the Foster-Lyapunov condition.

We now go over the proofs of \eqref{eq:ingredient2}, \eqref{eq:ingredient4} and \eqref{eq:ingredient7}. We first prove \eqref{eq:ingredient2} when $x \in [0,-\zeta]$. Observe that 
\begin{align}
\frac{1}{\nu(x)}\int_{x}^{\infty} \nu(y) dy=&\ e^{\frac{1}{2}x^2} \int_x^{\abs{\zeta}} e^{-\frac{1}{2}y^2}  dy
+ \frac{a_+}{a_-} e^{\frac{1}{2}x^2}\int^\infty_{\abs{\zeta}}e^{-\frac{\alpha}{2\mu}\left(y+\zeta-\frac{\mu}{\alpha}\zeta\right)^2}dy \notag \\
\leq&\ e^{\frac{1}{2}x^2} \int_x^{\infty} e^{-\frac{1}{2}y^2}  dy + e^{\frac{1}{2}(x^2-\zeta^2)}e^{\frac{\alpha}{2\mu}\left(\frac{\mu}{\alpha}\zeta\right)^2}\int^\infty_{\frac{\mu}{\alpha}\abs{\zeta}}e^{-\frac{\alpha}{2\mu}y^2}dy \notag  \\
\leq&\ \sqrt{\frac{\pi}{2}} + \sqrt{\frac{\pi}{2}\frac{\mu}{\alpha}}\wedge \frac{1}{\abs{\zeta}} \label{eq:undergrad1}
\end{align}
where in the first inequality we used a change of variables and the fact that $a_+/a_- = e^{-\zeta^2/2} e^{\frac{\alpha}{2\mu } (\frac{\mu }{\alpha}\zeta)^2}$, and in the last inequality we used both \eqref{eq:usefulbound} and \eqref{eq:normcdfbound}, and the fact that $e^{\frac{1}{2}(x^2-\zeta^2)} \leq 1$. The rest of \eqref{eq:ingredient2} is proved identically. We now prove \eqref{eq:ingredient4}. When $x \in [0,-\zeta]$, 
\begin{align*}
&\ \frac{1}{\nu(x)}\int_{x}^{\infty} \abs{y}\nu(y) dy \\
=&\  e^{\frac{1}{2}x^2} \int_{x}^{-\zeta} y e^{-\frac{1}{2}y^2}dy + \frac{a_+}{a_-}  e^{\frac{1}{2}x^2}\int_{-\zeta}^{\infty} ye^{-\frac{\alpha}{2\mu}(y+\zeta-\frac{\mu}{\alpha}\zeta)^2} dy \\
=&\ (1 - e^{\frac{1}{2}(x^2-\zeta^2)}) +  e^{\frac{\alpha}{2\mu } (\frac{\mu }{\alpha}\zeta)^2} e^{\frac{1}{2}(x^2-\zeta^2)}\int_{-\zeta}^{\infty} ye^{-\frac{\alpha}{2\mu}(y+\zeta-\frac{\mu}{\alpha}\zeta)^2} dy \\
\leq&\ 1 + e^{\frac{\alpha}{2\mu } (\frac{\mu }{\alpha}\zeta)^2}\int_{-\zeta}^{\infty} ye^{-\frac{\alpha}{2\mu} \big( (y+\zeta)^2 -2\frac{\mu }{\alpha} (y+\zeta)\zeta + (\frac{\mu }{\alpha}\zeta)^2\big)} dy \\
\leq&\ 1 + \int_{-\zeta}^{\infty} ye^{(y+\zeta)\zeta} dy = 1 + 1 + \frac{1}{\zeta^2},
\end{align*}
and when $x \geq -\zeta$,
\begin{align*}
\frac{1}{\nu(x)}\int_{x}^{\infty} \abs{y}\nu(y) dy =&\  e^{\frac{\alpha}{2\mu}(x+\zeta-\frac{\mu}{\alpha}\zeta)^2}\int_{x}^{\infty} ye^{-\frac{\alpha}{2\mu}(y+\zeta-\frac{\mu}{\alpha}\zeta)^2} dy \\
=&\  e^{\frac{\alpha}{2\mu}(x+\zeta-\frac{\mu}{\alpha}\zeta)^2}\int_{x+\zeta-\frac{\mu}{\alpha}\zeta}^{\infty} ye^{-\frac{\alpha}{2\mu}y^2} dy \\
&+ (1-\mu/\alpha)\abs{\zeta} e^{\frac{\alpha}{2\mu}(x+\zeta-\frac{\mu}{\alpha}\zeta)^2}\int_{x+\zeta-\frac{\mu}{\alpha}\zeta}^{\infty} e^{-\frac{\alpha}{2\mu}y^2} dy \\
\leq&\ \frac{\mu }{\alpha} + \abs{\zeta} \frac{1}{\frac{\alpha}{\mu } (x + \zeta - \frac{\mu }{\alpha}\zeta)}  \leq \frac{\mu }{\alpha} + 1,
\end{align*}
where in the last inequality we used \eqref{eq:normcdfbound}. Lastly, the argument for \eqref{eq:ingredient7}  is very similar to the chain of inequalities in \eqref{eq:undergrad1}, and we leave the details to the reader.

We now describe how to prove Lemma~\ref{lem:gradboundsAWunder}. To prove \eqref{eq:gwu1}, we apply Lemma~\ref{lem:lowlevWunder} to \eqref{eq:der11} and \eqref{eq:der12} just like in \eqref{eq:cgrad1} of Lemma~\ref{lem:gradboundsCW}. Using these bounds on $f_h'(x)$, we argue \eqref{eq:zerolimits} just like in the proof of Lemma~\ref{lem:gradboundsCW}. We now describe how to prove \eqref{eq:gwu2}. When $x \leq 0$,  we apply \eqref{eq:gwu1} and \eqref{eq:ingredient1} to \eqref{eq:der21}, and when $x \geq -\zeta$ we apply \eqref{eq:gwu1} and \eqref{eq:ingredient2} to \eqref{eq:der22}. The last region, when $x \in [0,-\zeta]$, has to be handled differently depending on the size of $\abs{\zeta}$. When $\abs{\zeta} \leq 1$, we just apply \eqref{eq:gwu1} and \eqref{eq:ingredient1} to \eqref{eq:der21}. However, when $\abs{\zeta} \geq 1$, we manipulate \eqref{eq:der22} to see that
\begin{align}
f_h''(x) =&\ -\frac{1}{\nu(x)} \int_{x}^{-\zeta} \frac{1}{\mu }(-h'(y) + \mu f_h'(y)) \nu(y) dy   \\
&-\frac{\nu(-\zeta)}{\nu(x)} \frac{1}{\nu(-\zeta)} \int_{-\zeta}^{\infty} \frac{1}{\mu }(-h'(y) + \alpha f_h'(y)) \nu(y) dy. \label{eq:der2manip}
\end{align}
We then apply \eqref{eq:gwu1}, \eqref{eq:ingredient2}, and the fact that $\nu(-\zeta)/\nu(x) \leq 1$ to conclude \eqref{eq:gwu2}. The proof of \eqref{eq:gwu3} relies on \eqref{eq:der3}, which tells us that 
\begin{align*}
\abs{f_h'''(x)} \leq&\ \frac{1}{\mu } \big[ 1 + \abs{f_h''(x)b(x)} + \abs{f_h'(x) b'(x)}\big].
\end{align*}
Bounding $\abs{f_h'(x) b'(x)}$  only relies on \eqref{eq:gwu1}. The term $\abs{f_h''(x)b(x)}$ is bounded similarly to the way it is done in Lemma~\ref{lem:gradboundsCW}; see for instance \eqref{eq:cgradfb}.  Namely, for $x \leq 0$ we multiply both sides of \eqref{eq:der21} by $b(x)$ and apply \eqref{eq:gwu1} with \eqref{eq:ingredient6}, and when $x \geq -\zeta$ we multiply  both sides of \eqref{eq:der22} by $b(x)$ and apply \eqref{eq:gwu1} with \eqref{eq:ingredient7}. Lastly, when $x \in [0,-\zeta]$, we manipulate \eqref{eq:der22} to get
\begin{align*}
b(x)f_h''(x) =&\ -\frac{b(x)}{\nu(x)} \int_{x}^{-\zeta} \frac{1}{\mu }(-h'(y) + \mu f_h'(y)) \nu(y) dy   \\
&-\frac{b(x)\nu(-\zeta)}{\nu(x)b(-\zeta)} \frac{b(-\zeta)}{\nu(-\zeta)} \int_{-\zeta}^{\infty} \frac{1}{\mu }(-h'(y) + \alpha f_h'(y)) \nu(y) dy.
\end{align*}
We then apply \eqref{eq:gwu1}, \eqref{eq:ingredient7}, and the fact that $\big|\frac{b(x)\nu(-\zeta)}{\nu(x)b(-\zeta)}\big| \leq 1$ to get the required bounds on $\abs{b(x)f_h''(x)}$ and conclude \eqref{eq:gwu3}. This concludes the proof outline for Lemma~\ref{lem:gradboundsAWunder}.

\subsubsection{Proof Outline for Lemma~\ref{lem:gradboundsAWunder}: The Overloaded System} \label{app:gradAWover}
\st{In this section we prove the overloaded case of Lemma~\ref{lem:gradboundsAWunder}. The proof is similar to that of Lemma~\ref{lem:gradboundsCW}. Therefore, we only outline some key intermediate steps needed to obtain the results. } For the overloaded case in Lemma~\ref{lem:gradboundsAWunder}, we again need the following version of Lemma~\ref{lem:lowlevelCW}.

\begin{lemma} \label{lem:lowlevWover}
Consider the Erlang-A model ($\alpha > 0$) with $R \geq n$, and let $\nu(x)$ be the density of $Y(\infty)$. Then
\allowdisplaybreaks
\begin{align}
&\frac{1}{\nu(x)} \int_{-\infty}^{x} \nu(y) dy \leq 
\begin{cases}
\sqrt{\frac{\pi}{2}}\wedge \frac{\mu}{\alpha \zeta}, \quad x \leq -\zeta , \\
\sqrt{\frac{\pi}{2}}+\sqrt{\frac{\pi}{2}\frac{\mu}{\alpha}}\wedge \zeta, \quad x \in [-\zeta,0],
\end{cases}\label{eq:oingredient1}\\
&\frac{1}{\nu(x)}\int_{x}^{\infty} \nu(y) dy \leq
\begin{cases}
\sqrt{2\pi\frac{\mu}{\alpha}} e^{\frac{\alpha}{2\mu}\zeta^2}, \quad x \in [-\zeta,0], \\
\sqrt{\frac{\pi}{2}\frac{\mu}{\alpha}}, \quad x \geq 0,
\end{cases} \label{eq:oingredient2} \\
&\frac{1}{\nu(x)} \int_{-\infty}^{x} \abs{y}\nu(y) dy \leq
\begin{cases}
1 + \sqrt{\frac{\pi}{2}}\zeta \wedge \frac{\mu}{\alpha}, \quad x \leq -\zeta,\\
\frac{\mu}{\alpha}+1, \quad x \in [-\zeta,0],
\end{cases} \label{eq:oingredient3} \\
&\frac{1}{\nu(x)}\int_{x}^{\infty} \abs{y}\nu(y) dy \leq
\begin{cases}
2\frac{\mu}{\alpha}e^{\frac{\alpha}{2\mu}\zeta^2}, \quad x \in [-\zeta,0],\\
\frac{\mu}{\alpha}, \quad x \geq 0 \label{eq:oingredient4},
\end{cases}\\
&\frac{\abs{b(x)}}{\mu \nu(x)}\int^x_{-\infty}\nu(y) dy \leq 2, \quad x\leq 0,
\label{eq:oingredient6} \\
&\frac{\abs{b(x)}}{\mu \nu(x)}\int^\infty_x \nu(y) dy \leq 1,\quad x\geq 0, \label{eq:oingredient7}\\
&\E \abs{Y(\infty)} \leq  \sqrt{\frac{\mu}{\alpha}} + 1. \label{eq:oingredient5}
\end{align}
\end{lemma}

To prove this lemma, we first observe that
\begin{align*}
b(x) = 
\begin{cases}
-\mu (x+\zeta)+ \alpha\zeta, \quad x \leq -\zeta,\\
-\alpha x, \quad x \geq -\zeta,
\end{cases}
\end{align*}
and 
\begin{align}
\nu(x) =
\begin{cases}
a_- e^{-\frac{1}{2}\left(x+\zeta-\frac{\alpha}{\mu}\zeta\right)^2},\quad x\leq -\zeta,\\
 a_+e^{-\frac{\alpha}{2\mu}x^2},\quad x\geq -\zeta,
\end{cases} \label{eq:densAover}
\end{align}
where $a_-$ and $a_+$ are normalizing constants that make $\nu(x)$ continuous and integrate to one. Observe that  in the region $x \geq -\zeta$, the density in \eqref{eq:densAover} looks very similar to the density in \eqref{eq:pic} in the region $x \leq -\zeta$. Hence, one can check that the arguments needed to prove Lemma~\ref{lem:lowlevWover}'s \eqref{eq:oingredient2}, \eqref{eq:oingredient4}, and \eqref{eq:oingredient7}  are nearly identical to the arguments used to prove Lemma~\ref{lem:lowlevelCW}'s \eqref{eq:fbound1}, \eqref{eq:fbound3}, and \eqref{eq:fbound5}, respectively.

The proof of \eqref{eq:oingredient5} involves applying $G_Y$ to the Lyapunov function $V(x) = x^2$ to see that
\begin{align*}
G_Y V(x) =&\ -2\alpha x^2 1(x > -\zeta) + 2\big(-\mu x^2 + x\zeta(\alpha - \mu ) \big)1(x \leq -\zeta) + 2\mu \\
\leq&\ -2\alpha x^2 1(x > -\zeta) -2 (\alpha \wedge\mu ) x^2 1(x \leq -\zeta) + 2\mu.  
\end{align*}
One can compare this inequality to \eqref{eq:cgradlyap} in the proof of Lemma~\ref{lem:lowlevelCW} to see that \eqref{eq:oingredient5} follows by the Foster-Lyapunov condition. 

We now describe how to prove \eqref{eq:oingredient1}, \eqref{eq:oingredient3}, and \eqref{eq:oingredient6}. To prove \eqref{eq:oingredient1} we use a series of arguments similar to those in \eqref{eq:undergrad1}, where we proved \eqref{eq:ingredient2} of Lemma~\ref{lem:lowlevWunder}. We now prove \eqref{eq:oingredient3}. When $x \leq -\zeta$, 
\begin{align}
\frac{1}{\nu(x)} \int_{-\infty}^{x} \abs{y}\nu(y) dy =&\   e^{\frac{1}{2}(x+\zeta-\frac{\alpha}{\mu}\zeta)^2}\int_{-\infty}^{x} -ye^{-\frac{1}{2}(y+\zeta-\frac{\alpha}{\mu}\zeta)^2} dy  \notag \\
=&\  e^{\frac{1}{2}(x+\zeta-\frac{\alpha}{\mu}\zeta)^2}\int_{-\infty}^{x+\zeta-\frac{\alpha}{\mu}\zeta} -ye^{-\frac{1}{2}y^2} dy \notag \\
&+ (1-\alpha/\mu )\zeta e^{\frac{1}{2}(x+\zeta-\frac{\alpha}{\mu}\zeta)^2}\int_{-\infty}^{x+\zeta-\frac{\alpha}{\mu}\zeta} e^{-\frac{1}{2}y^2} dy \notag  \\
\leq&\ 1 + \zeta \Big( \sqrt{\frac{\pi}{2}} \wedge  \frac{1}{ \frac{\alpha}{\mu }\zeta -x - \zeta } \Big) \leq 1 +  \sqrt{\frac{\pi}{2}}\zeta  \wedge \frac{\mu }{\alpha}, \label{eq:overgrad1}
\end{align}
where in the last inequality we used both \eqref{eq:usefulbound} and \eqref{eq:normcdfbound}. For $x \in [-\zeta, 0]$, 
\begin{align*}
&\ \frac{1}{\nu(x)} \int_{-\infty}^{x} \abs{y}\nu(y) dy \\
=&\   \frac{a_-}{a_+} e^{\frac{\alpha}{2\mu }x^2}\int_{-\infty}^{-\zeta} -ye^{-\frac{1}{2}(y+\zeta-\frac{\alpha}{\mu}\zeta)^2} dy + e^{\frac{\alpha}{2\mu }x^2}\int_{-\zeta}^{x} -ye^{\frac{\alpha}{2\mu }y^2} dy.
\end{align*}
Repeating arguments from \eqref{eq:overgrad1} and using  $a_-/a_+ = e^{-\frac{\alpha}{2\mu } \zeta^2} e^{\frac{1}{2} (\frac{\alpha}{\mu }\zeta)^2}$, we can show that the first term above satisfies 
\begin{align*}
\frac{a_-}{a_+} e^{\frac{\alpha}{2\mu }x^2}\int_{-\infty}^{-\zeta} -ye^{-\frac{1}{2}(y+\zeta-\frac{\alpha}{\mu}\zeta)^2} dy \leq e^{\frac{\alpha}{2\mu } (x^2-\zeta^2)}\Big(1 + \frac{\mu }{\alpha}\Big),
\end{align*}
and by computing the second term explicitly, we conclude that 
\begin{align*}
\frac{1}{\nu(x)} \int_{-\infty}^{x} \abs{y}\nu(y) dy \leq&\ e^{\frac{\alpha}{2\mu } (x^2-\zeta^2)}\Big(1 + \frac{\mu }{\alpha}\Big) + \frac{\mu }{\alpha} \big( 1 - e^{\frac{\alpha}{2\mu } (x^2-\zeta^2)}\big) \\
\leq&\ 1 + \frac{\mu }{\alpha},
\end{align*}
which proves \eqref{eq:oingredient3}. Lastly, it is not hard to see that \eqref{eq:oingredient6} follows from a straightforward application of \eqref{eq:normcdfbound}.

Having argued Lemma~\ref{lem:lowlevWover}, we now use it to prove the bounds in \eqref{eq:gwo1}--\eqref{eq:gwo42}. To prove \eqref{eq:gwo1}, we apply Lemma~\ref{lem:lowlevWover} to \eqref{eq:der11} and \eqref{eq:der12} just like in \eqref{eq:cgrad1} of Lemma~\ref{lem:gradboundsCW}. Using these bounds on $f_h'(x)$, we argue \eqref{eq:zerolimits} just like in the proof of Lemma~\ref{lem:gradboundsCW}. We now describe how to prove \eqref{eq:gwo2}. When $x \leq -\zeta$,  we apply \eqref{eq:gwo1} and \eqref{eq:oingredient1} to \eqref{eq:der21}.  When $x \geq -\zeta$, instead of using the expressions for $f_h''(x)$ in \eqref{eq:der21} and \eqref{eq:der22} like we would usually do, we instead apply \eqref{eq:gwo1} to the bound 
\begin{align*}
\abs{f_h''(x)} \leq \frac{1}{\mu }\abs{f_h'(x)} \abs{b(x)} + \frac{1}{\mu }\big(\abs{x} + \E \abs{Y(\infty)} \big), \quad x \in \R,
\end{align*}
which follows by rewriting the Poisson equation \eqref{eq:poisson} and using the Lipschitz property of $h(x)$. We now prove \eqref{eq:gwo3}--\eqref{eq:gwo42}. We recall \eqref{eq:der3} to see that
\begin{align*}
\abs{f_h'''(x)} \leq&\ \frac{1}{\mu } \big[ 1 + \abs{f_h''(x)b(x)} + \abs{f_h'(x) b'(x)}\big].
\end{align*}
Bounding $\abs{f_h'(x) b'(x)}$ is simple, and only relies on \eqref{eq:gwo1}. The other term, $\abs{f_h''(x)b(x)}$, is bounded as follows. To prove \eqref{eq:gwo3}, i.e.\ when $x \leq -\zeta$, we multiply both sides of \eqref{eq:der21} by $b(x)$, and apply \eqref{eq:gwo1} and \eqref{eq:oingredient6} to the result. When $x \geq -\zeta$ then 
\begin{align*}
\abs{f_h''(x)b(x)} = \alpha \abs{x}\abs{f_h''(x)},
\end{align*}
and the difference between \eqref{eq:gwo41} and \eqref{eq:gwo42} lies in the way that the quantity above is bounded. To get \eqref{eq:gwo41}, we simply apply the bounds on $f_h''(x)$ from \eqref{eq:gwo2} to the right hand side above. 

To prove \eqref{eq:gwo42}, we will first argue that
\begin{align}
\abs{f_h'''(x)} \leq &\ 
\begin{cases}
\frac{C}{\mu}\Big(\frac{\alpha}{\mu}+\sqrt{\frac{\alpha}{\mu}}+1\Big)
+ \frac{C}{\mu}\Big(\frac{\alpha}{\mu}+\sqrt{\frac{\alpha}{\mu}}+1\Big)^2 \abs{x},\quad x\in [-\zeta,0],\\
\frac{C}{\mu}\Big(\frac{\alpha}{\mu}+\sqrt{\frac{\alpha}{\mu}}+1\Big),\quad x\geq 0,
\end{cases} \label{eq:altern_third}
\end{align}
where $C$ is some positive constant independent of everything else; this will imply \eqref{eq:gwo42}. The only difference between the proof of \eqref{eq:altern_third} and the  bound on $f_h'''(x)$ in \eqref{eq:gwo41} is in how $\abs{f_h''(x)b(x)}$ is bounded; we now describe the different way to bound $\abs{f_h''(x)b(x)}$. When $x \geq 0$, we multiply both sides of \eqref{eq:der22} by $b(x)$ and use the bounds in \eqref{eq:gwo1} and \eqref{eq:oingredient7} to bound $\abs{f_h''(x) b(x)}$. When $x \in [-\zeta,0]$, we want to prove that
\begin{align}
\abs{f_h''(x)} \leq&\ \frac{C}{\mu}\Big(\frac{\alpha}{\mu}+\sqrt{\frac{\alpha}{\mu}}+1\Big)\Big(1+\sqrt{\frac{\mu}{\alpha}}\Big) + \frac{C}{\mu } \Big( \zeta \wedge \frac{\mu^2}{\alpha^2	\zeta}\Big), \label{eq:alternativesgb}
\end{align}
which, after considering separately the cases when $\zeta \leq \mu /\alpha$ and $\zeta \geq \mu /\alpha$, implies that 
\begin{align*}
\abs{f_h''(x)} \leq&\ \frac{C}{\mu}\Big(\frac{\alpha}{\mu}+\sqrt{\frac{\alpha}{\mu}}+1\Big)\Big(1+\sqrt{\frac{\mu}{\alpha}}\Big)
+ \frac{C}{\alpha}.
\end{align*}
We can then use this fact to bound $\abs{f_h''(x)b(x)} =  \alpha \abs{x}\abs{f_h''(x)}$. To prove \eqref{eq:alternativesgb} for $\zeta \leq \sqrt{\mu/\alpha}$, we bound \eqref{eq:der22} using \eqref{eq:gwo1} and \eqref{eq:oingredient2}. To prove \eqref{eq:alternativesgb} for $\zeta \geq \sqrt{\mu/\alpha}$, we bound \eqref{eq:der21} using \eqref{eq:gwo1} and \eqref{eq:oingredient1}. We point out that to bound \eqref{eq:der21} we need to perform a manipulation similar to the one in \eqref{eq:der2manip}.  This concludes the proof outline for the overloaded case.

\subsection{Kolmogorov Gradient Bounds: Proof of Lemmas~\ref{lem:gradboundsCK} and \ref{lem:gradboundsAK}} \label{app:kgradient}
For the remainder of this section, we take $\mathcal{H} = \mathcal{H}_K$ in \eqref{eq:poisson}. With this choice of test functions, any solution to the Poisson equation will have a discontinuity in its second derivative, which makes the gradient bounds for it differ from the Wasserstein setting. Fix $a \in \R$ and consider the Poisson equation 
\begin{align*}
G_Y f_a(x) = b(x) f_a'(x) + \mu f_a''(x) = F_Y(a) - 1_{(-\infty,a]}(x),
\end{align*}
where $F_Y(x)$ is the distribution function of $Y(\infty)$. If $f_a(x)$ is a solution the Poisson equation with $a_2 = 0$, then just as in \eqref{eq:fprime1} and \eqref{eq:fprime2},
\begin{align*}
f_a'(x) =&\ \frac{1}{\mu \nu(x)} \int_{-\infty}^{x} \big( F_Y(a) - 1_{(-\infty,a]}(y)\big) \nu(y) dy, \notag \\
f_a'(x)=&\ -\frac{1}{\mu \nu(x)} \int_{x}^{\infty} \big(F_Y(a) - 1_{(-\infty,a]}(y)\big) \nu(y) dy,
\end{align*}
which immediately implies that
\begin{align*}
\abs{f_a'(x)} \leq \frac{1}{\mu \nu(x)} \min \Big\{ \int_{-\infty}^{x} \nu(y) dy, \int_{x}^{\infty} \nu(y) dy \Big\}.
\end{align*}
Furthermore, 
\begin{align*}
f_a''(x) = \frac{1}{\mu } \big[ F_Y(a) - 1_{(-\infty,a]}(x) - b(x) f_a'(x)\big].
\end{align*}
We now prove the Kolmogorov gradient bounds for the Erlang-C model.
\begin{proof}[Proof of Lemma~\ref{lem:gradboundsCK}]
First of all, by \eqref{eq:fbound1} and \eqref{eq:fbound2}, 
\begin{align}
\mu \abs{f_a'(x)} \leq 
\begin{cases}
\sqrt{\frac{\pi}{2}}, \quad x \leq 0, \\
\min \big\{\sqrt{2\pi}e^{\frac{1}{2}\zeta^2},\sqrt{\frac{\pi}{2}} + \frac{1}{\abs{\zeta}} \big\}, \quad x \in [0,-\zeta], \\
\frac{1}{\abs{\zeta}}, \quad x \geq -\zeta,
\end{cases} \label{eq:kolmfp}
\end{align}
and \eqref{eq:arithmetic} implies that
\begin{align*}
\min\Big\{\sqrt{2\pi}e^{\frac{1}{2}\zeta^2},\sqrt{\frac{\pi}{2}} + \frac{1}{\abs{\zeta}} \Big\} \leq 5,
\end{align*}
which proves the bounds for $f_a'(x)$. Second, \eqref{eq:fbound5} and \eqref{eq:fbound6} imply that for all $x \in \R$, 
\begin{align}
\abs{f_a''(x)} \leq  \frac{1}{\mu } \bigg[ 1 + \frac{\abs{b(x)}}{\mu \nu(x)} \min \Big\{ \int_{-\infty}^{x} \nu(y) dy, \int_{x}^{\infty} \nu(y) dy \Big\}\bigg] \leq 3/\mu, \label{eq:kolmfpp}
\end{align}
where $f_a''(x)$ is understood to be the left derivative at the point $x = a$.

\end{proof}

\begin{proof}[Proof of Lemma~\ref{lem:gradboundsAK}]
The proof of this lemma is almost identical to the proof of Lemma~\ref{lem:gradboundsCK}. By using the analogues of \eqref{eq:fbound5} and \eqref{eq:fbound6} from Lemmas~\ref{lem:lowlevWunder} and \ref{lem:lowlevWover}, its not hard to check that \eqref{eq:kolmfpp} holds for the Erlang-A model as well. To prove the bounds on $f_a'(x)$, we obtain inequalities similar to \eqref{eq:kolmfp} by using analogues of \eqref{eq:fbound1} and \eqref{eq:fbound2} from Lemmas~\ref{lem:lowlevWunder} and \ref{lem:lowlevWover}.  These inequalities will imply \eqref{eq:ACuder1} and \eqref{eq:ACoder1} once we consider in them separately the cases when $\abs{\zeta} \leq 1$ and $\abs{\zeta} \geq 1$.

\end{proof}

\section{Proof Outlines of Erlang-A Theorems} \label{app:EAproofs}
Sections~\ref{app:AWoutline} and \ref{app:AKoutline} contain an outline for the proofs of Theorems~\ref{thm:erlangAW} and \ref{thm:erlangAK}, respectively.

\subsection{Outline for Theorem~\ref{thm:erlangAW}} \label{app:AWoutline}
Proving Theorem~\ref{thm:erlangAW} consists of bounding the four error terms in \eqref{eq:first_bounds}. Since the procedure is very similar to the proof of Theorem~\ref{thm:erlangCW}, we will only outline which gradient and moment bounds need to be used to bound each error term.

We start with the underloaded case, when $R \leq n$. To bound the first term in \eqref{eq:first_bounds}, we use moment bounds \eqref{eq:mwu1}, \eqref{eq:mwu2}, and \eqref{eq:mwu4}, together with the gradient bounds in \eqref{eq:gwu2}. For the second and third terms, we use moment bound \eqref{eq:mwu6} and the gradient bounds in \eqref{eq:gwu3}. For the fourth term, we use moment bounds \eqref{eq:mwu1}--\eqref{eq:mwu4}, and the gradient bounds in \eqref{eq:gwu3}.

We now prove the overloaded case, when $R \geq n$. To bound the first term in \eqref{eq:first_bounds}, we use moment bounds \eqref{eq:mwo7}--\eqref{eq:mwo5}, together with the gradient bounds in \eqref{eq:gwo2}. For the second and third terms, we use moment bounds \eqref{eq:mwo2},\eqref{eq:mwo1}, and \eqref{eq:mwo10}, together with the gradient bounds in \eqref{eq:gwo3} and \eqref{eq:gwo41}. For the fourth term, we use moment bounds \eqref{eq:mwo7}--\eqref{eq:mwo5}, and gradient bounds in \eqref{eq:gwo3} and \eqref{eq:gwo42}.

\subsection{Outline for Theorem~\ref{thm:erlangAK}}\label{app:AKoutline}

The proof of this theorem is nearly identical to the proof of Theorem~\ref{thm:erlangCK}. Therefore, we only outline the key steps and differences. The goal is to obtain a version of \eqref{eq:intermproofKC}, from which the theorem follows by applying Lemmas~\ref{lem:kolmfixA} and \ref{lem:densboundA}. To get a version of \eqref{eq:intermproofKC}, we bound each of the terms in \eqref{eq:third_bounds}, just like we did in the proof of Theorem~\ref{thm:erlangCK}. The proof varies between the underloaded and overloaded cases.

We begin with the underloaded case ($1 \leq R \leq n$). To bound the first term in \eqref{eq:third_bounds}, we use moment bounds \eqref{eq:mwu1}, \eqref{eq:mwu3}, and \eqref{eq:mwu4}, together with gradient bound \eqref{eq:ACder2}. For the second and third terms in \eqref{eq:third_bounds} we use the gradient bound in \eqref{eq:ACuder1}. For the fourth error term, we use gradient bound \eqref{eq:ACuder1}, and moment bounds \eqref{eq:mwuK1}, \eqref{eq:mwu3}, and
\begin{align*}
&\ \E \Big[ \big( b(\tilde X(\infty))\big)^2 1( \tilde X(\infty) \geq -\zeta) \Big] \\
=&\ \alpha^2\E \Big[ \big( \tilde X(\infty)+ \zeta\big)^2 1( \tilde X(\infty) \geq -\zeta) \Big]+ \mu^2 \zeta^2\Prob(\tilde X(\infty) \geq -\zeta) \\
& + 2\alpha \mu \abs{\zeta}\E \Big[ (\tilde X(\infty) + \zeta)  1( \tilde X(\infty) \geq -\zeta)\Big] \\
\leq&\ \alpha^2 \frac{1}{3}\Big(\frac{\mu }{\alpha}\delta^2  + \frac{\mu }{\alpha}4 + \delta^2\Big)+ \mu^2 \zeta^2\Prob(\tilde X(\infty) \geq -\zeta) \\
&+ 2\alpha \mu \Big( \frac{\delta^2}{4}\frac{\alpha}{\mu }+ \frac{\delta^2}{4} + 1 \Big),
\end{align*}
where the last inequality follows from moment bounds \eqref{eq:mwuK2} and \eqref{eq:mwu5}.

In the overloaded case ($n \leq R$), to bound the first term in \eqref{eq:third_bounds} we use moment bounds \eqref{eq:mwo7}, \eqref{eq:mwo1}, and \eqref{eq:mwo4} with gradient bound \eqref{eq:ACder2}. To bound the second and third terms in \eqref{eq:third_bounds} we use gradient bound \eqref{eq:ACoder1}. To bound the fourth term in \eqref{eq:third_bounds}, we use gradient bound \eqref{eq:ACder2}, with moment bounds \eqref{eq:mwo2} and
\begin{align*}
&\ \E \Big[ \big( b(\tilde X(\infty))\big)^2 1( \tilde X(\infty) \leq -\zeta) \Big] \\
=&\ \mu^2\E \Big[ \big( \tilde X(\infty)+ \zeta\big)^2 1( \tilde X(\infty) \leq -\zeta) \Big]+ \alpha^2 \zeta^2\Prob(\tilde X(\infty) \leq -\zeta) \\
& + 2\alpha \mu \zeta\E \Big[\big| (\tilde X(\infty) + \zeta)  1( \tilde X(\infty) \leq -\zeta)\big|\Big] \\
\leq&\ \mu^2 \Big(\frac{\delta^2}{4}\frac{\alpha}{\mu }+1\Big) + \alpha^2 \Big(\frac{\delta^2}{4}+\frac{\mu}{\alpha}\Big) + 2\alpha \mu \Big(\frac{\delta^2}{4}+1\Big),
\end{align*}
where the last inequality follows from moment bounds \eqref{eq:mwo8}, \eqref{eq:mwo3}, and \eqref{eq:mwoK1}.

\section{Miscellaneous Lemmas}  \label{app:misc}
This appendix proves Lemmas~\ref{LEM:GZ}, \ref{lem:kolmfixC},  \ref{lem:densboundC}, and \ref{lem:order_mag}. 

\subsection{Proof of Lemma~\ref{LEM:GZ}} \label{app:pflemgz}
\begin{proof}[Proof of Lemma~\ref{LEM:GZ}]
Let  $f(x): \R \to \R$ satisfy $\abs{f(x)} \leq C(1+x)^2$. A sufficient condition to ensure that
\begin{align*}
\E \big[ G_{\tilde X} f(\tilde X(\infty)) \big] = 0
\end{align*}
is given by \cite[Proposition 1.1]{Hend1997} (alternatively, see \cite[Proposition 3]{GlynZeev2008}). Namely, we require that 
\begin{align}
\E \Big[\big| G_{\tilde X} (\tilde X(\infty), \tilde X(\infty)) f(\tilde X(\infty))\big| \Big] < \infty, \label{eq:gzcond}
\end{align}
where $G_{\tilde X} (x,x)$ is the diagonal entry of the generator matrix $G_{\tilde X}$ corresponding to state $x$. 

We begin with the Erlang-C model, where the transition rates of the CTMC are bounded by $\lambda + n\mu$. Since $\abs{f(x)} \leq C(1+x)^2$, it suffices to show that $\E (\tilde X(\infty))^2 < \infty$, or that $\E (X(\infty))^2 < \infty$, where $X(\infty)$ has the stationary distribution of the CTMC $X$. Consider the function $V(k) = k^3$, where $k \in \Z_+$. Let $G_{X}$ be the generator of $X$, which is a simple birth death process with constant birth rate $\lambda$ and departure rate $\mu (k \wedge n)$ in state $k \in \Z_+$. Then for $k \geq n$, 
\begin{align}
G_{X}V(k) =&\ \lambda ( (k + 1)^3 - k^3) + n\mu ((k-1)^3 - k^3)  \notag \\
=&\ \lambda (3k^2 + 3k + 1) + n\mu  (-3k^2 + 3k - 1) \notag \\
=&\ -3k^2 ( n\mu - \lambda) + 3k(\lambda + n \mu ) + (\lambda - n\mu). \label{eq:gz1}
\end{align}
It is not hard to see that there exists some $k_0 \in \Z_+$, and a constant $c > 0$ (that depends on $\lambda, n$, and $\mu$), such that for all $k \geq k_0$,
\begin{align}
-3k^2 ( n\mu - \lambda) + 3k(\lambda + n \mu ) \leq -ck^2.\label{eq:gz3}
\end{align} 
We combine \eqref{eq:gz1}--\eqref{eq:gz3} to conclude that there exists some constant $d > 0$ (that depends on $\lambda, n$, and $\mu$) satisfying 
\begin{align*}
G_{X} V(x)  \leq -c x^2 + d 1(k < (k_0 \vee n)),
\end{align*}
and invoking \cite[Theorem 4.3]{MeynTwee1993b}, we see that $\E (X(\infty))^2 < \infty$.

The case of the Erlang-A model is not very different. When $\alpha > 0$, the transition rates of the CTMC depend linearly on its state. Hence, to satisfy \eqref{eq:gzcond} we need to show that $\E (X(\infty))^3 < \infty$. This is readily proven by repeating the procedure above with the Lyapunov function $V(k) = k^4$, and we omit the details.
\end{proof}

\subsection{Proof of Lemma~\ref{lem:kolmfixC}}
\label{app:kolmfixC}
\begin{proof}[Proof of Lemma~\ref{lem:kolmfixC}]
We let $F_W(w)$ and $F_{\tilde X}(x)$ be the distribution functions of $W$ and $\tilde X(\infty)$, respectively. For any $a \in \R$, let  $\tilde a = \delta(a - x(\infty))$. We want to show that
\begin{align}
\Prob( \tilde a - \delta <  \tilde X(\infty) \leq \tilde a + \delta) =&\ F_{\tilde X}(\tilde a + \delta) - F_{\tilde X}(\tilde a - \delta) \notag \\
\leq&\ 2\delta \omega(F_W)  + d_K(\tilde X(\infty), W) + 9\delta^2 + 8\delta^4. \label{eq:fix_result}
\end{align}
Define 
\begin{align*}
k^* = \inf \{k \geq 0 : \nu_k \geq \nu_j, \text{ for all $j \neq k$}\}.
\end{align*}
Then for any $\tilde a \in \R$, 
\begin{align*}
F_{\tilde X}(\tilde a + \delta) - F_{\tilde X}(\tilde a - \delta) \leq 2\nu_{k^*},
\end{align*}
because $\tilde X(\infty)$ takes at most two values in the interval $(\tilde a - \delta, \tilde a + \delta]$. Observe that by the flow balance equations, we know that for any $k \in \Z_+$, 
\begin{align*}
 \nu_k=  \frac{d(k+1)}{\lambda}\nu_{k+1}.
\end{align*}
Since $k^*$ is the maximizer of $\{\nu_k\}$, we know that 
\begin{align*}
d(k^*) \leq \lambda \leq d(k^*+1) \leq \lambda + \mu,
\end{align*}
where in the last inequality we have used the fact that the increase in departure rate between state $k^*$ and $k^*+1$ is at most $\mu$. Likewise, $d(k^*+i) \leq \lambda + i \mu$ for $i = 2,3$. Hence,
\begin{align*}
\nu_{k^*}=&\  \frac{d(k^*+1)}{\lambda}\nu_{k^*+1} \leq \Big(1 + \frac{\mu}{\lambda}\Big)\nu_{k^*+1} \leq  \nu_{k^*+1} + \delta^2,\\
\nu_{k^*}=&\  \frac{d(k^*+1)}{\lambda}\frac{d(k^*+2)}{\lambda} \nu_{k^*+2} \\
\leq&\ (1 + \delta^2)(1 + 2\delta^2)\nu_{k^*+2} \leq \nu_{k^*+2} + 3\delta^2  + 2\delta^4,\\
\nu_{k^*+1} =&\ \frac{d(k^*+2)}{\lambda}\frac{d(k^*+3)}{\lambda} \nu_{k^*+3} \\
\leq&\ (1 + 2\delta^2)(1 + 3\delta^2)\nu_{k^*+3} \leq \nu_{k^*+3} + 5\delta^2 + 6\delta^4,
\end{align*}
which implies that for any $\tilde a \in \R$, 
\begin{align*}
F_{\tilde X}(\tilde a + \delta) - F_{\tilde X}(\tilde a - \delta) \leq 2\nu_{k^*} \leq&\ \nu_{k^*} + \nu_{k^* + 1} + \delta^2\\
 =&\ F_{\tilde X}(\tilde k^* + \delta) - F_{\tilde X}(\tilde k^* - \delta) + \delta^2.
\end{align*}

There are now 4 cases to consider, with the first three being simple to handle. Recall that $\omega(F_W)$ is the modulus of continuity of $F_W(w)$. 
\begin{enumerate}
\item  \label{case:1} If $F_W(\tilde k^* - \delta) \leq F_{\tilde X}(\tilde k^* - \delta)$ and $F_W(\tilde k^* + \delta) \geq F_{\tilde X}(\tilde k^* + \delta)$, then 
\begin{align}
F_{\tilde X}(\tilde k^* + \delta) - F_{\tilde X}(\tilde k^* - \delta) \leq F_{W}(\tilde k^* + \delta) - F_{W}(\tilde k^* - \delta) \leq 2\delta \omega(F_W). \label{eq:case1}
\end{align}
\item  \label{case:2} If $F_W(\tilde k^* - \delta) \leq F_{\tilde X}(\tilde k^* - \delta)$ but $F_W(\tilde k^* + \delta) < F_{\tilde X}(\tilde k^* + \delta)$, then 
\begin{align}
&\ F_{\tilde X}(\tilde k^* + \delta) - F_{\tilde X}(\tilde k^* - \delta) \notag \\
 \leq&\  F_{\tilde X}(\tilde k^* + \delta) - F_{W}(\tilde k^* + \delta) +  F_{W}(\tilde k^* + \delta) - F_{W}(\tilde k^* - \delta) \notag \\
 \leq&\  2\delta \omega(F_W)  + d_K(\tilde X(\infty), W). \label{eq:case2}
\end{align}
\item \label{case:3} Similarly, if $F_W(\tilde k^* - \delta) > F_{\tilde X}(\tilde k^* - \delta)$ and $F_W(\tilde k^* + \delta) \geq F_{\tilde X}(\tilde k^* + \delta)$, then 
\begin{align}
&\ F_{\tilde X}(\tilde k^* + \delta) - F_{\tilde X}(\tilde k^* - \delta) \notag  \\
\leq&\  F_{W}(\tilde k^* + \delta) - F_{W}(\tilde k^* - \delta) + F_{W}(\tilde k^* - \delta) - F_{\tilde X}(\tilde k^* - \delta)  \notag \\
 \leq&\  2\delta \omega(F_W)  + d_K(\tilde X(\infty), W). \label{eq:case3}
\end{align}
\item \label{case:4}  Suppose $F_W(\tilde k^* - \delta) > F_{\tilde X}(\tilde k^* - \delta)$ and $F_W(\tilde k^* + \delta) < F_{\tilde X}(\tilde k^* + \delta)$, then we need to use a different approach. We know that
\begin{align*}
F_{\tilde X}(\tilde k^* + \delta) - F_{\tilde X}(\tilde k^* - \delta) =&\ \nu_{k^*} + \nu_{k^*+1} \\
\leq&\ \nu_{k^*+2} + \nu_{k^* + 3} + 8\delta^2 + 8\delta^4 \\
=&\ F_{\tilde X}(\tilde k^* + 3\delta) - F_{\tilde X}(\tilde k^* + \delta)+ 8\delta^2 + 8\delta^4.
\end{align*}
Since  $F_{W}(\tilde k^* + \delta) \leq F_{\tilde X}(\tilde k^* + \delta)$, we are either in  case~\ref{case:1}  or \ref{case:2} for $F_{\tilde X}(\tilde k^* + 3\delta) - F_{\tilde X}(\tilde k^* + \delta)$, and hence we have
\begin{align*}
F_{\tilde X}(\tilde k^* + 3\delta) - F_{\tilde X}(\tilde k^* + \delta) \leq 2\delta \omega(F_W)  + d_K(\tilde X(\infty), W).
\end{align*}
\end{enumerate}
This proves \eqref{eq:fix_result}, concluding the proof of this lemma.

\end{proof}

\subsection{Proof of Lemma~\ref{lem:densboundC} }
\label{app:densboundC}
\begin{proof}[Proof of Lemma~\ref{lem:densboundC}]
In the Erlang-C model,
\begin{align}
\nu(x) = 
\begin{cases}
a_{-} e^{-\frac{1}{2}x^2}, \quad x \leq - \zeta,\\
a_{+} e^{-\abs{\zeta} x}, \quad x \geq -\zeta.
\end{cases} \label{eq:densEC}
\end{align}
To bound this density, we need to bound $a_-$ and $a_+$. We know that $\nu(x)$ must integrate to one, which implies that 
\begin{align*}
a_- \int_{-\infty}^{-\zeta} e^{-\frac{1}{2}y^2} dy + a_+ \int_{-\zeta}^{\infty} e^{-\abs{\zeta} y} dy = 1
\end{align*}
Furthermore, since $\nu(x)$ is continuous at $x = -\zeta$, 
\begin{align*}
a_- e^{-\frac{1}{2}\zeta^2} = a_{+} e^{-\zeta^2}.
\end{align*}
Combining these two facts, we see that 
\begin{align}
 a_- = \frac{1}{\int_{-\infty}^{-\zeta} e^{-\frac{1}{2}y^2} dy + e^{\frac{1}{2}\zeta^2} \int_{-\zeta}^{\infty} e^{-\abs{\zeta} y} dy} \leq \frac{1}{\int_{-\infty}^{0} e^{-\frac{1}{2}y^2} dy} = \sqrt{\frac{2}{\pi}}, \label{eq:aminus}
\end{align}
and 
\begin{align}
a_+ = \frac{1}{e^{-\frac{1}{2}\zeta^2}\int_{-\infty}^{-\zeta} e^{-\frac{1}{2}y^2} dy +  \int_{-\zeta}^{\infty} e^{-\abs{\zeta} y} dy} \leq \frac{1}{e^{-\frac{1}{2}\zeta^2}\int_{-\infty}^{0} e^{-\frac{1}{2}y^2} dy} = e^{\frac{1}{2}\zeta^2}\sqrt{\frac{2}{\pi}}. \label{eq:aplus}
\end{align}
Therefore, for $x \leq -\zeta$, 
\begin{align*}
\abs{\nu(x)} \leq a_- \leq  \sqrt{\frac{2}{\pi}},
\end{align*}
and for $x \geq -\zeta$, we recall that $\zeta < 0$ to see that
\begin{align*}
\abs{\nu(x)} \leq a_+ e^{-\abs{\zeta} x} \leq \sqrt{\frac{2}{\pi}}e^{\frac{1}{2}\zeta^2}e^{-\abs{\zeta} x} \leq \sqrt{\frac{2}{\pi}}.
\end{align*}
\end{proof}

\subsection{Proof of Lemma~\ref{lem:order_mag} } \label{app:order_mag}
\begin{proof}[Proof of Lemma~\ref{lem:order_mag} ]
\blue{The density of $Y(\infty)$ is given in \eqref{eq:densEC}, and so
\begin{align*}
\E (Y(\infty))^m = a_- \int_{-\infty}^{-\zeta} y^{m}e^{-\frac{1}{2}y^2} dy + a_+ \int_{-\zeta}^{\infty} y^{m}e^{-\abs{\zeta} y} dy,
\end{align*}
where $a_-$ and $a_+$ are as in \eqref{eq:aminus} and \eqref{eq:aplus}. In particular, 
\begin{align*}
 a_- = \frac{1}{\int_{-\infty}^{-\zeta} e^{-\frac{1}{2}y^2} dy + e^{\frac{1}{2}\zeta^2} \int_{-\zeta}^{\infty} e^{-\abs{\zeta} y} dy}  = \frac{1}{\int_{-\infty}^{-\zeta} e^{-\frac{1}{2}y^2} dy + \frac{1}{\abs{\zeta}} e^{-\frac{1}{2}\zeta^2}},
\end{align*}
which implies that 
\begin{align*}
\lim_{\zeta \uparrow 0} \abs{\zeta}^{m} a_- \int_{-\infty}^{-\zeta} y^{m}e^{-\frac{1}{2}y^2} dy = 0.
\end{align*}
Furthermore, 
\begin{align*}
a_+ = \frac{1}{e^{-\frac{1}{2}\zeta^2}\int_{-\infty}^{-\zeta} e^{-\frac{1}{2}y^2} dy +  \int_{-\zeta}^{\infty} e^{-\abs{\zeta} y} dy} = \frac{1}{e^{-\frac{1}{2}\zeta^2}\int_{-\infty}^{-\zeta} e^{-\frac{1}{2}y^2} dy + \frac{1}{\abs{\zeta}} e^{-\zeta^2}},
\end{align*}
and using integration by parts,
\begin{align*}
\int_{-\zeta}^{\infty} y^{m}e^{-\abs{\zeta} y} dy =&\ e^{-\zeta^2}\sum_{j=0}^{m} \frac{m!}{(m-j)!} \frac{1}{\abs{\zeta}^{j+1}}\abs{\zeta}^{m-j} \\
=&\ e^{-\zeta^2}\sum_{j=0}^{m-1} \frac{m!}{(m-j)!} \frac{1}{\abs{\zeta}^{j+1}}\abs{\zeta}^{m-j} + 
e^{-\zeta^2} \frac{m!}{\abs{\zeta}^{m+1}}.
\end{align*}
Hence, 
\begin{align*}
\lim_{\zeta \uparrow 0} \abs{\zeta}^{m} a_+ \int_{-\zeta}^{\infty} y^{m}e^{-\abs{\zeta} y} dy = 
m!.
\end{align*}}

\end{proof}

\bibliography{dai20151228}
\end{document}